\documentclass[a4paper,reqno,10pt]{amsart}

\usepackage{a4wide}
\usepackage{amssymb,amsmath,amsthm}

\usepackage{txfonts}
\usepackage{mathrsfs}
\usepackage{enumerate}
\usepackage{graphicx}

\usepackage{bm}
\usepackage{multirow}
\usepackage{subfig}
\usepackage{placeins}

\usepackage[pdfpagelabels,hypertexnames=true,
            plainpages=false,
            naturalnames=false]{hyperref}
\usepackage{color}
\definecolor{darkblue}{rgb}{0,0.1,0.5}
\hypersetup{colorlinks,
            linkcolor=darkblue,
            anchorcolor=darkblue,
            citecolor=darkblue}

\usepackage{enumerate}
\usepackage{bbm, hyphenat}

\usepackage[backend=bibtex,natbib=true,style=numeric,doi=false,isbn=false,url=false]{biblatex}
\usepackage{csquotes}

\def\blx@bblfile@bibtex{
  \blx@secinit
  \begingroup
  \blx@bblstart
 
\begingroup
\makeatletter
\@ifundefined{ver@biblatex.sty}
  {\@latex@error
     {Missing 'biblatex' package}
     {The bibliography requires the 'biblatex' package.}
      \aftergroup }
  {}
\endgroup

\entry{apostol1974mathematical}{book}{}
  \name{author}{1}{}{%
    {{}%
     {Apostol}{A.}%
     {T.~M.}{T.~M.}%
     {}{}%
     {}{}}%
  }
  \list{publisher}{1}{%
    {Addison-Wesley Publishing Co., Reading--London--Don Mills}%
  }
  \strng{namehash}{ATM1}
  \strng{fullhash}{ATM1}
  \field{labelyear}{1974}
  \field{sortinit}{A}
  \field{edition}{Second edition}
  \field{title}{Mathematical analysis}
  \field{year}{1974}
\endentry

\entry{applebaum2004lpa}{book}{}
  \name{author}{1}{}{%
    {{}%
     {Applebaum}{A.}%
     {D.}{D.}%
     {}{}%
     {}{}}%
  }
  \list{publisher}{1}{%
    {Cambridge University Press}%
  }
  \strng{namehash}{AD1}
  \strng{fullhash}{AD1}
  \field{labelyear}{2004}
  \field{sortinit}{A}
  \field{isbn}{0-521-83263-2}
  \field{series}{Cambridge Studies in Advanced Mathematics}
  \field{title}{L\'evy processes and stochastic calculus}
  \field{volume}{93}
  \list{location}{1}{%
    {Cambridge}%
  }
  \field{year}{2004}
\endentry

\entry{barndorff2011nongaussian}{article}{}
  \name{author}{2}{}{%
    {{}%
     {Barndorff-Nielsen}{B.-N.}%
     {O.~E.}{O.~E.}%
     {}{}%
     {}{}}%
    {{}%
     {Shephard}{S.}%
     {N.}{N.}%
     {}{}%
     {}{}}%
  }
  \strng{namehash}{BNOESN1}
  \strng{fullhash}{BNOESN1}
  \field{labelyear}{2001}
  \field{sortinit}{B}
  \verb{doi}
  \verb 10.1111/1467-9868.00282
  \endverb
  \field{issn}{1369-7412}
  \field{number}{2}
  \field{pages}{167\bibrangedash 241}
  \field{title}{Non\hyp{}{G}aussian {O}rnstein--{U}hlenbeck-based models and
  some of their uses in financial economics}
  \verb{url}
  \verb http://dx.doi.org/10.1111/1467-9868.00282
  \endverb
  \field{volume}{63}
  \field{journaltitle}{J. R. Stat. Soc. Ser. B}
  \field{year}{2001}
\endentry

\entry{barndorff2011multivariate}{article}{}
  \name{author}{2}{}{%
    {{}%
     {Barndorff-Nielsen}{B.-N.}%
     {O.~E.}{O.~E.}%
     {}{}%
     {}{}}%
    {{}%
     {Stelzer}{S.}%
     {R.}{R.}%
     {}{}%
     {}{}}%
  }
  \strng{namehash}{BNOESR1}
  \strng{fullhash}{BNOESR1}
  \field{labelyear}{2011}
  \field{sortinit}{B}
  \field{issn}{1050-5164}
  \field{number}{1}
  \field{pages}{140\bibrangedash 182}
  \field{title}{Multivariate sup{OU} Processes}
  \field{volume}{21}
  \field{journaltitle}{Ann. Appl. Probab.}
  \field{year}{2011}
\endentry

\entry{bauer2002wahrscheinlich}{book}{}
  \name{author}{1}{}{%
    {{}%
     {Bauer}{B.}%
     {H.}{H.}%
     {}{}%
     {}{}}%
  }
  \list{publisher}{1}{%
    {Walter de Gruyter \& Co.}%
  }
  \strng{namehash}{BH1}
  \strng{fullhash}{BH1}
  \field{labelyear}{2002}
  \field{sortinit}{B}
  \field{edition}{Fifth edition}
  \field{isbn}{3-11-017236-4}
  \field{series}{de Gruyter Lehrbuch}
  \field{title}{Wahrscheinlichkeitstheorie}
  \list{location}{1}{%
    {Berlin}%
  }
  \field{year}{2002}
\endentry

\entry{benth2009dynamic}{article}{}
  \name{author}{2}{}{%
    {{}%
     {Benth}{B.}%
     {F.~E.}{F.~E.}%
     {}{}%
     {}{}}%
    {{}%
     {{\v{S}}altyt{\.e}~Benth}{{\v{S}}.~B.}%
     {J.}{J.}%
     {}{}%
     {}{}}%
  }
  \list{publisher}{1}{%
    {Elsevier}%
  }
  \strng{namehash}{BFESBJ1}
  \strng{fullhash}{BFESBJ1}
  \field{labelyear}{2009}
  \field{sortinit}{B}
  \field{number}{1}
  \field{pages}{16\bibrangedash 24}
  \field{title}{Dynamic pricing of wind futures}
  \field{volume}{31}
  \field{journaltitle}{Energy Economics}
  \field{year}{2009}
\endentry

\entry{brockwell2005lda}{article}{}
  \name{author}{2}{}{%
    {{}%
     {Brockwell}{B.}%
     {P.}{P.}%
     {}{}%
     {}{}}%
    {{}%
     {Marquardt}{M.}%
     {T.~G.}{T.~G.}%
     {}{}%
     {}{}}%
  }
  \strng{namehash}{BPMTG1}
  \strng{fullhash}{BPMTG1}
  \field{labelyear}{2005}
  \field{sortinit}{B}
  \field{issn}{1017-0405}
  \field{number}{2}
  \field{pages}{477\bibrangedash 494}
  \field{title}{L\'evy\hyp{}driven and fractionally integrated {ARMA} processes
  with continuous time parameter}
  \field{volume}{15}
  \field{journaltitle}{Stat. Sinica}
  \field{year}{2005}
\endentry

\entry{brockwell2001levy}{article}{}
  \name{author}{1}{}{%
    {{}%
     {Brockwell}{B.}%
     {P.~J.}{P.~J.}%
     {}{}%
     {}{}}%
  }
  \strng{namehash}{BPJ1}
  \strng{fullhash}{BPJ1}
  \field{labelyear}{2001}
  \field{sortinit}{B}
  \field{issn}{0020-3157}
  \field{number}{1}
  \field{pages}{113\bibrangedash 124}
  \field{title}{L\'evy\hyp{}driven {CARMA} processes}
  \field{volume}{53}
  \field{journaltitle}{Ann. Inst. Stat. Math.}
  \field{year}{2001}
\endentry

\entry{brockwell2004rct}{article}{}
  \name{author}{1}{}{%
    {{}%
     {Brockwell}{B.}%
     {P.~J.}{P.~J.}%
     {}{}%
     {}{}}%
  }
  \strng{namehash}{BPJ1}
  \strng{fullhash}{BPJ1}
  \field{labelyear}{2004}
  \field{sortinit}{B}
  \verb{doi}
  \verb 10.1239/jap/1082552212
  \endverb
  \field{issn}{0021-9002}
  \field{note}{Stochastic methods and their applications}
  \field{pages}{375\bibrangedash 382}
  \field{title}{Representations of continuous\hyp{}time {ARMA} processes}
  \verb{url}
  \verb http://dx.doi.org/10.1239/jap/1082552212
  \endverb
  \field{volume}{41A}
  \field{journaltitle}{J. Appl. Probab.}
  \field{year}{2004}
\endentry

\entry{brockwell1991tst}{book}{}
  \name{author}{2}{}{%
    {{}%
     {Brockwell}{B.}%
     {P.~J.}{P.~J.}%
     {}{}%
     {}{}}%
    {{}%
     {Davis}{D.}%
     {R.~A.}{R.~A.}%
     {}{}%
     {}{}}%
  }
  \list{publisher}{1}{%
    {Springer\hyp{}Verlag}%
  }
  \strng{namehash}{BPJDRA1}
  \strng{fullhash}{BPJDRA1}
  \field{labelyear}{1991}
  \field{sortinit}{B}
  \field{edition}{Second edition}
  \field{isbn}{0-387-97429-6}
  \field{series}{Springer Series in Statistics}
  \field{title}{Time series: theory and methods}
  \list{location}{1}{%
    {New York}%
  }
  \field{year}{1991}
\endentry

\entry{brockwell2010estimation}{article}{}
  \name{author}{3}{}{%
    {{}%
     {Brockwell}{B.}%
     {P.~J.}{P.~J.}%
     {}{}%
     {}{}}%
    {{}%
     {Davis}{D.}%
     {R.~A.}{R.~A.}%
     {}{}%
     {}{}}%
    {{}%
     {Yang}{Y.}%
     {Y.}{Y.}%
     {}{}%
     {}{}}%
  }
  \strng{namehash}{BPJDRAYY1}
  \strng{fullhash}{BPJDRAYY1}
  \field{labelyear}{2011}
  \field{sortinit}{B}
  \field{number}{2}
  \field{pages}{250\bibrangedash 259}
  \field{title}{Estimation for nonnegative {L\'e}vy\hyp{}driven {CARMA}
  processes}
  \field{volume}{29}
  \field{journaltitle}{J. Bus. Econ. Stat.}
  \field{year}{2011}
\endentry

\entry{brockwell2009existence}{article}{}
  \name{author}{2}{}{%
    {{}%
     {Brockwell}{B.}%
     {P.~J.}{P.~J.}%
     {}{}%
     {}{}}%
    {{}%
     {Lindner}{L.}%
     {A.}{A.}%
     {}{}%
     {}{}}%
  }
  \strng{namehash}{BPJLA1}
  \strng{fullhash}{BPJLA1}
  \field{labelyear}{2009}
  \field{sortinit}{B}
  \field{issn}{0304-4149}
  \field{number}{8}
  \field{pages}{2660\bibrangedash 2681}
  \field{title}{Existence and uniqueness of stationary {L}\'evy\hyp{}driven
  {CARMA} processes}
  \field{volume}{119}
  \field{journaltitle}{Stoch. Process. Their Appl.}
  \field{year}{2009}
\endentry

\entry{Deuflhard2008}{book}{}
  \name{author}{2}{}{%
    {{}%
     {Deuflhard}{D.}%
     {P.}{P.}%
     {}{}%
     {}{}}%
    {{}%
     {Hohmann}{H.}%
     {A.}{A.}%
     {}{}%
     {}{}}%
  }
  \list{publisher}{1}{%
    {Walter de Gruyter \& Co.}%
  }
  \strng{namehash}{DPHA1}
  \strng{fullhash}{DPHA1}
  \field{labelyear}{2008}
  \field{sortinit}{D}
  \field{edition}{Fourth edition}
  \field{isbn}{978-3-11-020354-7}
  \field{note}{Eine algorithmisch orientierte Einf{\"u}hrung}
  \field{series}{de Gruyter Lehrbuch}
  \field{title}{Numerische {M}athematik. 1}
  \list{location}{1}{%
    {Berlin}%
  }
  \field{year}{2008}
\endentry

\entry{doob1944egp}{article}{}
  \name{author}{1}{}{%
    {{}%
     {Doob}{D.}%
     {J.~L.}{J.~L.}%
     {}{}%
     {}{}}%
  }
  \strng{namehash}{DJL1}
  \strng{fullhash}{DJL1}
  \field{labelyear}{1944}
  \field{sortinit}{D}
  \field{issn}{0003-4851}
  \field{number}{3}
  \field{pages}{229\bibrangedash 282}
  \field{title}{The elementary {G}aussian processes}
  \field{volume}{15}
  \field{journaltitle}{Ann. Math. Statistics}
  \field{year}{1944}
\endentry

\entry{figueroa2009nonparametric}{incollection}{}
  \name{author}{1}{}{%
    {{}%
     {Figueroa-L{\'o}pez}{F.-L.}%
     {J.~E.}{J.~E.}%
     {}{}%
     {}{}}%
  }
  \list{publisher}{1}{%
    {Inst. Math. Stat.}%
  }
  \strng{namehash}{FLJE1}
  \strng{fullhash}{FLJE1}
  \field{labelyear}{2009}
  \field{sortinit}{F}
  \field{booktitle}{Optimality}
  \verb{doi}
  \verb 10.1214/09-LNMS5709
  \endverb
  \field{pages}{117\bibrangedash 146}
  \field{series}{IMS Lecture Notes Monogr. Ser.}
  \field{title}{Nonparametric estimation of {L}\'evy models based on
  discrete-sampling}
  \verb{url}
  \verb http://dx.doi.org/10.1214/09-LNMS5709
  \endverb
  \field{volume}{57}
  \list{location}{1}{%
    {Beachwood}%
  }
  \field{year}{2009}
\endentry

\entry{gugushvili2009nonparametric}{article}{}
  \name{author}{1}{}{%
    {{}%
     {Gugushvili}{G.}%
     {S.}{S.}%
     {}{}%
     {}{}}%
  }
  \strng{namehash}{GS1}
  \strng{fullhash}{GS1}
  \field{labelyear}{2009}
  \field{sortinit}{G}
  \verb{doi}
  \verb 10.1080/10485250802645824
  \endverb
  \field{issn}{1048-5252}
  \field{number}{3}
  \field{pages}{321\bibrangedash 343}
  \field{title}{Nonparametric estimation of the characteristic triplet of a
  discretely observed {L}\'evy process}
  \verb{url}
  \verb http://dx.doi.org/10.1080/10485250802645824
  \endverb
  \field{volume}{21}
  \field{journaltitle}{J. Nonparametr. Stat.}
  \field{year}{2009}
\endentry

\entry{hall2005generalized}{book}{}
  \name{author}{1}{}{%
    {{}%
     {Hall}{H.}%
     {A.~R.}{A.~R.}%
     {}{}%
     {}{}}%
  }
  \list{publisher}{1}{%
    {Oxford University Press}%
  }
  \strng{namehash}{HAR1}
  \strng{fullhash}{HAR1}
  \field{labelyear}{2005}
  \field{sortinit}{H}
  \field{isbn}{0-19-877520-2}
  \field{series}{Advanced Texts in Econometrics}
  \field{title}{Generalized method of moments}
  \list{location}{1}{%
    {Oxford}%
  }
  \field{year}{2005}
\endentry

\entry{hansen1982}{article}{}
  \name{author}{1}{}{%
    {{}%
     {Hansen}{H.}%
     {L.~P.}{L.~P.}%
     {}{}%
     {}{}}%
  }
  \strng{namehash}{HLP1}
  \strng{fullhash}{HLP1}
  \field{labelyear}{1982}
  \field{sortinit}{H}
  \verb{doi}
  \verb 10.2307/1912775
  \endverb
  \field{issn}{0012-9682}
  \field{number}{4}
  \field{pages}{1029\bibrangedash 1054}
  \field{title}{Large sample properties of generalized method of moments
  estimators}
  \verb{url}
  \verb http://dx.doi.org/10.2307/1912775
  \endverb
  \field{volume}{50}
  \field{journaltitle}{Econometrica}
  \field{year}{1982}
\endentry

\entry{jacod2003probability}{book}{}
  \name{author}{2}{}{%
    {{}%
     {Jacod}{J.}%
     {J.}{J.}%
     {}{}%
     {}{}}%
    {{}%
     {Protter}{P.}%
     {P.}{P.}%
     {}{}%
     {}{}}%
  }
  \list{publisher}{1}{%
    {Springer\hyp{}Verlag}%
  }
  \strng{namehash}{JJPP1}
  \strng{fullhash}{JJPP1}
  \field{labelyear}{2003}
  \field{sortinit}{J}
  \field{edition}{Second edition}
  \field{isbn}{3-540-43871-8}
  \field{series}{Universitext}
  \field{title}{Probability essentials}
  \list{location}{1}{%
    {Berlin}%
  }
  \field{year}{2003}
\endentry

\entry{jacod2003limit}{book}{}
  \name{author}{2}{}{%
    {{}%
     {Jacod}{J.}%
     {J.}{J.}%
     {}{}%
     {}{}}%
    {{}%
     {Shiryaev}{S.}%
     {A.~N.}{A.~N.}%
     {}{}%
     {}{}}%
  }
  \list{publisher}{1}{%
    {Springer\hyp{}Verlag}%
  }
  \strng{namehash}{JJSAN1}
  \strng{fullhash}{JJSAN1}
  \field{labelyear}{2003}
  \field{sortinit}{J}
  \field{edition}{Second edition}
  \field{isbn}{3-540-43932-3}
  \field{series}{Grundlehren der Mathematischen Wissenschaften}
  \field{title}{Limit theorems for stochastic processes}
  \field{volume}{288}
  \list{location}{1}{%
    {Berlin}%
  }
  \field{year}{2003}
\endentry

\entry{kailath1980linear}{book}{}
  \name{author}{1}{}{%
    {{}%
     {Kailath}{K.}%
     {T.}{T.}%
     {}{}%
     {}{}}%
  }
  \list{publisher}{1}{%
    {Prentice-Hall Inc.}%
  }
  \strng{namehash}{KT1}
  \strng{fullhash}{KT1}
  \field{labelyear}{1980}
  \field{sortinit}{K}
  \field{isbn}{0-13-536961-4}
  \field{series}{Prentice-Hall Information and System Sciences Series}
  \field{title}{Linear systems}
  \list{location}{1}{%
    {Englewood Cliffs}%
  }
  \field{year}{1980}
\endentry

\entry{kailath1978fubini}{article}{}
  \name{author}{3}{}{%
    {{}%
     {Kailath}{K.}%
     {T.}{T.}%
     {}{}%
     {}{}}%
    {{}%
     {Segall}{S.}%
     {A.}{A.}%
     {}{}%
     {}{}}%
    {{}%
     {Zakai}{Z.}%
     {M.}{M.}%
     {}{}%
     {}{}}%
  }
  \strng{namehash}{KTSAZM1}
  \strng{fullhash}{KTSAZM1}
  \field{labelyear}{1978}
  \field{sortinit}{K}
  \field{issn}{0581-572X}
  \field{number}{2}
  \field{pages}{138\bibrangedash 143}
  \field{title}{Fubini\hyp{}type theorems for stochastic integrals}
  \field{volume}{40}
  \field{journaltitle}{Sankhy\=a Ser. A}
  \field{year}{1978}
\endentry

\entry{klenke2008probability}{book}{}
  \name{author}{1}{}{%
    {{}%
     {Klenke}{K.}%
     {A.}{A.}%
     {}{}%
     {}{}}%
  }
  \list{publisher}{1}{%
    {Springer\hyp{}Verlag London Ltd.}%
  }
  \strng{namehash}{KA1}
  \strng{fullhash}{KA1}
  \field{labelyear}{2008}
  \field{sortinit}{K}
  \field{series}{Universitext}
  \field{title}{Probability theory}
  \list{location}{1}{%
    {London}%
  }
  \field{year}{2008}
\endentry

\entry{krengel1985ergodic}{book}{}
  \name{author}{1}{}{%
    {{}%
     {Krengel}{K.}%
     {U.}{U.}%
     {}{}%
     {}{}}%
  }
  \list{publisher}{1}{%
    {Walter de Gruyter \& Co.}%
  }
  \strng{namehash}{KU1}
  \strng{fullhash}{KU1}
  \field{labelyear}{1985}
  \field{sortinit}{K}
  \field{isbn}{3-11-008478-3}
  \field{note}{With a supplement by Antoine Brunel}
  \field{series}{de Gruyter Studies in Mathematics}
  \field{title}{Ergodic theorems}
  \field{volume}{6}
  \list{location}{1}{%
    {Berlin}%
  }
  \field{year}{1985}
\endentry

\entry{leveque2007}{book}{}
  \name{author}{1}{}{%
    {{}%
     {LeVeque}{L.}%
     {R.~J.}{R.~J.}%
     {}{}%
     {}{}}%
  }
  \list{publisher}{2}{%
    {Society for Industrial}%
    {Applied Mathematics}%
  }
  \strng{namehash}{LRJ1}
  \strng{fullhash}{LRJ1}
  \field{labelyear}{2007}
  \field{sortinit}{L}
  \field{isbn}{978-0-898716-29-0}
  \field{note}{Steady\hyp{}state and time\hyp{}dependent problems}
  \field{title}{Finite difference methods for ordinary and partial differential
  equations}
  \list{location}{1}{%
    {Philadelphia}%
  }
  \field{year}{2007}
\endentry

\entry{marquardt2007mficarma}{article}{}
  \name{author}{1}{}{%
    {{}%
     {Marquardt}{M.}%
     {T.}{T.}%
     {}{}%
     {}{}}%
  }
  \strng{namehash}{MT1}
  \strng{fullhash}{MT1}
  \field{labelyear}{2007}
  \field{sortinit}{M}
  \field{issn}{0047-259X}
  \field{number}{9}
  \field{pages}{1705\bibrangedash 1725}
  \field{title}{Multivariate fractionally integrated {CARMA} processes}
  \field{volume}{98}
  \field{journaltitle}{J. Multivar. Anal.}
  \field{year}{2007}
\endentry

\entry{marquardt2007multivariate}{article}{}
  \name{author}{2}{}{%
    {{}%
     {Marquardt}{M.}%
     {T.}{T.}%
     {}{}%
     {}{}}%
    {{}%
     {Stelzer}{S.}%
     {R.}{R.}%
     {}{}%
     {}{}}%
  }
  \strng{namehash}{MTSR1}
  \strng{fullhash}{MTSR1}
  \field{labelyear}{2007}
  \field{sortinit}{M}
  \field{issn}{0304-4149}
  \field{number}{1}
  \field{pages}{96\bibrangedash 120}
  \field{title}{Multivariate {CARMA} processes}
  \field{volume}{117}
  \field{journaltitle}{Stoch. Process. Their Appl.}
  \field{year}{2007}
\endentry

\entry{masuda2004onmultidimensionalOU}{article}{}
  \name{author}{1}{}{%
    {{}%
     {Masuda}{M.}%
     {H.}{H.}%
     {}{}%
     {}{}}%
  }
  \strng{namehash}{MH1}
  \strng{fullhash}{MH1}
  \field{labelyear}{2004}
  \field{sortinit}{M}
  \verb{doi}
  \verb 10.3150/bj/1077544605
  \endverb
  \field{issn}{1350-7265}
  \field{number}{1}
  \field{pages}{97\bibrangedash 120}
  \field{title}{On multidimensional {O}rnstein--{U}hlenbeck processes driven by
  a general {L}\'evy process}
  \verb{url}
  \verb http://dx.doi.org/10.3150/bj/1077544605
  \endverb
  \field{volume}{10}
  \field{journaltitle}{Bernoulli}
  \field{year}{2004}
\endentry

\entry{mokkadem1988mixing}{article}{}
  \name{author}{1}{}{%
    {{}%
     {Mokkadem}{M.}%
     {A.}{A.}%
     {}{}%
     {}{}}%
  }
  \strng{namehash}{MA1}
  \strng{fullhash}{MA1}
  \field{labelyear}{1988}
  \field{sortinit}{M}
  \field{issn}{0304-4149}
  \field{number}{2}
  \field{pages}{309\bibrangedash 315}
  \field{title}{Mixing properties of {ARMA} processes}
  \field{volume}{29}
  \field{journaltitle}{Stoch. Process. Their Appl.}
  \field{year}{1988}
\endentry

\entry{newey1994largesample}{incollection}{}
  \name{author}{2}{}{%
    {{}%
     {Newey}{N.}%
     {W.~K.}{W.~K.}%
     {}{}%
     {}{}}%
    {{}%
     {McFadden}{M.}%
     {D.}{D.}%
     {}{}%
     {}{}}%
  }
  \list{publisher}{1}{%
    {North\hyp{}Holland}%
  }
  \strng{namehash}{NWKMD1}
  \strng{fullhash}{NWKMD1}
  \field{labelyear}{1994}
  \field{sortinit}{N}
  \field{booktitle}{Handbook of econometrics, {V}ol.\ {IV}}
  \field{pages}{2111\bibrangedash 2245}
  \field{series}{Handbooks in Econom.}
  \field{title}{Large sample estimation and hypothesis testing}
  \field{volume}{2}
  \list{location}{1}{%
    {Amsterdam}%
  }
  \field{year}{1994}
\endentry

\entry{ngo2010parametric}{article}{}
  \name{author}{1}{}{%
    {{}%
     {Ngo}{N.}%
     {H.-L.}{H.-L.}%
     {}{}%
     {}{}}%
  }
  \strng{namehash}{NHL1}
  \strng{fullhash}{NHL1}
  \field{labelyear}{2010}
  \field{sortinit}{N}
  \field{pages}{1443\bibrangedash 1469}
  \field{title}{Parametric estimation for discretely observed stochastic
  processes with jumps}
  \field{volume}{4}
  \field{journaltitle}{Electron. J. Stat.}
  \field{year}{2010}
\endentry

\entry{ogihara2011quasi}{article}{}
  \name{author}{2}{}{%
    {{}%
     {Ogihara}{O.}%
     {T.}{T.}%
     {}{}%
     {}{}}%
    {{}%
     {Yoshida}{Y.}%
     {N.}{N.}%
     {}{}%
     {}{}}%
  }
  \strng{namehash}{OTYN1}
  \strng{fullhash}{OTYN1}
  \field{labelyear}{2011}
  \field{sortinit}{O}
  \field{number}{3}
  \field{pages}{189\bibrangedash 229}
  \field{title}{Quasi-likelihood analysis for the stochastic differential
  equation with jumps}
  \field{volume}{14}
  \field{journaltitle}{Stat. Inference Stoch. Process.}
  \field{year}{2011}
\endentry

\entry{protter1990sia}{book}{}
  \name{author}{1}{}{%
    {{}%
     {Protter}{P.}%
     {P.}{P.}%
     {}{}%
     {}{}}%
  }
  \list{publisher}{1}{%
    {Springer\hyp{}Verlag}%
  }
  \strng{namehash}{PP1}
  \strng{fullhash}{PP1}
  \field{labelyear}{1990}
  \field{sortinit}{P}
  \field{isbn}{3-540-50996-8}
  \field{note}{A new approach}
  \field{series}{Applications of Mathematics}
  \field{title}{Stochastic integration and differential equations}
  \field{volume}{21}
  \list{location}{1}{%
    {Berlin}%
  }
  \field{year}{1990}
\endentry

\entry{rajput1989spectral}{article}{}
  \name{author}{2}{}{%
    {{}%
     {Rajput}{R.}%
     {B.~S.}{B.~S.}%
     {}{}%
     {}{}}%
    {{}%
     {Rosi{\'n}ski}{R.}%
     {J.}{J.}%
     {}{}%
     {}{}}%
  }
  \strng{namehash}{RBSRJ1}
  \strng{fullhash}{RBSRJ1}
  \field{labelyear}{1989}
  \field{sortinit}{R}
  \field{issn}{0178-8051}
  \field{number}{3}
  \field{pages}{451\bibrangedash 487}
  \field{title}{Spectral representations of infinitely divisible processes}
  \field{volume}{82}
  \field{journaltitle}{Probab. Theory Relat. Field}
  \field{year}{1989}
\endentry

\entry{rosenblatt1956central}{article}{}
  \name{author}{1}{}{%
    {{}%
     {Rosenblatt}{R.}%
     {M.}{M.}%
     {}{}%
     {}{}}%
  }
  \strng{namehash}{RM1}
  \strng{fullhash}{RM1}
  \field{labelyear}{1956}
  \field{sortinit}{R}
  \field{issn}{0027-8424}
  \field{number}{1}
  \field{pages}{43\bibrangedash 47}
  \field{title}{A central limit theorem and a strong mixing condition}
  \field{volume}{42}
  \field{journaltitle}{Proc. Nat. Acad. Sci. U. S. A.}
  \field{year}{1956}
\endentry

\entry{sato1991lpa}{book}{}
  \name{author}{1}{}{%
    {{}%
     {Sato}{S.}%
     {K.}{K.}%
     {}{}%
     {}{}}%
  }
  \list{publisher}{1}{%
    {Cambridge University Press}%
  }
  \strng{namehash}{SK1}
  \strng{fullhash}{SK1}
  \field{labelyear}{1999}
  \field{sortinit}{S}
  \field{isbn}{0-521-55302-4}
  \field{series}{Cambridge Studies in Advanced Mathematics}
  \field{title}{L\'evy processes and infinitely divisible distributions}
  \field{volume}{68}
  \list{location}{1}{%
    {Cambridge}%
  }
  \field{year}{1999}
\endentry

\entry{schlemm2011quasi}{misc}{}
  \name{author}{2}{}{%
    {{}%
     {Schlemm}{S.}%
     {E.}{E.}%
     {}{}%
     {}{}}%
    {{}%
     {Stelzer}{S.}%
     {R.}{R.}%
     {}{}%
     {}{}}%
  }
  \strng{namehash}{SESR1}
  \strng{fullhash}{SESR1}
  \field{labelyear}{2011}
  \field{sortinit}{S}
  \field{note}{Preprint: Available at \url{http://www-m4.ma.tum.de}}
  \field{title}{Quasi maximum likelihood estimation for strongly mixing linear
  state space models and multivariate {CARMA} processes}
  \field{year}{2011}
\endentry

\entry{schlemmmixing2010}{article}{}
  \name{author}{2}{}{%
    {{}%
     {Schlemm}{S.}%
     {E.}{E.}%
     {}{}%
     {}{}}%
    {{}%
     {Stelzer}{S.}%
     {R.}{R.}%
     {}{}%
     {}{}}%
  }
  \strng{namehash}{SESR1}
  \strng{fullhash}{SESR1}
  \field{labelyear}{2012}
  \field{sortinit}{S}
  \field{number}{1}
  \field{pages}{46\bibrangedash 63}
  \field{title}{Multivariate {CARMA} Processes, Continuous\hyp{}Time State
  Space Models and Complete Regularity of the Innovations of the Sampled
  Processes}
  \field{volume}{18}
  \field{journaltitle}{Bernoulli}
  \field{year}{2012}
\endentry

\entry{shimizu2006estimation}{article}{}
  \name{author}{2}{}{%
    {{}%
     {Shimizu}{S.}%
     {Y.}{Y.}%
     {}{}%
     {}{}}%
    {{}%
     {Yoshida}{Y.}%
     {N.}{N.}%
     {}{}%
     {}{}}%
  }
  \strng{namehash}{SYYN1}
  \strng{fullhash}{SYYN1}
  \field{labelyear}{2006}
  \field{sortinit}{S}
  \field{number}{3}
  \field{pages}{227\bibrangedash 277}
  \field{title}{Estimation of parameters for diffusion processes with jumps
  from discrete observations}
  \field{volume}{9}
  \field{journaltitle}{Stat. Inference Stoch. Process.}
  \field{year}{2006}
\endentry

\entry{shiryaev1996probability}{book}{}
  \name{author}{1}{}{%
    {{}%
     {Shiryaev}{S.}%
     {A.~N.}{A.~N.}%
     {}{}%
     {}{}}%
  }
  \list{publisher}{1}{%
    {Springer\hyp{}Verlag}%
  }
  \strng{namehash}{SAN1}
  \strng{fullhash}{SAN1}
  \field{labelyear}{1996}
  \field{sortinit}{S}
  \field{edition}{Second edition}
  \field{isbn}{0-387-94549-0}
  \field{note}{Translated from the first (1980) Russian edition by R. P. Boas}
  \field{series}{Graduate Texts in Mathematics}
  \field{title}{Probability}
  \field{volume}{95}
  \list{location}{1}{%
    {New York}%
  }
  \field{year}{1996}
\endentry

\entry{stone1962}{article}{}
  \name{author}{1}{}{%
    {{}%
     {Stone}{S.}%
     {B.~J.}{B.~J.}%
     {}{}%
     {}{}}%
  }
  \strng{namehash}{SBJ1}
  \strng{fullhash}{SBJ1}
  \field{labelyear}{1962}
  \field{sortinit}{S}
  \field{pages}{114\bibrangedash 116}
  \field{title}{Best possible ratios of certain matrix norms}
  \field{volume}{4}
  \field{journaltitle}{Numer. Math.}
  \field{year}{1962}
\endentry

\entry{todorov2006simulation}{article}{}
  \name{author}{2}{}{%
    {{}%
     {Todorov}{T.}%
     {V.}{V.}%
     {}{}%
     {}{}}%
    {{}%
     {Tauchen}{T.}%
     {G.}{G.}%
     {}{}%
     {}{}}%
  }
  \strng{namehash}{TVTG1}
  \strng{fullhash}{TVTG1}
  \field{labelyear}{2006}
  \field{sortinit}{T}
  \field{issn}{0735-0015}
  \field{number}{4}
  \field{pages}{455\bibrangedash 469}
  \field{title}{Simulation methods for {L}\'evy\hyp{}driven
  continuous\hyp{}time autoregressive moving average ({CARMA}) stochastic
  volatility models}
  \field{volume}{24}
  \field{journaltitle}{J. Bus. Econ. Stat.}
  \field{year}{2006}
\endentry

\entry{tucker1965necessary}{article}{}
  \name{author}{1}{}{%
    {{}%
     {Tucker}{T.}%
     {H.~G.}{H.~G.}%
     {}{}%
     {}{}}%
  }
  \strng{namehash}{THG1}
  \strng{fullhash}{THG1}
  \field{labelyear}{1965}
  \field{sortinit}{T}
  \field{issn}{0002-9947}
  \field{pages}{316\bibrangedash 330}
  \field{title}{On a necessary and sufficient condition that an infinitely
  divisible distribution be absolutely continuous}
  \field{volume}{118}
  \field{journaltitle}{Trans. Am. Math. Soc.}
  \field{year}{1965}
\endentry

\entry{vanderwaart1998}{book}{}
  \name{author}{1}{}{%
    {{}%
     {Vaart}{V.}%
     {A.~W.}{A.~W.}%
     {van~der}{v.~d.}%
     {}{}}%
  }
  \list{publisher}{1}{%
    {Cambridge University Press}%
  }
  \strng{namehash}{VAWvd1}
  \strng{fullhash}{VAWvd1}
  \field{labelyear}{1998}
  \field{sortinit}{V}
  \field{isbn}{0-521-49603-9; 0-521-78450-6}
  \field{series}{Cambridge Series in Statistical and Probabilistic Mathematics}
  \field{title}{Asymptotic statistics}
  \field{volume}{3}
  \list{location}{1}{%
    {Cambridge}%
  }
  \field{year}{1998}
\endentry

\lossort
\endlossort

  \blx@bblend
  \endgroup
  \csnumgdef{blx@labelnumber@\the\c@refsection}{0}%
  \iftoggle{blx@reencode}{\blx@reencode}{}}

\usepackage[capitalise]{cleveref}
\crefname{lemma}{Lemma}{Lemmas}
\crefname{equation}{Eq.}{Eqs.}
\Crefname{equation}{Equation}{Equations}
\crefname{assumption}{Assumption}{Assumptions}
\crefname{pluralequation}{Eqs.}{Eqs.}
\creflabelformat{pluralequation}{(#2#1#3)}

\creflabelformat{enumi}{#2#1#3)}

\numberwithin{equation}{section}

\newcommand{\rank}{\operatorname{rank}}
\newcommand{\argmin}{\operatorname{argmin}}
\newcommand{\argmax}{\operatorname{argmax}}
\newcommand{\convd}{\xrightarrow{d}}
\newcommand{\convp}{\xrightarrow{p}}
\newcommand{\eqd}{\stackrel{d}{=}}

\newcommand{\re}{\operatorname{Re}}
\newcommand{\imag}{\operatorname{Im}}

\renewcommand{\leq}{\leqslant}
\renewcommand{\geq}{\geqslant}

\newcommand{\E}{\mathbb{E}}
\newcommand{\Pb}{\mathbb{P}}
\newcommand{\N}{\mathbb{N}}
\newcommand{\Z}{\mathbb{Z}}
\newcommand{\R}{\mathbb{R}}
\newcommand{\I}{{\bf 1 	}}
\newcommand{\C}{\mathbb{C}}
\newcommand{\K}{\mathbb{K}}
\newcommand{\Y}{\boldsymbol{Y}}
\newcommand{\X}{\boldsymbol{X}}
\newcommand{\Lb}{\boldsymbol{L}}

\newcommand{\W}{\boldsymbol{W}}
\newcommand{\ZZ}{\boldsymbol{Z}}
\newcommand{\bth}{\boldsymbol{\vartheta}}
\newcommand{\beps}{\boldsymbol{\varepsilon}}

\newcommand{\bu}{\boldsymbol{u}}
\newcommand{\bx}{\boldsymbol{x}}
\newcommand{\bzero}{\boldsymbol{0}}
\newcommand{\be}{\boldsymbol{e}}
\newcommand{\nuL}{\nu^{\Lb}}
\newcommand{\bgammaL}{\bgamma^{\Lb}}

\newcommand{\dd}{\mathrm{d}}
\newcommand{\ee}{\mathrm{e}}
\newcommand{\ii}{\mathrm{i}}
\newcommand{\DD}{\mathrm{D}}

\newcommand{\bgamma}{\boldsymbol{\gamma}}
\newcommand{\ppartial}{\mathrm{\partial}}

\newcommand{\A}{{\bf\operatorname{A}}}
\newcommand{\B}{{\bf\operatorname{B}}}
\newcommand{\M}{{\bf\operatorname{M}}}

\theoremstyle{plain}
\newtheorem{theorem}{Theorem}[section]
\newtheorem{proposition}[theorem]{Proposition}
\newtheorem{lemma}[theorem]{Lemma}
\newtheorem{corollary}[theorem]{Corollary}

\theoremstyle{definition}
\newtheorem{definition}[theorem]{Definition}

\theoremstyle{remark}
\newtheorem{remark}[theorem]{Remark}

\newcounter{assumption}
\newtheorem{assumption}[assumption]{Assumption}

\title[Estimation of the driving L\'evy process of MCARMA processes]{Parametric estimation of the driving L\'evy process of multivariate CARMA processes from discrete observations}

\author{Peter J. Brockwell}
\address{Departments of Statistics, Colorado State University and Columbia University,  USA}
\email{pjbrock@stat.colostate.edu}

\author{Eckhard Schlemm}
\address{Wolfson College, University of Cambridge}
\email{es555@cam.ac.uk}

\begin{document}


\begin{abstract}
We consider the parametric estimation of the driving L\'evy process of a multivariate continuous-time autoregressive moving average (MCARMA) process, which is observed on the discrete time grid $(0,h,2h,\ldots)$. Beginning with a new state space representation, we develop a method to recover the driving L\'evy process exactly from a continuous record of the observed MCARMA process. We use tools from numerical analysis and the theory of infinitely divisible distributions to extend this result to allow for the approximate recovery of unit increments of the driving L\'evy process from discrete-time observations of the MCARMA process. We show that, if the sampling interval $h=h_N$ is chosen dependent on $N$, the length of the observation horizon, such that $N h_N$ converges to zero as $N$ tends to infinity, then any suitable generalized method of moments estimator based on this reconstructed sample of unit increments has the same asymptotic distribution as the one based on the true increments, and is, in 
particular, asymptotically normally distributed.
\end{abstract}


\maketitle

\tableofcontents

\section{Introduction}
Continuous-time autoregressive moving average (CARMA) processes generalize the widely employed discrete-time ARMA process to a continuous-time setting. Heuristically, a multivariate CARMA process of order $(p,q)$ can be thought of as a stationary solution $\Y$ of the linear differential equation
\begin{equation}
\label{eq-DefMCARMAheu}
\left[\DD^p+A_1\DD^{p-1}+\ldots+A_p\right]\Y(t) = \left[B_0+B_1\DD+\ldots+B_q\DD^{q}\right]\DD\Lb(t),\quad \DD=\frac{\dd}{\dd t},\quad p>q,
\end{equation}
where $\Lb$ is a L\'evy process and $A_i$, $B_j$ are coefficient matrices, see \cref{section-MCARMA} for a precise definition. They first appeared in the literature in \citep{doob1944egp}, where univariate Gaussian CARMA processes were defined. Recent years have seen a rapid development in both the theory and the applications of this class of stochastic processes (see, e.g., \citep{brockwell2004rct} and references therein). In \citep{brockwell2001levy}, the restriction of Gaussianity was relaxed and CARMA processes driven by L\'evy processes with finite moments of any order greater than zero were introduced (see also \citep{brockwell2009existence}). This extension allowed for CARMA processes to have jumps as well as a wide variety of marginal distributions, possibly exhibiting fat tails. Shortly after that, \citep{marquardt2007multivariate} defined multivariate CARMA processes and thereby made it possible to model a set of dependent time series jointly by a single continuous-time linear process. Further 
developments of the concept led to fractionally integrated CARMA (FICARMA, \citep{brockwell2005lda,marquardt2007mficarma}) and superpositions of CARMA (supCARMA, \citep{barndorff2011multivariate}) processes, both allowing for long-memory effects. In many contexts continuous-time processes are particularly suitable for stochastic modelling because they allow for irregularly-spaced observations and high-frequency sampling. We refer the reader to \citep{barndorff2011nongaussian,benth2009dynamic,todorov2006simulation} for an overview of successful applications of CARMA processes in economics and mathematical finance.

Despite the growing interest of practitioners in using CARMA processes as stochastic models for observed time series, the statistical theory for such processes has received little attention in the past. One of the basic questions with regard to parameter inference or model selection is how to determine which particular member of a class of stochastic models best describes the characteristic statistical properties of an observed time series. If one decides to model a phenomenon by a CARMA process as in \cref{eq-DefMCARMAheu}, which can often be argued to be a reasonable choice of model class, this problem reduces to the three tasks of choosing suitable integers $p$, $q$ describing the order of the process; estimating the coefficient matrices $A_i$, $B_j$; and suggesting an appropriate model for the driving L\'evy process $\Lb$.

In this paper, we address the last of these three problems and develop a method to estimate a parametric model for the driving L\'evy process of a multivariate CARMA process, building on an idea suggested in \citep{brockwell2010estimation} for the special case of a univariate CARMA process of order $(2,1)$. The strategy is to observe that the distribution of a L\'evy process $\Lb$ is uniquely determined by the distribution of the unit increments $\Delta\Lb_n=\Lb(n)-\Lb(n-1)$; if one therefore had access to the increments $(\Delta\Lb_n)_{n=1,\ldots,N}$ over a sufficiently long time-horizon, one could easily estimate a model for $\Lb$ by any of several well-established methods, including parametric as well as non-parametric approaches (\citep{figueroa2009nonparametric,gugushvili2009nonparametric} and references therein). It is thus natural to try and express the increments of the driving L\'evy process -- at least approximately -- in terms of the observed values of the CARMA process and subject this 
approximate sample from the unit-increment distribution to the same estimation method one would use with the true sample. One difficulty arising in this step is that one usually does not observe a CARMA processes continuously but that one instead only has access to its values on a discrete, yet possibly very fine, time grid; in fact, as we shall see in \cref{section-recovery}, it is this assumption of discrete-time observations that prevents us from exactly recovering the increments of the L\'evy process from the recorded CARMA process. 

In this paper, we concentrate on the parametric generalized moment estimators (see, e.g., \citep{hansen1982,newey1994largesample}) and prove that the estimate based on the reconstructed increments of $\Lb$ has the same asymptotic distribution as the estimate based on the true increments, provided that both the length $N$ of the observation period and the sampling frequency $h^{-1}$, at which the CARMA process is recorded, go to infinity at the right rate. In fact we obtain the quantitative criterion that $h=h_N$ must be chosen dependent on $N$ such that $Nh_N$ converges to zero as $N$ tends to infinity. The generalized method of moments (GMM) estimators contain as special cases the classical maximum likelihood estimators as well as non-linear least squares estimators that are based on fitting the empirical characteristic function of the observed sample to its theoretical counterpart. In view of the structure of the L\'evy--Khintchine formula, the latter method is particularly suited for the estimation of L\'
evy processes. We impose no assumptions on the driving L\'evy process except for the finiteness of certain moments that depend on the particular moment function used in the GMM approach. In our main result, \cref{theorem-GMMLevy}, we prove the consistency and asymptotic normality of a wide class of GMM estimators that satisfy a set of mild standard technical assumptions. 

For some recent results about the estimation of discretely observed diffusion processes with jumps we refer the reader to \citep{ngo2010parametric,ogihara2011quasi, shimizu2006estimation} and their references. The theory developed in these papers is not applicable to the problem of estimating the driving L\'evy process of a discretely observed MCARMA process because MCARMA processes are not, in general, diffusions.

It seems possible to relax the assumption of uniform sampling as long as the maximal distance between two recording times in the observation interval tends to zero. More important, however, is the natural question if there exist methods to estimate the driving L\'evy process of a CARMA process that do not require high frequency sampling but still have desirable asymptotic properties. Another interesting topic for further investigation is the behaviour of non-parametric estimators for the driving L\'evy process if they are used with a disturbed sample of the unit increments as described in this paper.

\subsection{Outline of the paper}
The paper is structured as follows. In \cref{section-MLP} we take a closer look at multivariate L\'evy processes and infinitely divisible distributions, the fundamental ingredients in the definition of a multivariate CARMA process. First, we briefly review their definition and some important basic properties. In \cref{sec-absolutemoments} we obtain a new quantitative bound for the absolute moments of an infinitely divisible distribution in terms of its characteristic triplet, which is essential for many of the subsequent proofs. We also derive the exact polynomial time-dependence of the absolute moments of a L\'evy process in \cref{prop-levymoments}. As a further preparation for the proofs of our main results, \cref{theorem-fubini} in \cref{sec-fubini} establishes a Fubini-type result for double integrals with respect to a L\'evy process over an unbounded domain.

The definition of multivariate CARMA processes as well as important properties, such as moments, mixing and smoothness of sample paths, are presented in \cref{section-MCARMA}. In \cref{theorem-AlternativeSSR}, we prove an alternative state space representation for multivariate CARMA processes, called the controller canonical form, which lends itself more easily to the estimation of the driving L\'evy process than the original definition.

In \cref{section-recovery} we show that, conditional on an initial value, whose influence decays exponentially, one can exactly recover the value of the driving L\'evy process from a continuous record of the multivariate CARMA process. The functional dependence is explicit and given in \cref{theorem-recoveryDeltaL}.

Since such a continuous record is usually not available, \cref{section-DTestimation} is devoted to discretizing the result found in \cref{theorem-recoveryDeltaL}. To this end, we analyse how pathwise derivatives and definite integrals of L\'evy-driven CARMA processes can be approximated from observations on a discrete time grid, and we determine the asymptotic behaviour of these approximations as the mesh size tends to zero. To our knowledge, this is the first time that numerical differentiation and integration schemes are investigated quantitatively for this class of stochastic processes. The results of this section are summarized in \cref{theorem-Levyapprox}.

In \cref{sec-GMM}, we prove consistency and asymptotic normality of the generalized method of moments estimator when the sample is not i.i.d$.$ but instead disturbed by a noise sequence, which corresponds to the discretization error from the previous section. \Cref{theorem-GMMnew} shows that if the sampling frequency $h_N^{-1}$ goes to infinity fast enough with the length $N$ of the observation interval, such that $N h_N$ converges to zero, then the effect of the discretization becomes asymptotically negligible and the limiting distribution of the estimated parameter is identical to the one obtained from an unperturbed sample. Finally, in \cref{theorem-GMMLevy}, we apply this result to give an answer to the question of how to estimate a parametric model of the driving L\'evy process of a multivariate CARMA process if high-frequency observations are available.

Finally, we present the results of a simulation study for a Gamma-driven CARMA(3,1) process in \cref{section-simulation}.

\Cref{appendix} contains auxiliary results and some technical proofs that complement the presentation of our results in the main part of the paper.
\subsection{Notation}
Throughout the paper we use the following notation. The natural, real, complex numbers and the integers are denoted by $\N$, $\R$, $\C$ and $\Z$, respectively. Vectors in $\R^m$ are printed in bold, and we use superscripts to denote the components of a vector, e.g., $\R^m\ni\bx=(x^1,\ldots,x^m)$. We write $\bzero_m$ for the zero vector in $\R^m$, and we let $\left\|\cdot\right\|$ and $\langle\cdot\rangle$ represent the Euclidean norm and inner product, respectively. The ring of polynomial expressions in $z$ over a ring $\K$ is denoted by $\K[z]$. The symbols $M_{m,n}(\K)$, or $M_m(\K)$ if $m=n$, stand for the space of $m\times n$ matrices with entries in $\K$. The transpose of a matrix $A$ is written as $A^T$, and $\I_m$ and $0_m$ denote the identity and the zero element in $M_m(\K)$, respectively. The symbol $\left\|\cdot\right\|$ is also used for the operator norm on $M_{m,n}(\R)$ induced by the Euclidean vector norm. For any topological space $X$, the symbol $\mathscr{B}(X)$ denotes the Borel $\sigma$-
algebra on $X$. We frequently use the following Landau notation: for two functions $f$ and $g$ defined on the interval $[0,1]$ we write $f(h)=O\left(g(h)\right)$ if there exists a constant $C$ such that $\left\|f(h)\right\|\leq C g(h)$ for all $h<1$. We use the notation $\left\|\cdot\right\|_{L^p}$ for the norm on the classical $L^p$ spaces. The symbol $\lambda$ stands for the Lebesgue measure, and the indicator function of a set $B$ is denoted by $I_B(\cdot)$, defined to be one if the argument lies in $B$ and zero otherwise. We write $\convp$ and $\convd$ for convergence in probability and convergence in distribution, respectively, and use the symbol $\eqd$ to denote equality in distribution of two random variables. For a positive real number $\alpha$, we write $(\alpha)_0$ for the smallest even integer greater than or equal to $\alpha$.

Throughout the paper, the symbol $h$ denotes a sampling interval or, equivalently, the inverse of the sampling frequency at which a continuous-time process is recorded. $N$ is the length of the observation horizon and thus also the number of unit increments of the the L\'evy process that can be reconstructed from observing the MCARMA process over that period. $N$ is \emph{not} the total number of observations, which is $N/h$.

\section{L\'evy processes and infinitely divisible distributions}
\label{section-MLP}
\subsection{Definition and L\'evy--It\^o decomposition}
L\'evy processes are the main ingredient in the definition of  a multivariate CARMA process and an important object of study in this paper. In this section we review their definition and some elementary properties. A detailed account can be found in \citep{applebaum2004lpa,sato1991lpa}.
\begin{definition}
\label{Def-LevyProcess}
A (one-sided) $\R^m$-valued {\em L\'evy process} $\left(\Lb(t)\right)_{t\geq 0}$ is a stochastic process, defined on the probability space $(\Omega,\mathscr{F},\Pb)$, with stationary, independent increments, continuous in probability and satisfying $\Lb(0)=\bzero_m$ almost surely.
\end{definition}
Every $\R^m$-valued L\'evy process $\left(\Lb(t)\right)_{t\geq 0}$ can without loss of generality be assumed to be c\`adl\`ag, which means that the sample paths are right-continuous and have left limits; it is completely characterized by its characteristic function in the L\'evy--Khintchine form $\E \ee^{\ii\langle\bu,\Lb(t)\rangle}=\exp\{t\psi^{\Lb}(\bu)\}$, $\bu\in\R^m$, $t\geq 0$, where $\psi^{\Lb}$ has the special form
\begin{equation}
\label{eq-characfunc}
 \psi^{\Lb}(\bu)=\ii\langle \bgammaL,\bu\rangle-\frac{1}{2}\langle \bu,\Sigma^{\mathcal{G}}\bu\rangle+\int_{\R^m}{\left[\ee^{\ii\langle\bu,\bx\rangle}-1-\ii\langle\bu,\bx\rangle I_{\{||x||\leq 1\}}\right]\nuL (\dd\bx)}.
\end{equation}
The vector $\bgammaL\in\R^m$ is called the {\it drift}, the non-negative definite, symmetric $m\times m$ matrix $\Sigma^{\mathcal{G}}$ is the {\it Gaussian covariance matrix} and $\nuL $ is a measure on $\R^m$, referred to as the {\it L\'evy measure}, satisfying
\begin{equation*}
\nuL (\{\bzero_m\})=0,\quad \int_{\R^m}\min(||\bx||^2,1)\nuL (\dd\bx)<\infty.
\end{equation*}
Put differently, for every $t\geq 0$, the distribution of $\Lb(t)$ is infinitely divisible with characteristic triplet $(t\bgamma,t\Sigma^{\mathcal{G}},t\nuL)$. By the L\'evy--It\^o decomposition the paths of $\Lb$ can be decomposed almost surely into a Brownian motion with drift, a compound Poisson process and a purely discontinuous $L^2$-martingale according to
\begin{equation*}
\label{eq-LevyItodecomposition}
\Lb(t) = \bgamma t + \Sigma^{\mathcal{G},1/2}\W_t + \int_{\left\|\bx\right\|\geq 1}\int_0^t{\bx N(\dd s,\dd\bx)}+\lim_{\varepsilon\searrow 0}\int_{\varepsilon\leq\left\|\bx\right\|\leq 1}\int_0^t{\bx \tilde N(\dd s,\dd\bx)},
\end{equation*}
where $\W$ is a standard $m$-dimensional Wiener process and $\Sigma^{\mathcal{G},1/2}$ is the unique positive semidefinite matrix square root of $\Sigma^{\mathcal{G}}$. The measure $N$ is a Poisson random measure on $\R\times\R^m\backslash\{\bzero_m\}$, independent of $\W$ with intensity measure $\lambda\otimes\nuL$ describing the jumps of $\Lb$. More precisely, for any measurable set $B\in\mathscr{B}(\R\times\R^m\backslash\{\bzero_m\})$,
\begin{equation*}
N(B) = \#\left\{s\geq 0: \left(s,\Lb(s)-\Lb(s-)\right)\in B\right\},\quad \Lb(s-)\coloneqq\lim_{t\nearrow s}\Lb(t).
\end{equation*}
Finally, $\tilde N$ is the compensated jump measure defined by $\tilde N(\dd s,\dd\bx) = N(\dd s,\dd\bx)-\dd s\nuL(\dd\bx)$. We will work with two-sided L\'evy processes $\Lb=\left(\Lb(t)\right)_{t\in\R}$. These are obtained from two independent copies $\left(\Lb_1(t)\right)_{t\geq 0}$, $\left(\Lb_2(t)\right)_{t\geq 0}$ of a one-sided L\'evy process via the construction
\begin{equation*}
\Lb(t)=\begin{cases}
        \Lb_1(t), & t\geq 0, \\
	-\Lb_2(-t-), & t<0.
       \end{cases}
\end{equation*}

In the following we present some elementary facts about stochastic integrals with respect to L\'evy processes, which we will use later. Comprehensive accounts of this wide field are given in the textbooks \citep{applebaum2004lpa,protter1990sia}. Let $f:\R\to M_{d,m}(\R)$ be a Lebesgue measurable, square-integrable function. Under the condition that $\Lb(1)$ has finite second moments, the stochastic integral
\begin{equation*}
I = \int_{\R}f(s)\dd\Lb(s)
\end{equation*}
exists in $L^2(\Omega,\Pb)$. Moreover, the distribution of the random variable $I$ is infinitely divisible with characteristic triplet $(\bgamma_f,\Sigma_f,\nu_f)$ which can be expressed explicitly in terms of the characteristic triplet of $\Lb$ via the formulas (\citep[Theorem 2.7]{rajput1989spectral})
\begin{subequations}
\label[pluralequation]{eq-transformtriplet}
\begin{align}
\label{eq-transformtripletgamma}\bgamma_f =& \int_{\R}f(s)\left[\bgammaL +\int_{\R^d}{\bx\left(I_{[0,1]}(\left\|f(s)\bx\right\|)-I_{[0,1]}(\left\|\bx\right\|)\right)\nuL(\dd\bx)}\right]\dd s,\\
\label{eq-transformtripletsigma}\Sigma_f =& \int_{\R}{f(s)\Sigma^{\mathcal{G}}f(s)^T\dd s},\\ 
\label{eq-transformtripletnu}\nu_f(B) =& \int_{\R}\int_{\R^m}{I_B(f(s)\bx)\nuL(\dd\bx)\dd s},\quad B\in \mathscr{B}(\R^d\backslash\{\bzero_d\}).
\end{align}
\end{subequations}

\subsection[Absolute moments of infinitely divisible distributions]{Bounds for the absolute moments of infinitely divisible distributions and L\'evy processes}
\label{sec-absolutemoments}
In this short section we derive some bounds for the absolute moments of multivariate infinitely divisible distributions and L\'evy processes which will turn out to be essential for the proofs of our main results later. It is well known that the $k$th absolute moment of an infinitely divisible random variable $X$ with characteristic triplet $(\bgamma, \Sigma,\nu)$ is finite if and only if the measure $\nu$, restricted to $\left\{\left\|\bx\right\|\geq 1\right\}$, has a finite $k$th absolute moment. We need the following stronger result, which establishes a quantitative bound for the absolute moments of an infinitely divisible distribution in terms of its characteristic triplet. 
\begin{lemma}
\label{lemma-momentsidrv}
Let $X$ be an infinitely divisible, $\R^m$-valued random variable with characteristic triplet $(\bgamma,\Sigma,\nu)$ and let $k$ be a positive even integer. Assume that the constants $c_i,C_i$, $i=1,2$, satisfy
\begin{subequations}
\begin{align}
\label{eq-momentsIDassum0}
\int_{\left\|\bx\right\|< 1}\left\|\bx\right\|^r\nu(\dd\bx)\leq C_0 c_0^r,\quad r=2,\ldots,k,\\
\label{eq-momentsIDassum1}\int_{\left\|\bx\right\|\geq 1}\left\|\bx\right\|^r\nu(\dd\bx)\leq C_1 c_1^r,\quad r=1,\ldots,k.
\end{align}
\end{subequations}
Then there exists a constant $C>0$, depending on $m$ and $k$, but not on $(\bgamma,\Sigma,\nu)$, such that
\begin{equation}
\label{eq-momentsID}
\E\left\|X\right\|^k\leq C\left[\left\|\bgamma\right\|^k+\left\|\Sigma\right\|^{k/2} + c_0^k + c_1^k\right].
\end{equation}
\end{lemma}
\begin{proof}
Denote by $\nu_0=\nu|_{\{\left\|\bx\right\|<1\}}$ and $\nu_1=\nu|_{\{\left\|\bx\right\|\geq 1\}}$ the restrictions of the measure $\nu$ to the unit ball of $\R^m$ and its complement, respectively. It follows from the L\'evy--Khintchine formula \labelcref{eq-characfunc} that we can construct a standard normal random variable $\W$ and two infinitely divisible random variables  $X_0$, $X_1$, with characteristic triplets $(\bzero_m,0_m,\nu_0)$, $(\bzero_m,0_m,\nu_1)$, and distributions $\mu_0$, $\mu_1$, respectively, such that $X\eqd \bgamma+\Sigma^{1/2}\W+X_0+X_1$. Using the notation $n!!$ for the double factorial of the natural number $n$ as well as \citep[Eq. (4.20)]{bauer2002wahrscheinlich} for the absolute moments of a standard normal random variable, the $k$th absolute moment of the Gaussian part is readily estimated as
\begin{equation*}
\E\left\|\Sigma^{1/2}\W\right\|^k \leq \left\|\Sigma\right\|^{k/2}\E\left\|\W\right\|^k\leq \left\|\Sigma\right\|^{k/2} \E\left(\sum_{i=1}^m\left|W^i\right|\right)^k \leq \left\|\Sigma\right\|^{k/2} m^{k+1}\E\left|W^1\right|^k\leq (k-1)!!\left\|\Sigma\right\|^{k/2} m^{k+1},
\end{equation*}
which implies that
\begin{equation}
\label{eq-momentsIDdecomp}
\E\left\|X\right\|^k \leq 4^k\left[\left\|\bgamma\right\|^k + m^{k+1}(k-1)!!\left\|\Sigma\right\|^{k/2} + \E\left\|X_0\right\|^k + \E\left\|X_1\right\|^k\right].
\end{equation}
The first two terms in this sum are already of the form asserted in \cref{eq-momentsID}. We next consider the fourth term. By construction, the characteristic function of $X_1$ is given by
\begin{equation*}
\widehat{\mu_1}(\bu)\coloneqq\E \ee^{\ii\langle\bu,X_1\rangle} = \exp\left\{\int_{\left\|\bx\right\|\geq1}\left[\ee^{\ii\langle\bu,\bx\rangle}-1\right]\nu(\dd\bx)\right\},\quad\bu\in\R^m.
\end{equation*}
By assumption \labelcref{eq-momentsIDassum1} and \citep[Corollary 25.8]{sato1991lpa}, $\int\left\|\bx\right\|^k\mu_1(\dd\bx)<\infty$ and \citep[Proposition 2.5(ix)]{sato1991lpa} shows that the mixed moments of $X_1$ of order $k$ are given by
\begin{equation*}
\E\left(X_1^{i_1}\cdot\ldots\cdot X_1^{i_k}\right) = \int_{\R^m}{x^{i_1}\cdot\ldots\cdot x^{i_k}\mu_1(\dd\bx)} = \frac{1}{i^k}\left.\frac{\ppartial^k}{\ppartial u^{i_1}\cdot\ldots\cdot\ppartial u^{i_k}}\widehat{\mu_1}(\bu)\right|_{\bu=\bzero_m},\quad i_j=1,\ldots,m.
\end{equation*}
It is easy to see by induction that
\begin{equation*}
\frac{\ppartial^k}{\ppartial u^{i_1}\cdot\ldots\cdot\ppartial u^{i_k}}\widehat{\mu_1}(\bu) = \left[\widehat{\mu_1}(\bu)\right]^ki^k\sum_{\pi\in\mathcal{P}_k}\prod_{B\in\pi}\int_{\left\|\bx\right\|\geq1}{\left[\prod_{j\in B}x^{i_j}\right]\ee^{\ii\langle\bu,\bx\rangle}\nu(\dd\bx)},
\end{equation*}
where $\mathcal{P}_k$ denotes the set of partitions of $\{1,2,\ldots,k\}$, a partition being a subset of the power set of $\{1,\ldots,k\}$ with pairwise disjoint elements such that their union is equal to $\{1,\ldots,k\}$. We write $\#\pi$ for the number of sets in a partition $\pi$ and $|B|$ for the number of elements in such a set. Setting $\bu=\bzero_m$, specializing to $i_j=i$ and making use of the assumption that $k$ is even, the last display yields the explicit formula
\begin{equation*}
\E\left|X_1^i\right|^k=\E\left(X_1^i\right)^k=\sum_{\pi\in\mathcal{P}_k}\prod_{B\in\pi}\int \left(x^i\right)^{|B|}\nu(\dd\bx),\quad i=1,\ldots,m.
\end{equation*}
Using the fact that $x^i\leq\left\|\bx\right\|$ for every $\bx\in\R^m$ as well as assumption \labelcref{eq-momentsIDassum1} we thus obtain that
\begin{equation}
\label{eq-momentsIDX1}
\E\left\|X_1\right\|^k \leq m^{k/2}\sum_{i=1}^m\E\left|X_1^{i}\right|^k\leq m^{k/2+1}\sum_{\pi\in\mathcal{P}_k}\prod_{B\in\pi}\int_{\left\|\bx\right\|\geq1}\left\|x\right\|^{|B|}\nu(\dd\bx)\leq c_1^km^{k/2+1}\sum_{\pi\in\mathcal{P}_k}C_1^{\#\pi}.
\end{equation}
The third term in \cref{eq-momentsIDdecomp} can be analysed similarly: the characteristic function of $X_0$ has the form
\begin{equation*}
\widehat{\mu_0}(\bu)\coloneqq\E \ee^{\ii\langle\bu,X_0\rangle} = \exp\left\{\int_{\left\|\bx\right\|< 1}\left[\ee^{\ii\langle\bu,\bx\rangle}-1-\ii\langle\bu,\bx\rangle\right]\nu(\dd\bx)\right\},\quad\bu\in\R^m.
\end{equation*}
With $\nu_0$ having bounded support, all moments of $X_0$ are finite, which implies that $\widehat{\mu_0}$ is infinitely often differentiable and that the mixed moments of $X_0$ are given by partial derivatives of $\widehat{\mu_0}$, as before. The additional compensatory term $\ii\langle\bu,\bx\rangle$ in the integral ensures that the first derivative of $\widehat{\mu_0}$ vanishes at zero, which leads to
\begin{equation*}
\E \left|X_0^i\right|^k=\E \left(X_0^i\right)^k=\sum_{\substack{\pi\in\mathcal{P}_k\\\min\{|B|,B\in\pi\}\geq 2}}\prod_{B\in\pi}\int_{\left\|\bx\right\|< 1}\left(x^i\right)^{|B|}\nu(\dd\bx),\quad i=1,\ldots,m.
\end{equation*}
Using assumption \labelcref{eq-momentsIDassum0} we can thus estimate
\begin{equation}
\label{eq-momentsIDX0}
\E\left\|X_0\right\|^k \leq m^{k/2}\sum_{i=1}^m\E\left|X_0^{i}\right|^k\leq m^{k/2+1}\sum_{\substack{\pi\in\mathcal{P}_k\\\min\{|B|,B\in\pi\}\geq 2}}\prod_{B\in\pi}\int_{\left\|\bx\right\|<1}\left\|x\right\|^{|B|}\nu(\dd\bx)\leq c_0^km^{k/2+1}\sum_{\substack{\pi\in\mathcal{P}_k\\\min\{|B|,B\in\pi\}\geq 2}}C_0^{\#\pi}.
\end{equation}
The bounds \labelcref{eq-momentsIDdecomp,eq-momentsIDX1,eq-momentsIDX0} show that the claim \labelcref{eq-momentsID} holds with
\begin{equation*}
C\coloneqq 4^k\left[m^{k+1}(k-1)!!+m^{k/2+1}\left(\sum_{\pi\in\mathcal{P}_k}C_1^{\#\pi}+\sum_{\substack{\pi\in\mathcal{P}_k\\\min\{|B|,B\in\pi\}\geq 2}}C_0^{\#\pi}\right)\right].\qedhere
\end{equation*}
\end{proof}
Since the marginal distributions of a L\'evy process $\Lb$ are infinitely divisible, the behaviour of their moments can be analysed by the previous \cref{lemma-momentsidrv}. We prefer, however, to give an exact description of the time-dependence of $\E\left\|\Lb(t)\right\|^k$ for even exponents $k$ and derive from that the asymptotic behaviour as $t$ tends to zero.

\begin{proposition}
\label{prop-levymoments}
Let $k$ be a positive real number and $\Lb$ be a L\'evy process.
\begin{enumerate}[i)]
\item\label{prop-levymoments-even} If $k$ is an even integer and $\E\left\|\Lb(1)\right\|^k$ is finite, then there exist real numbers $m_1,\ldots,m_k$ such that
\begin{equation}
\label{eq-Levypprocessmoments}
\E\left\|\Lb(t)\right\|^k=m_1 t+\ldots + m_k t^k,\quad t\geq 0.
\end{equation}
\item\label{prop-levymoments-general} If $\E\left\|\Lb(1)\right\|^{(k)_0}$ is finite, then $\E\left\|\Lb(h)\right\|^k=O(h^{k/(k)_0})$ as $h\to 0$.
\end{enumerate}
\end{proposition}
\begin{proof}
For the proof of \labelcref{prop-levymoments-even} we introduce the notation $\K\left(L^{i_1}(t),\ldots,L^{i_k}(t)\right)$, $1\leq i_1,\ldots,i_k\leq m$, for the mixed cumulants of $\Lb(t)$ of order $k$. They are defined in terms of the characteristic function of $\Lb$ as
\begin{equation*}
\K\left(L^{i_1}(t),\ldots,L^{i_k}(t)\right)= \left.\frac{\ppartial^k}{\ppartial u_{i_1}\cdots\ppartial_{u_{i_k}}}\log\E \ee^{\ii\langle\bu,\Lb(t)\rangle}\right|_{\bu=\bzero_m},
\end{equation*}
and are clearly homogeneous of degree one in $t$. There is a close combinatoric relationship between moments and cumulants, which was used implicitly in the proof of \cref{lemma-momentsidrv} and which explicitly reads (see \citep[\S 12, Theorem 6]{shiryaev1996probability}): 
\begin{equation*}
\E L^{i_1}(t)\cdot\cdots\cdot L^{i_k}(t)=\sum_{\pi\in\mathcal{P}_k}\prod_{B\in\pi}\K\left(L^{i_j}(t):j\in B\right) = \sum_{\pi\in\mathcal{P}_k}t^{\#\pi}\prod_{B\in\pi}\K\left(L^{i_j}(1):j\in B\right)=\sum_{\kappa=1}^km_{k,\kappa}^{i_1,\ldots,i_k}t^\kappa,
\end{equation*}
where
\begin{equation*}
m_{k,\kappa}^{i_1,\ldots,i_k}=\sum_{\substack{\pi\in\mathcal{P}_k\\\#\pi=\kappa}}\prod_{B\in\pi}\K\left(L^{i_j}(1):j\in B\right).
\end{equation*}
Writing $k=2l$, the Multinomial Theorem implies that
\begin{align*}
\|\Lb(t)\|^k =& \left[\left(L^1(t)\right)^2+\ldots + \left(L^m(t)\right)^2\right]^l = \sum_{\substack{0\leq l_1,\ldots,l_m\leq l\\l_1+\ldots+l_m=l}}{\frac{l!}{l_1!\cdot\cdots\cdot l_m!}\prod_{i=1}^m{\left(L^i(t)\right)^{2l_i}}}
\end{align*}
and thus it follows, by what was just shown and the linearity of expectation, that
\begin{align*}
\E\|\Lb(t)\|^k = \sum_{\kappa=1}^k{\left[\sum_{\substack{0\leq l_1,\ldots,l_m\leq l\\l_1+\ldots+l_m=l}}\frac{l!}{l_1!\cdot\cdots\cdot l_m!}m_{k,\kappa}^{\overbrace{\scriptstyle 1,\ldots,1}^{2l_1 \text{ times}},\ldots,\overbrace{\scriptstyle m,\ldots,m}^{2l_m \text{ times}}}\right]t^\kappa}.
\end{align*}
This proves \cref{eq-Levypprocessmoments}. Assertion \labelcref{prop-levymoments-general} follows for even $k$ directly from the polynomial time-dependence of $\E\left\|\Lb(t)\right\|^k$ which we have just established. For general $k$ we use H\"older's inequality which implies that
\begin{equation*}
\E\left\|\Lb(t)\right\|^k\leq \left(\E\left\|\Lb(t)\right\|^{(k)_0}\right)^{\frac{k}{(k)_0}}
\end{equation*}
and since $(k)_0$ is even by definition the claim follows again from part \labelcref{prop-levymoments-even}.
\end{proof}

\subsection[A Fubini-type theorem for stochastic integrals]{A Fubini-type theorem for stochastic integrals with respect to L\'evy processes}
\label{sec-fubini}

The next result is a Fubini-type theorem for a special class of stochastic integrals with respect to L\'evy processes over an unbounded domain.
\begin{theorem}
\label{theorem-fubini}
Let $[a,b]\subset\R$ be a bounded interval and $\Lb$ be a L\'evy process with finite second moments. Assume that $F:[a,b]\times\R\to M_{d,m}(\R)$ is a bounded function, and that the family $\{u\mapsto F(s,u)\}_{s\in[a,b]}$ is uniformly absolutely integrable and uniformly converges to zero as $|u|\to\infty$. It then holds that
\begin{equation}
\label{eq-fubini}
\int_a^b\int_\R F(s,u)\dd\Lb(u)\dd s = \int_\R\int_a^b F(s,u)\dd s\dd\Lb(u),
\end{equation}
almost surely.
\end{theorem}
\begin{proof}
We first note that since $\Lb$ has finite second moments and $F$ is square-integrable, both integrals in \cref{eq-fubini} are well-defined as $L^2$-limits of approximating Riemann-Stieltjes sums. We start the proof by introducing the notations
\begin{align*}
I=\int_a^b\int_\R F(s,u)\dd\Lb(u)\dd s,\quad& I_N=\int_a^b\int_{-N}^NF(s,u)\dd\Lb(u)\dd s,\\
\overleftrightarrow{I}=\int_\R\int_a^bF(s,u)\dd s\dd\Lb(u),\quad& \overleftrightarrow{I_N}=\int_{-N}^N\int_a^bF(s,u)\dd s\dd\Lb(u).
\end{align*}
It follows from \citep[Theorem 1]{kailath1978fubini} (see also \citep[Theorem 64]{protter1990sia}) that, for each $N$, $I_N=\overleftrightarrow{I_N}$ almost surely. We also write
\begin{equation*}
\Delta_N \coloneqq I - I_N,\quad \overleftrightarrow{\Delta_N} \coloneqq \overleftrightarrow{I} - \overleftrightarrow{I_N},\quad N>0.
\end{equation*}
The strategy of the proof is to show that both $\Delta_N$ and $\overleftrightarrow{\Delta_N}$ converge to zero as $N$ tends to infinity, and then to use the uniqueness of limits to conclude that $I$ must equal $\overleftrightarrow{I}$. We first investigate $\E\left\|\Delta_N\right\|^2$. Clearly,
\begin{equation}
\label{eq-DeltaN}
\Delta_N = \int_a^b\int_{|u|>N} F(s,u)\dd\Lb(u)\dd s.
\end{equation}
Consequently, in order to analyse the absolute moments of $\Delta_N$ it suffices to consider the absolute moments of the infinite divisible random variables $\int_{|u|>N} F(s,u)\dd\Lb(u)$, $s\in[a,b]$. By \cref{eq-transformtriplet}, their characteristic triplets $(\bgamma_{F,N}^s,\Sigma_{F,N}^s,\nu_{F,N}^s)$ satisfy
\begin{align}
\label{eq-gammaFNs}
\left\|\bgamma_{F,N}^s\right\| \leq& \int_{|u|>N}\left\|F(s,u)\right\|\dd u\left\|\bgammaL\right\| +\int_{|u|>N}\left\|F(s,u)\right\|\int_{\left\|\bx\right\|<1}{\left\|\bx\right\| I_{[1,\infty)}(\left\|F(s,u)\bx\right\|)\nuL(\dd\bx)}\dd u\notag\\
  &+\int_{|u|>N}\left\|F(s,u)\right\|\int_{\left\|\bx\right\|\geq 1}{\left\|\bx\right\| I_{[0,1]}(\left\|F(s,u)\bx\right\|)\nuL(\dd\bx)}\dd u\notag\\
 \leq & \int_{|u|>N}\left\|F(s,u)\right\|\dd u\left[\left\|\bgammaL\right\|+\int_{\left\|\bx\right\|\geq 1}{\left\|\bx\right\|\nuL(\dd\bx)}\right]
\end{align}
for all $N$ exceeding some $N_0$ which satisfies $\left\|F(s,u)\right\|<1$ for all $|u|>N_0$, $s\in[a,b]$; Such an $N_0$ exists by assumption. Similarly, one obtains that
\begin{equation}
\label{eq-SigmaFNs}
\left\|\Sigma_{F,N}^s\right\|\leq\left\|\Sigma^{\mathcal{G}}\right\|\int_{|u|>N}\left\|F(s,u)\right\|^2\dd u\leq\left\|\Sigma^{\mathcal{G}}\right\|\int_{|u|>N}\left\|F(s,u)\right\|\dd u,\quad\forall N>N_0.
\end{equation}
and
\begin{align*}
\int_{\left\|\bx\right\|<1}{\left\|\bx\right\|^2\nu_{F,N}^s(\dd\bx)} =&\int_{|u|>N}\int_{\R^d}{I_{[0,1]}(\left\|F(s,u)\bx\right\|)\left\|F(s,u)\bx\right\|^2\nuL(\dd\bx)\dd u}\\
 \leq&\int_{|u|>N}\left\|F(s,u)\right\|\dd u\int_{\R^d}{\left\|\bx\right\|^2\nuL(\dd\bx)},\quad\forall N>N_0,\\
\int_{\left\|\bx\right\|\geq1}{\left\|\bx\right\|^r\nu_{F,N}^s(\dd\bx)} =&\int_{|u|>N}\int_{\R^d}{I_{[1,\infty)}(\left\|F(s,u)\bx\right\|)\left\|F(s,(u)\bx\right\|^r\nuL(\dd\bx)\dd u}\\
 \leq & \int_{|u|>N}\int_{\R^d}I_{[1,\infty]}\left(\left\|F\right\|_{L^\infty([a,b]\times\R)}\left\|\bx\right\|\right)\left\|F(s,u)\right\|^r\left\|\bx\right\|^r\nuL(\dd\bx)\dd u\\
  \leq&\int_{|u|>N}\left\|F(s,u)\right\|\dd u\int_{\left\|\bx\right\|\geq \max\{1,\left\|F\right\|_{L^\infty([a,b]\times \R)}^{-1}\}}\left\|\bx\right\|^2\nuL(\dd\bx),\quad r=1,2.
\end{align*}
Applying \cref{lemma-momentsidrv} with $k=2$ and using the assumed uniform absolute integrability of the family $\{u\mapsto F(s,u)\}_{s\in[a,b]}$ we can deduce that
\begin{equation*}
\sup_{s\in[a,b]}\E\left\|\int_{|u|>N} F(s,u)\dd\Lb(u)\right\|^2 \to 0,\quad \text{ as }N\to\infty.
\end{equation*}
Together with \cref{eq-DeltaN} and Jensen's inequality this implies that
\begin{align}
\label{eq-EDeltaN2}
\E\left\|\Delta_N\right\|^2 \leq& \E\left(\int_a^b\left\|\int_{|u|>N}F(s,u)\dd\Lb(u)\right\|\dd s\right)^2\notag\\
  \leq& \E\int_a^b\left\|\int_{|u|>N}F(s,u)\dd\Lb(u)\right\|^2\dd s\leq (b-a)\sup_{s\in[a,b]}\E\left\|\int_{|u|>N}F(s,u)\dd\Lb(u)\right\|^2\to 0,
\end{align}
as $N\to\infty$, showing that $\Delta_N$ converges to zero in $L^2$. In order to prove the same convergence also for
\begin{equation*}
\overleftrightarrow{\Delta_N} = \overleftrightarrow{I}-\overleftrightarrow{I_N} = \int_{|u|>N}\int_a^b{F(s,u)\dd s}\dd\Lb(du),
\end{equation*}
we first define the function $\widetilde F:\R\to M_{d,m}(\R)$ by $\widetilde F(u)=\int_a^bF(s,u)\dd s$. Since for all $u\in\R$, $\left\|\widetilde F(u)\right\|$ is smaller than $(b-a)\left\|F\right\|_{L^\infty([a,b]\times \R)}$, the function $\widetilde F$ is bounded. It is also integrable because the normal variant of Fubini's theorem and the assumed uniform integrability of $\{F(s,\cdot)\}_{s\in[a,b]}$ imply that
\begin{equation*}
\int_{|u|>N}\left\|\widetilde F(u)\right\|\dd u\leq \int_a^b\int_{|u|>N}\left\|F(s,u)\right\|\dd u\dd s\leq(b-a)\sup_{s\in[a,b]}\int_{|u|>N}\left\|F(s,u)\right\|\dd u \to 0,\quad N\to\infty.
\end{equation*}
Similar arguments to the ones given above then show that $\overleftrightarrow{\Delta_N}$ converges to zero in $L^2$ as well. It thus follows by the triangle inequality that, for every $N$ and every $\epsilon$,
\begin{align*}
\Pb\left(\left\|I-\overleftrightarrow{I}\right\|\geq \epsilon\right)\leq & \Pb\left(\left\{\left\|I-I_N\right\|\geq \frac{\epsilon}{2}\right\} \cup \left\{\left\|\overleftrightarrow{I}-I_N\right\|\geq \frac{\epsilon}{2}\right\}\right)\\
  \leq & \Pb\left(\left\{\left\|I-I_N\right\|\geq \frac{\epsilon}{2}\right\}\right) + \Pb\left(\left\{\left\|\overleftrightarrow{I}-\overleftrightarrow{I}_N\right\|\geq \frac{\epsilon}{2}\right\}\right),
\end{align*}
where we have used the subadditivity of $\Pb$ as well as the fact that $I_N$ is equal to $\overleftrightarrow{I}_N$ almost surely. Since $L^2$-convergence implies convergence in probability (\citep[Theorems 17.2]{jacod2003probability}), it follows that the right hand side of the last display is less than any positive $\delta$ if only $N$ is large enough and thus that the probability of the absolute difference between $I$ and $\overleftrightarrow{I}$ exceeding $\epsilon$ is equal to zero for every positive $\epsilon$. This means that $I$ equals $\overleftrightarrow{I}$ almost surely and completes the proof.
\end{proof}

\section{Controller canonical parametrization of multivariate CARMA processes}
\label{section-MCARMA}
Multivariate, continuous-time autoregressive moving average (abbreviated MCARMA) processes are the continuous-time analogue of the well known vector ARMA processes. They also generalize the much-studied univariate CARMA processes to a multidimensional setting. A $d$-dimensional MCARMA process $\Y$, specified by an autoregressive polynomial
\begin{equation}
\label{eq-ARpolynomial}
\tilde P(z)=z^{\tilde p}+\tilde A_1 z^{\tilde p-1}+\ldots+\tilde A_{\tilde p}\in M_d(\R[z]),
\end{equation}
a moving average polynomial
\begin{equation}
\label{eq-MApolynomial}
 \tilde Q(z)=\tilde B_0+\tilde B_1z+\ldots+\tilde B_{\tilde q}z^{\tilde q}\in M_{d,m}(\R[z]),
\end{equation}
and driven by an $m$-dimensional L\'evy process $\Lb$ is defined as a solution of the formal differential equation
\begin{equation}
\label{eq-MCARMA-ODE}
\tilde P(\DD)\Y(t)=\tilde Q(\DD)\DD\Lb(t),\quad \DD=\frac{\dd}{\dd t},\quad t\in\R,
\end{equation}
the continuous-time version of the well-known ARMA equations. \Cref{eq-MCARMA-ODE} is only formal because, in general, the paths of a L\'evy process are not differentiable. It has been shown in \citep[]{marquardt2007multivariate} that an MCARMA process $\Y$ can equivalently be defined by the continuous-time state space model
\begin{equation}
\label{eq-MCARMAssm}
\dd\X(t)=\tilde\A\X(t)\dd t+\beta \dd\Lb(t),\quad \Y(t)= C\X(t),\quad t\in\R,
\end{equation}
where the matrices $\tilde\A,\beta$ and $C$ are given by
\begin{subequations}
\label{eq-MCARMAcoeffABC}
\begin{align}
\label{eq-MCARMAcoeffA} \tilde\A =& \left(\begin{array}{ccccc}
       0 & \I_d & 0 & \ldots & 0 \\
	0 & 0 & \I_d & \ddots & \vdots \\
	\vdots && \ddots & \ddots & 0\\
	0 & \ldots & \ldots & 0 & \I_d\\
	-\tilde A_{\tilde p} & -\tilde A_{\tilde p-1} & \ldots & \ldots & -\tilde A_1
      \end{array}\right)\in M_{\tilde pd}(\R),\\
\label{eq-MCARMAcoeffB}\beta=&\left(\begin{array}{ccc}\beta_1^T & \cdots & \beta_{\tilde p}^T\end{array}\right)^T\in M_{\tilde pd,m}(\R),\quad\beta_{\tilde p-j} = -I_{\{0,\ldots,\tilde q\}}(j)\left[\sum_{i=1}^{\tilde p-j-1}{\tilde A_i\beta_{\tilde p-j-i}-\tilde B_j}\right]\quad\text{and}\\
\label{eq-MCARMAcoeffC}C=&\left(\I_d,0,\ldots,0\right)\in M_{d,\tilde pd}(\R).
\end{align}
\end{subequations}
This is but one of several possible parametrizations of the general continuous-time state space model and is in the discrete-time literature often referred to as the {\it observer canonical form} (\citep{kailath1980linear}). For the purpose of estimating the driving L\'evy process $\Lb$ it is more convenient to work with a different parametrization, which, in analogy to a canonical state space representation used in discrete-time control theory, might be called the {\it controller canonical form}. It is the multivariate generalization of the parametrization used for univariate CARMA processes in \citep{brockwell2010estimation}. We first state an auxiliary lemma which we could not find in the literature.
\begin{lemma}
\label{lemma-MatrixInverse}
Let $r,s$ be positive integers. Assume that $R(z)=z^r+M_1z^{r-1}+\ldots+M_r\in M_s(\R[z])$ is a matrix polynomial and denote by 
\begin{equation}
\M=\left[\begin{array}{ccccc}
        0 & \I_s & 0 & \cdots & 0\\
	0 & 0  & \I_s& \cdots & 0\\
	\vdots&\vdots&\vdots&\ddots&\vdots\\
	0 & 0 & 0 & \cdots & \I_s\\
	-M_r&-M_{r-1} & -M_{r-2}& \cdots & -M_1
       \end{array}\right]\in M_{rs}(\R)
\end{equation}
the associated multi-companion matrix. The rational matrix function
\begin{equation}
S(z)=[S_{ij}(z)]_{1\leq i,j\leq r} = \left(z\I_{rs}-\M\right)^{-1}\in M_{rs}(\R\{z\}),\quad S_{ij}(z)\in M_s(\R\{z\}),
\end{equation}
is then given by the following formula for the block $S_{ij}(z)$:
\begin{equation}
S_{ij}(z)=R(z)^{-1}\begin{cases}
           z^{r-1+i-j}\I_s+\sum_{k=1}^{r-j}M_kz^{r-1-k+i-j}, & j\geq i,\\
	   -\sum_{k=r-j+1}^r{M_kz^{r-1-k+i-j}},		& j<i.
          \end{cases}
\end{equation}
\end{lemma}
\begin{proof}
We compute the $(i,j)$th block of $S(z)\left(z\I_{rs}-\M\right)$. Assuming $i < j$, this block is given by
\begin{align*}
\left[S(z)\left(z\I_{rs}-\M\right)\right]_{ij} =& \sum_{k=1}^r{S_{ik}(z)\left(z\I_{rs}-\M\right)_{kj}}\\
  						   =& zS_{ij}(z)-S_{i,j-1}(z)+S_{ir}(z)M_{r-j+1}\\
						   =& R(z)^{-1}\left[z^{r+i-j}\I_s+\sum_{k=1}^{r-j}M_kz^{r-k+i-j}-z^{r+i-j}\I_s-\sum_{k=1}^{r-j+1}M_kz^{r-k+i-j}+z^{i-1}M_{r-j+1}\right]=0.
\end{align*}
A similar calculation shows that for $i>j$, $\left[S(z)\left(z\I_{rs}-\M\right)\right]_{ij}=0$. For the blocks on the diagonal we obtain for $i\geq 2$,
\begin{align*}
\left[S(z)\left(z\I_{rs}-\M\right)\right]_{ii} =& \sum_{k=1}^r{S_{ik}(z)\left(z\I_{rs}-\M\right)_{ki}}\\
						   =& zS_{ii}(z)-S_{i,i-1}(z)+S_{ir}(z)M_{r-i+1}\\
						   =&R(z)^{-1}\left[z^r\I_s+\sum_{k=1}^{r-i}M_kz^{r-k} +\sum_{k=r-i+2}^r{M_kz^{r-k}} +z^{i-1}M_{r-i+1}\right]=\I_s,
\end{align*}
and finally
\begin{align*}
\left[S(z)\left(z\I_{rs}-\M\right)\right]_{11} = \sum_{k=1}^r{S_{1k}(z)\left(z\I_{rs}-\M\right)_{k1}} =& zS_{11}(z)+S_{1r}(z)M_r\\
 =& R(z)^{-1}\left[z^r\I_s+\sum_{k=1}^{r-1}M_kz^{r-k}+M_r\right]=\I_s.
\end{align*}
This shows that $S(z)$ is the inverse of $z\I_{rs}-\M$ and completes the proof.
\end{proof}
\begin{theorem}[Controller canonical state space representation]
\label{theorem-AlternativeSSR}
Assume that $\Lb$ is an $m$-dimensional L\'evy process and that $\Y$ is a $d$-dimensional $\Lb$-driven MCARMA process with autoregressive polynomial $\tilde P\in M_d(\R[z])$ and moving average polynomial $\tilde Q\in M_{d,m}(\R[z])$. Then there exist integers $p>q>0$ and matrix polynomials 
\begin{subequations}
\label{eq-armapolymod}
\begin{align}
z\mapsto P(z) =& z^p+A_1 z^{p-1}+\ldots+A_p\in M_m(\R[z]),\\
z\mapsto Q(z) =& B_0+B_1z+\ldots+B_qz^q\in M_{d,m}(\R[z])
\end{align}
\end{subequations}
satisfying $\tilde P(z)^{-1}\tilde Q(z)=Q(z)P(z)^{-1}$ for all $z\in\C$ and $\det P(z)=0$ if and only if $\det \tilde P(z)=0$. Moreover, the process $\Y$ has the state space representation
\begin{subequations}
\label[pluralequation]{eq-statespacerepmod}
\begin{align}
\label{eq-statespacerepmodstate}\dd\X(t)=&\A\X(t)\dd t+E_p\dd\Lb(t),\\
\label{eq-statespacerepmodobservation}\Y(t)=&\underline B\X(t),
\end{align}
\end{subequations}
where
\begin{subequations}
\label[pluralequation]{eq-SSRmatrices}
\begin{align}
\A =& \left[\begin{array}{ccccc}
        0 & \I_m & 0 & \cdots & 0\\
	0 & 0  & \I_m& \cdots & 0\\
	\vdots&\vdots&\vdots&\ddots&\vdots\\
	0 & 0 & 0 & \cdots & \I_m\\
	-A_p&-A_{p-1} & -A_{p-2}& \cdots & -A_1
       \end{array}\right]\in M_{p m}(\R),\quad E_p = \left[\begin{array}{c}0\\0\\\vdots\\0\\\I_m\end{array}\right]\in M_{p m,m}(\R),\\
\underline B =& \left[\begin{array}{cccc}B_0&B_1&\cdots& B_{p-1}\end{array}\right]\in M_{d,p m}(\R),\quad B_j=0_{d,m},\quad q+1\leq j\leq p-1.
\end{align}
\end{subequations}
\end{theorem}
\begin{proof}
The existence of matrix polynomials $P\in M_m(\R[z])$ and $Q\in M_{m,d}(\R[z])$ with the asserted properties has been shown in \citep[Lemma 6.3-8]{kailath1980linear}. In order to prove \cref{eq-statespacerepmod} it suffices, by \cite[Theorem 1]{schlemmmixing2010}, to prove that the triple $(\A,E_p,\underline B)$, defined in \cref{eq-SSRmatrices}, is a realization of the right matrix fraction $QP^{-1}$, that is
\begin{equation*}
{\underline B}\left[z\I_{ p m}-{\A}\right]^{-1}E_{ p} =  Q(z) P(z)^{-1},\quad\forall z\in \C.
\end{equation*}
Using \cref{lemma-MatrixInverse} and the fact that right multiplication by $E_{ p}$ selects the last block-column one sees that
\begin{equation*}
\left[z\I_{ p m}-{\A}\right]^{-1}E_{ p} = \left[\begin{array}{cccc}1&z&\cdots&z^{  p-1}\end{array}\right]^T\otimes  P(z)^{-1},
\end{equation*}
where $\otimes$ denotes the Kronecker product of two matrices. By definition it holds that
\begin{equation*}
{\underline B}\left[\begin{array}{cccc}1&z&\cdots&z^{  p-1}\end{array}\right]^T =  B_0 +  B_1z +  \ldots +  B_{ q}z^{ q} =  Q(z),
\end{equation*}
and so the claim follows.
\end{proof}
In view of \cref{theorem-AlternativeSSR} one can assume without loss of generality that an MCARMA process $\Y$ is given by a state space representation \labelcref{eq-statespacerepmod} with coefficient matrices of the form \labelcref{eq-SSRmatrices}. We make the following assumptions about the zeros of the polynomials $P$, $Q$ in equations \labelcref{eq-armapolymod}. The first one is a stability assumption guaranteeing the existence of a stationary solution of the state equation \labelcref{eq-statespacerepmodstate}.
\begin{assumption}
\label{assum-eigA}
The zeros of the polynomial $\det P(z)\in \R[z]$ have strictly negative real parts.
\end{assumption}
The second assumption corresponds to the minimum-phase assumption in classical time series analysis. For a matrix $M\in M_{d,m}(\R)$, any matrix $M^{\sim 1}$ satisfying $M^{\sim1}M=\I_m$ is called a {\it left inverse} of $M$. It is easy to check that the existence of a left inverse of $M$ is equivalent to the conditions $m\leq d$, $\rank M=m$, and that in this case $M^{\sim1}$ can be computed as $M^{\sim1}=(M^TM)^{-1}M^T$.
\begin{assumption}
\label{assum-eigB}
The dimension $m$ of the driving L\'evy process $\Lb$ is smaller than or equal to the dimension of the multivariate CARMA process $\Y$, and both $B_q$ and $B_q^TB_0$ have full rank $m$. The zeros of the polynomial $\det B_q^{\sim1}Q(z)\in\R[z]$ have strictly negative real parts.
\end{assumption}
It is well known that every solution of \cref{eq-statespacerepmodstate} satisfies
\begin{equation*}
\X(t) = \ee^{\A(t-s)}\X(s) + \int_s^t \ee^{\A(t-u)}E_p\dd\Lb(u),\quad s,t\in\R,\quad s<t.
\end{equation*}
Under \cref{assum-eigA}, the state equation \labelcref{eq-statespacerepmodstate} has a unique strictly stationary, causal solution given by
\begin{equation}
\label{eq-Xmovingaverage}
\X(t) = \int_{-\infty}^t \ee^{\A(t-u)}E_p\dd\Lb(u),\quad t\in\R.
\end{equation}
and consequently, the multivariate CARMA process $\Y$ has the moving-average representation
\begin{equation}
\label{eq-Ymovingaverage}
\Y(t) = \int_{-\infty}^\infty g(t-u)\dd\Lb(u),\quad t\in\R;\quad g(t) = \underline{B}\ee^{\A t}E_pI_{[0,\infty]}(t).
\end{equation}

We recall that we denote by $X^i(t)$ the $i$th component of the vector $\X(t)$ and define, for $j=1,\ldots,p$, the $j$th {\it $m$-block} of $\X$ by the formula
\begin{equation}
\label{eq-Xmblock}
\X^{(j)}(t)=\left[\begin{array}{ccc}{X^{(j-1)m+1}(t)}^T & \cdots & {X^{jm}(t)}^T \end{array}\right]^T,\quad t\in\R.
\end{equation}
A  very useful property, which the sequence of approximation errors $\left(\Delta\Lb_n-\widehat{\Delta\Lb}(n)\right)_{n\in\N}$ might enjoy, is asymptotic independence; heuristically this means that $\Delta\Lb_n-\widehat{\Delta\Lb}(n)$  and $\Delta\Lb_m-\widehat{\Delta\Lb}(m)$ are almost independent if $|n-m|\gg 1$. One possibility of making this concept precise is to introduce the notion of strong (or $\alpha$-) mixing, which has first been defined in \citep{rosenblatt1956central}. Since then it has turned out to be a very powerful tool for establishing asymptotic results in the theory of inference for stochastic processes. For a stationary stochastic process $X=(X_t)_{t\in I}$, where $I$ is either $\R$ or $\Z$, we first introduce the $\sigma$-algebras $\mathscr{F}_n^m=\sigma(X_j:j\in I,n<j<m)$, where $-\infty\leq n<m\leq \infty$. For $m\in I$, the strong mixing coefficient $\alpha(m)$ is defined as
\begin{equation}
\label{eq-Defalpha}
\alpha(m) = \sup_{A\in\mathscr{F}_{-\infty}^0,B\in\mathscr{F}_m^\infty}\left|\Pb(A\cap B)-\Pb(A)\Pb(B)\right|.
\end{equation}
The process $X$ is called {\it strongly mixing} if $\lim_{m\to\infty}\alpha(m)=0$; if $\alpha(m)=O(\lambda^m)$ for some $0<\lambda<1$ it is called {\it exponentially strongly mixing}.


\section{Recovery of the driving L\'evy process from continuous-time observations}
\label{section-recovery}

In this section we address the problem of recovering the driving L\'evy process of a multivariate CARMA process given by a state space representation \labelcref{eq-statespacerepmod}, if continuous-time observations are available. We assume that the order $(p,q)$ as well as the coefficient matrices $\A$ and $\underline{B}$ are known. If they are not they can first be estimated by, e.g. maximization of the Gaussian likelihood \citep{schlemm2011quasi}, although the precise statistical properties of this two-step estimator are beyond the scope of the present paper. More precisely, we show that, conditional on the value $\X(0)$ of the state vector at time zero, one can write the value of $\Lb(t)$, for any $t\in[0,T]$, as a function of the continuous-time record $\left(\Y(t):0\leq t\leq T\right)$. In particular one can obtain an i.i.d$.$ sample from the distribution of the unit increments $\Lb(n)-\Lb(n-1)$, $1\leq n\leq T$, which, when subjected to one of several well-established estimation procedures, can be used 
to estimate a parametric model for $\Lb$. It can be argued that most of the time a continuous record of observations is not available. The results of this section will, however, serve as the starting point for the recovery of an approximate sample from the unit increment distribution based on discrete-time observation of $\Y$, which is presented in \cref{section-DTestimation}.

The strategy is to first express the state vector $\X$ in terms of the observations $\Y$ and then to invert the state equation \labelcref{eq-statespacerepmodstate} to obtain the driving L\'evy process as a function of the state vector. We first define the {\it upper $q$-block-truncation} of $\X$, denoted by $\X_q$, by
\begin{equation*}
\X_q(t)=\left[\begin{array}{ccc}{\X^{(1)}(t)}^T&\cdots&{\X^{(q)}(t)}^T\end{array}\right]^T,\quad t\in \R,
\end{equation*}
where the $m$-blocks $\X^{(j)}$ have been defined in \cref{eq-Xmblock}. 
\begin{lemma}
\label{lemma-OUXq}
Assume that $\Lb$ is a L\'evy process and that $\Y$ is a multivariate CARMA process given as a solution of the state space equations \labelcref{eq-statespacerepmod}. If \cref{assum-eigB} holds, the truncated state vector $\X_q$ satisfies the stochastic differential equation
\begin{equation}
\label{eq-OUXq}
\dd\X_q(t)=\B\X_q(t)\dd t+E_q\Y(t)\dd t,
\end{equation}
where
\begin{equation}
\label{eq-OUXqcoeff}
\B = \left[\begin{array}{ccccc}
                0 & \I_m & 0 & \cdots & 0\\
		0 & 0  & \I_m&  & 0\\
		\vdots&\vdots&&\ddots&\vdots\\
		0 & 0 & 0 &  & \I_m\\
		-B_q^{\sim 1}B_0&-B_q^{\sim 1}B_1 & -B_q^{\sim 1}B_2& \cdots & -B_q^{\sim 1}B_{q-1} 
               \end{array}\right]\in M_{mq}(\R),\quad E_q = \left[\begin{array}{c}0\\0\\\vdots\\0\\ B_q^{\sim 1}\end{array}\right]\in M_{mq,d}(\R),
\end{equation}
and $B_q^{\sim1}$ denotes the left inverse of $B_q$. Moreover, the eigenvalues of the matrix $\B$ have strictly negative real parts.
\end{lemma}
\begin{proof}
\Cref{eq-OUXq} follows easily from combining the first $q$ block-rows of the state transition equation \labelcref{eq-statespacerepmodstate} with the observation equation \labelcref{eq-statespacerepmodobservation}. The assertion about the eigenvalues of $\B$ is a consequence of the well-known correspondence between the eigenvalues of a multi-companion matrix and the zeros of the associated polynomial, see, e.g., \citep[Lemma 3.8]{marquardt2007multivariate}. By this correspondence, the eigenvalues of $\B$ are exactly the zeros of the polynomial $\det\left(\I_mz^q+B_q^{\sim1}B_{q-1}z^{q-1}+\ldots+B_q^{\sim1}B_0\right)$, whose zeros have strictly negative real parts by \cref{assum-eigB}.
\end{proof}
As before we see that \cref{eq-OUXq} is readily integrated to
\begin{equation}
\label{eq-OUXqintegrated}
\X_q(t)=\ee^{\B (t-s)}\X_q(s)+\int_s^t{\ee^{\B(t-u)}E_q\Y(u)\dd u},\quad s,t\in\R,\quad s<t.
\end{equation}
The remaining blocks $\X^{(i)}$, $q<i\leq p$, are obtained from $\X_q$ and $\Y$ by differentiation. The existence of the occurring derivatives of the state process $\X$ and the MCARMA process $\Y$ is guaranteed by \cref{lemma-propMCARMA} in the appendix.
\begin{lemma}
For $1\leq n\leq p-q$, the block $\X^{(q+n)}$ is given by
\begin{equation}
\label{eq-Xqblock}
\X^{(q+n)}(t)=E_q^T\left[\B^n\X_q(t)+\sum_{\nu=0}^{n-1}{\B^{n-1-\nu}E_q\DD^{\nu}\Y(t)}\right],\quad t\in\R.
\end{equation}
\end{lemma}
\begin{proof}
We first observe that \cref{eq-statespacerepmod,eq-OUXq} imply that
\begin{equation*}
\label{eq-DXq}
\X^{(q+n)}(t)=\DD\X^{(q+n-1)}(t),\qquad \DD\X_q(t)=\B\X_q(t)+E_q\Y(t).
\end{equation*}
Therefore the claim is true for $n=1$. Assuming it is true for some $1<n<p-q$ it follows that
\begin{align*}
\X^{(q+n+1)}(t)=&\DD\X^{(q+n)}(t)\\
	       =&DE_q^T\left[\B^n\X_q(t)+\sum_{\nu=0}^{n-1}{\B^{n-1-\nu}E_q\DD^{\nu}\Y(t)}\right]\\
	       =&E_q^T\left[\B^{n+1}\X_q(t)+\B^nE_q\Y(t)+\sum_{\nu=0}^{n-1}{\B^{n-1-\nu}E_q\DD^{\nu+1}\Y(t)}\right]\\
	       =&E_q^T\left[\B^{n+1}\X_q(t)+\sum_{\nu=0}^n{\B^{n-\nu}E_q\DD^{\nu}\Y(t)}\right].\qedhere
\end{align*}
\end{proof}
\Cref{eq-OUXqintegrated,eq-Xqblock} allow to compute the value of $\X(t)$ based on the knowledge of the initial value $\X(0)$ and $\{\Y(s):0\leq s\leq t\}$. In order to obtain the value of $\Lb(t)$ we integrate the last block-row of the state transition equation \labelcref{eq-statespacerepmodstate} to obtain
\begin{equation}
\label{eq-recoveryL}
\Lb(t)=\X^{(p)}(t)-\X^{(p)}(0)+\underline{A}\int_0^t\X(s)\dd s,
\end{equation}
where $\underline{A}=\left[\begin{array}{ccc}A_p& \ldots & A_1\end{array}\right]$. We also write $\underline{A}_q=\left[\begin{array}{ccc}A_p& \ldots & A_{p-q+1}\end{array}\right]$.
\begin{theorem}
\label{theorem-recoveryDeltaL}
Let $\Y$ be the multivariate CARMA process defined by the state space representation \labelcref{eq-statespacerepmod} and assume that \cref{assum-eigB} holds. The increment $\Delta\Lb_n=\Lb(n)-\Lb(n-1)$ is then given by
\begin{align}
\label{eq-recoveryDeltaL}
\Delta\Lb_n =&\sum_{\nu=0}^{p-q-1}{\left[E_q^T\B^{p-q-1-\nu}E_q + \sum_{k=\nu}^{p-q-2}{ A_{p-q-k-1}E_q^T\B^{k-\nu}E_q}\right]\left[\DD^{\nu}\Y(n)-\DD^{\nu}\Y(n-1)\right]}\notag\\
	      &+\left[\underline{A}_q\B^{-1}+\sum_{k=1}^{p-q}{ A_{p-q-k+1}E_q^T\B^{k-1}}+E_q^T\B^{p-q}\right]\left[\X_q(n)-\X_q(n-1)\right]\notag\\
  &+A_p\left[B_q^{\sim 1}B_0\right]^{-1}B_q^{\sim 1}\int_{n-1}^n{\Y(s)\dd s}
\end{align}
and
\begin{equation}
\label{eq-Xqnrecursion}
\X_q(n)=\ee^{\B}\X_q(n-1)+\int_{n-1}^n{\ee^{\B(n-u)}E_q\Y(u)\dd u},\quad n\geq 1.
\end{equation}
\end{theorem}
\begin{proof}
Substituting \cref{eq-Xqblock} into \cref{eq-recoveryL} leads to
\begin{align*}
\Delta\Lb_n =& \sum_{\nu=0}^{p-q-1}{\left[E_q^T\B^{p-q-1-\nu} + \sum_{k=\nu}^{p-q-2}{ A_{p-q-k-1}E_q^T\B^{k-\nu}}\right]E_q\left[\DD^{\nu}\Y(n)-\DD^{\nu}\Y(n-1)\right]}\\
  &+E_q^T\B^{p-q}\left[\X_q(n)-\X_q(n-1)\right]  +   \left[\underline{A}_q+\sum_{k=1}^{p-q}{ A_{p-q-k+1}E_q^T\B^k}\right]\int_{n-1}^n{\X_q(s)\dd s}\\
  &+\sum_{k=1}^{p-q}{ A_{p-q-k+1}E_q^T\B^{k-1}E_q\int_{n-1}^n{\Y(s)\dd s}}.
\end{align*}
\Cref{assum-eigB} implies that $B_q^{\sim 1}B_0$ is invertible and, by \cref{lemma-MatrixInverse}, the matrix $\B$ is invertible as well. Thus, integration of \cref{eq-OUXq} shows that
\begin{equation*}
\int_{n-1}^n{\X_q(s)\dd s} = \B^{-1}\left[\X_q(n)-\X_q(n-1)-E_q\int_{n-1}^n{\Y(s)\dd s}\right].
\end{equation*}
Plugging this into the last expression for $\Delta\Lb_n$ and using the equality $\underline A_q\B^{-1}E_q=A_p\left[B_q^{\sim 1}B_0\right]^{-1}B_q^{\sim 1}$ proves \cref{eq-recoveryDeltaL}. \Cref{eq-Xqnrecursion} follows from setting $t=n$, $s=n-1$ in \cref{eq-OUXqintegrated}.
\end{proof}
In order to keep the notation simple we restrict our attention to unit increments $\Delta\Lb$. In all our arguments and results, $\Delta\Lb_n$ can be replaced by $\Delta_\delta\Lb_n\coloneqq\Lb(n\delta)-\Lb((n-1)\delta)$ for some $\delta>0$.
\section{Approximate recovery of the driving L\'evy process from discrete-time observations}
\label{section-DTestimation}
In this section we consider the question of how to obtain estimates of the increments $\Delta\Lb_n$ of the driving L\'evy process based on a discrete-time record of the multivariate CARMA process $\Y$. The starting point is \cref{eq-recoveryDeltaL} which expresses the increment $\Delta\Lb_n$ in terms of derivatives and integrals of $\Y$. In order to approximate $\Delta\Lb_n$ by a function $\widehat{\Delta\Lb}_n^{(h)}$ of the discrete-time record, it is therefore necessary to approximate these derivatives and integrals. For this purpose we will employ forward differences (\cref{eq-forwarddifference}) and the trapezoidal rule of numerical integration (\cref{eq-DefTrapez}). We always assume that values of $\Y$ are available at the discrete times $(0,h,2h,\ldots,T)$ only. For notational convenience we also assume that $h^{-1}\in\N$; our results continue to hold if this restriction is dropped.

Our main result in this section is \cref{theorem-Levyapprox}. It states that the moments of the approximation error $\widehat{\Delta\Lb}_n^{(h)}-\Delta\Lb_n$ are of order $h^{1/2}$, and thus converge to zero as the sampling frequency $h^{-1}$ tends to infinity. Before we can prove this result we need to give a quantitative account of the approximation theory of derivatives and integrals of MCARMA processes; this is achieved in \cref{subsection-derivatives,subsection-integrals}, respectively.
\subsection{Approximation of derivatives}
\label{subsection-derivatives}
Throughout we will approximate derivatives by so-called forward differences which can be interpreted as iterated difference quotients. For a general introduction to finite difference approximations, see \citep[Chapter 1]{leveque2007}. For any function $f$ and any positive integer $\nu$ we define
\begin{equation}
\label{eq-forwarddifference}
\Delta_h^\nu[f](t)\coloneqq\frac{1}{h^\nu}\sum_{i=0}^\nu{(-1)^{\nu-i}\binom{\nu}{i}f(t+ih)}.
\end{equation}
It is apparent from this formula that knowledge of $f$ on the discrete time grid $(0,h,\ldots,T)$ is sufficient to compute $\Delta_h^\nu[f](t)$ for any $t\in[0,T-\nu h]\cap h\Z$. We will consider the differentiation of integrals of functions, for which we introduce the notations
\begin{equation}
\label{eq-intf}
I_{f}(t) \coloneqq \int_0^t{f(s)\dd s},\quad\text{ as well as }\quad \be_{I_f,n}^{(h)}\coloneqq \Delta_h^1\left[I_f\right](n)-f(n)
\end{equation}
for the corresponding approximation error. In the next lemma we analyse this approximation for the case when $f$ is a L\'evy process.
\begin{lemma}
\label{lemma-eLstochcont}
The sequence of approximation errors $\be_{I_{\Lb}}^{(h)}$ is i.i.d$.$ Moreover, for every $\omega\in\Omega$ and for every integer $n$ the approximation error $\be_{I_{\Lb},n}^{(h)}$ converges to zero as $h\to 0$. If, for some positive integer $k$, the absolute moment $\E\left\|\Lb(1)\right\|^{(k)_0}$ is finite, then $\E\left\|\be_{I_{\Lb},n}^{(h)}\right\|^k=O(h^{k/(k)_0})$, as $h\to 0$, where the constant implicit in the $O(\cdot)$ notation does not depend on $n$.
\end{lemma}
\begin{proof}
We first observe that
\begin{align*}
\left\|I_{\Lb}(n+h)-I_{\Lb}(n)-h\Lb^{\mathscr{n}}(n)\right\|=&\left\|\int_n^{n+h}\left[\Lb(s)-\Lb(n)\right]\dd s\right\|\leq \int_n^{n+h}\left\|\Lb(s)-\Lb(n)\right\|\dd s.
\end{align*}
The right continuity of $t\mapsto \Lb(t)$ implies that for every integer $n$ and each $\epsilon>0$ there exists a $\delta_{\epsilon,n}$ such that $\left\|\Lb(n+t)-\Lb(n)\right\|\leq \epsilon$, for all $0\leq t\leq \delta_{\epsilon,n}$. This means that $\left\|I_{\Lb}(n+h)-I_{\Lb}(n)-h\Lb(t)\right\|\leq h\epsilon$, provided $h\leq \delta_{\epsilon,n}$. Dividing by $h$ thus proves $\be_{I_{\Lb},n}^{(h)}\to 0$. The proof also shows that $\be_{I_{\Lb},n}^{(h)}$ is a deterministic function of the increments $\left\{\Lb(s)-\Lb(n),n\leq s\leq n+h\right\}$. Since the increments of a L\'evy process are stationary and independent, this implies that $\be_{I_{\Lb}}^{(h)}$ is an i.i.d$.$ sequence.

For the second claim about the size of the absolute moments of $\be_{I_{\Lb},n}^{(h)}$ for small $h$ it is no restriction to assume that $n=0$. Successive application of the triangle inequality and H\"older's inequality with the dual exponent $k' $ determined by $1/k+1/k'=1$ shows that
\begin{align*}
\E\left\|\be_{I_{\Lb},0}^{(h)}\right\|^k = \frac{1}{h^k}\E\left\|\int_0^h{\Lb(s)\dd s}\right\|^k \leq&\frac{1}{h^k}\E\left(\int_0^h{\left\|\Lb(s)\right\|\dd s}\right)^k \leq\frac{1}{h^k}\E\left(\left(\int_0^h{\left\|\Lb(s)\right\|^k\dd s}\right)^{1/k}\left(\int_0^h{1\dd s}\right)^{1/k'}\right)^k.
\end{align*}
Using $k/k'=k-1$ it follows that
\begin{equation*}
\E\left\|\be_{I_{\Lb},0}^{(h)}\right\|^k\leq\frac{1}{h}\E\int_0^h{\left\|\Lb(s)\right\|^k\dd s}. 
\end{equation*}
Since $\left\|\Lb(s)\right\|^k$ is positive we can interchange the expectation and integral. By \cref{prop-levymoments}, $\E\left\|\Lb(s)\right\|^k$ is of order $O(s^{k/(k)_0})$ which implies that $\left\|\be_{I_{\Lb},0}^{(h)}\right\|^k=O(h^{k/(k)_0})$.
\end{proof}
\Cref{lemma-eLstochcont} was dedicated to the analysis of the error of approximating the first derivative of the integral of a L\'evy process. We will also need analogous results for higher order derivatives of iterated integrals of L\'evy processes. The proofs are similar in spirit and only technically more complicated. For a positive integer $\nu$ we generalize the notations \labelcref{eq-intf} to
\begin{equation}
\label{eq-intfnu}
I^\nu_f(t)=\int_0^t{I^{\nu-1}_f(s)\dd s},\quad I^1_f(t)=\int_0^t{f(s)\dd s},\quad\text{and}\quad\be_{I_f^\nu,n}^{\nu,(h)}\coloneqq \Delta_h^\nu\left[I^\nu_f\right](n)-f(n).
\end{equation}
Clearly, if the function $f$ has only countably many jump discontinuities then $\DD^\nu I^\nu[f](t)=f(t)$ almost everywhere. 
\begin{lemma}
\label{lemma-integrallevyapprox}
For every positive integer $\nu\geq 1$ and every integer $n$, the error $\be_{I_{\Lb}^\nu,n}^{\nu,(h)}$ converges to zero as $h\to 0$. If, moreover, $\E\left\|\Lb(1)\right\|^{(k)_0}$ is finite for some $k>0$, then $\E\left\|\be_{I_{\Lb}^\nu,n}^{\nu,(h)}\right\|^k=O(h^{k/(k)_0)})$ as $h\to 0$.
\end{lemma}
\begin{proof}
Deferred to the appendix.
\end{proof}

With these auxiliary results finished, we turn to approximating derivatives of the multivariate CARMA process $\Y$. This is the first big step towards discretizing \cref{eq-recoveryDeltaL}.
\begin{proposition}
\label{prop-derivMCARMA}
Let $\Y$ be an $\Lb$-driven multivariate CARMA process satisfying \cref{assum-eigA}, let $n\geq 0$ be an integer and denote by $\be_{\Y,n}^{\nu,(h)} = \Delta_h^\nu[\Y](n) - \DD^\nu\Y(n)$ the error of approximating the $\nu$th derivative of $\Y$ by the forward differences defined in \cref{eq-forwarddifference}. Assume that, for some $k>0$, $\E\left\|\Lb(1)\right\|^{(k)_0}<\infty$. It then holds that:
\begin{enumerate}[i)]
 \item\label{prop-derivMCARMA-moments} If $1\leq\nu\leq p-q-2$, then $\E\left\|\be_{\Y,n}^{\nu,(h)}\right\|^k=O(h^k)$. If $\nu=p-q-1$, then $\E\left\|\be_{\Y,n}^{\nu,(h)}\right\|^k=O(h^{k/(k)_0})$.
 \item\label{prop-derivMCARMA-mixing} The sequence $\be_{\Y}^{\nu,(h)}$ is strictly stationary and strongly mixing with exponentially decaying mixing coefficients.
\end{enumerate}
\end{proposition}
\begin{proof}
We first prove the assertions \labelcref{prop-derivMCARMA-moments} about the behaviour of the absolute moments of $\be_{\Y,n}^{\nu,(h)}$ for small values of $h$. If $1\leq\nu\leq p-q-2$ it follows from \cref{lemma-propMCARMA} that the paths of $\Y$ are at least $\nu+1$ times differentiable; therefore, \cref{lemma-propforwarddifference} implies that $\left\|\be_{\Y,n}^{\nu,(h)}\right\|\leq h\sup_{n\leq s\leq n+\nu h}\left\|\DD^{\nu+1}\Y(s)\right\|$. To prove the claim it is thus sufficient to show that $\E \sup_{n\leq s\leq n+\nu h}\left\|\DD^{\nu+1}\Y(s)\right\|^k<\infty$. By the defining observation equation \labelcref{eq-statespacerepmodobservation}, $\Y$ is  a linear combination of the first $q+1$ $m$-blocks of the state process $\X$; the state equation \labelcref{eq-statespacerepmodstate} implies $\DD\X^{i}=\X^{i+1}$, $i=1,\ldots,p-1$, and since $\nu$ is assumed to be no bigger than $p-q-2$ it follows that $\DD^{\nu+1}\Y$ is a linear combination of the first $p-1$ $m$-blocks of $\X$, say $\DD^{\nu+1}\Y=\
Lambda\X$, for some matrix $\Lambda\in M_{d,pm}(\R)$. We can then apply \cref{lemma-momentssup} to estimate
\begin{align*}
\E \sup_{n\leq s\leq n+\nu h}\left\|\DD^{\nu+1}\Y(s)\right\|^k \leq \left\|\Lambda\right\|^k\E \sup_{n\leq s\leq n+\nu h}\left\|\X(s)\right\|^k<\infty,
\end{align*}
which proves the first claim. If $\nu=p-q-1$ we start again from the observation that $\Y$ is a linear combination of the first $q+1$ $m$-blocks of $\X$, namely,
\begin{equation*}
\Y(t) = \underline{B}_q\X_q(t)+B_q\X^{(q+1)}(t),\quad t\in\R,\quad \underline B_q = \left[\begin{array}{ccc}B_0 & \cdots & B_{q-1}\end{array}\right].
\end{equation*}
By solving the last $p-q+1$ block-rows of the state equation \labelcref{eq-statespacerepmodstate} one can express $\X^{(q+1)}$ as
\begin{equation*}
\X^{(q+1)}(t) = \frac{t^{p-q-1}}{(p-q-1)!}\X^{(p)}(0)-\underline{A}I^{p-q}_{\X}(t)+I^{p-q-1}_{\Lb}(t),
\end{equation*}
where the notation $I_f^\nu$ for the $\nu$-fold iterated integral of a function $f$ has been introduced in \cref{eq-intfnu}. By linearity and the fact that $\Delta_h^\nu[\mathfrak{p}]-\DD^\nu\mathfrak{p}=0$ for polynomials $\mathfrak{p}$ of degree $\nu$ (\cref{lemma-propforwarddifference},\labelcref{lemma-propforwarddifference-smooth}), it follows that
\begin{align*}
\be_{\Y,n}^{p-q-1,(h)} =& \Delta_h^{p-q-1}[\Y](n) - \DD^{p-q-1}\Y(n) \\\
= & \underline{B}_q\left[\Delta_h^{p-q-1}[\X_q](n)-\DD^{p-q-1}\X_q(n)\right]\\
	&- B_q\underline{A}\left[\Delta_h^{p-q-1}\left[I^{p-q}_{\X}\right](n)-\DD^{p-q-1}I^{p-q}_{\X}(n)\right]\\
	&+ B_q\left[ \Delta_h^{p-q-1}\left[I^{p-q-1}_{\Lb}\right](n)-\DD^{p-q-1}I^{p-q-1}_{\Lb}(n)\right].
\end{align*}
Both $\X_q$ (by \cref{lemma-propMCARMA}) and $I^{p-q}_{\X}$ are $p-q$ times differentiable so we can apply \cref{lemma-propforwarddifference},\labelcref{lemma-propforwarddifference-nonsmooth} to bound the differences in the first two lines of the last display by $h$ times the supremum of the $(p-q)$th derivative of $\X_q$ and $I^{p-q}_{\X}$, respectively. The contribution from the last line is the approximation error for the $(p-q-1)$th derivative of the $(p-q-1)$-fold iterated integral of the L\'evy process $\Lb$ which has been investigated in \cref{lemma-integrallevyapprox}. We thus obtain that
\begin{equation*}
\left\|\be_{\Y,n}^{p-q-1,(h)}\right\|\leq h\left[\left\|\underline{B}_q\right\|\sup_{n\leq t\leq n+(p-q-1)h}\left\|\DD^{p-q}\X_q(t)\right\|+\left\|B_q\right\|\left\|\underline{A}\right\|\sup_{n\leq t\leq n+(p-q-1)h}\left\|\X(t)\right\|\right]+\left\|B_q\right\|\left\|\be_{I_{\Lb}^{p-q-1},n}^{p-q-1,(h)}\right\|.
\end{equation*}
As before, one shows that the first term has finite $k$th moments which is of order $O(h^k)$. The second term has been shown in \cref{lemma-integrallevyapprox} to have finite $k$th moment of order $O(h^{k/(k)_0})$ which dominates the first term for $h<1$; this completes the proof of \labelcref{prop-derivMCARMA-moments}.

In order to prove that the sequence $\be_{\Y}^{\nu,(h)}$ is strongly mixing, it is enough, by virtue of \cref{lemma-propMCARMA},\labelcref{lemma-propMCARMA-mixing} and \cref{lemma-mixingfunctional}, to show that the approximation error $\be_{\Y,n}^{\nu,(h)}$ is measurable with respect to $\mathscr{Y}_n^{n+\nu h}$, the $\sigma$-algebra generated by $\left\{\Y(t):n\leq t\leq \nu h\right\}$. Clearly, $\Delta_h^\nu[\Y](t)$ is measurable with respect to the $\sigma$-algebra generated by $\{\Y_t,\Y_{t+h},\ldots, \Y_{t+\nu h}\}$. By the definition of derivatives as the limit of different quotients and the assumed differentiability of $t\mapsto \Y(t)$, the derivative $D_t^\nu \Y_t$ is the $\omega$-wise limit, as $s$ goes to zero, of the functions $\omega\mapsto \Delta_s^\nu[\Y(\omega)](t)$. Each of these functions is measurable with respect to $\sigma(\Y_{t},\Y_{t+s},\ldots,\Y_{t+\nu s})$, and therefore in particular with respect to the larger $\sigma$-algebra $\mathscr{Y}_n^{n+\nu h}$. Since pointwise limits of 
measurable functions are measurable (\citep[Theorem 1.92]{klenke2008probability}), the claim follows.

The claim that the sequence $\be_{\Y}^{\nu,(h)}$ is strictly stationary is a consequence of the fact that the multivariate CARMA process $\Y$ is strictly stationary (\cref{lemma-propMCARMA},\labelcref{lemma-propMCARMA-stationary}). By the definition of stationarity it is enough to show that for every natural number $K$, all indices $n_1,\ldots,n_K\in\Z$ and every integer $k$, the two arrays $(\be_{\Y,n_1}^{\nu,(h)},\ldots,\be_{\Y,n_K}^{\nu,(h)})$ and $(\be_{\Y,n_1+k}^{\nu,(h)},\ldots,\be_{\Y,n_K+k}^{\nu,(h)})$ have the same distribution. We first observe that for each $n\in\Z$ and each $\omega\in\Omega$, $\be_{\Y,n}^{\nu,(h)}=\lim_{s\to 0^+}\be_{\Y,n}^{\nu,(h,s)}$, where $\be_{\Y,n}^{\nu,(h,s)}\coloneqq \Delta_h^\nu[\Y](n) - \Delta_s^\nu[\Y](n)$. In particular, since $\omega$-wise convergence implies convergence in distribution, it holds that
\begin{align*}
(\be_{\Y,n_1}^{\nu,(h,s)},\ldots,\be_{\Y,n_K}^{\nu,(h,s)})\convd& (\be_{\Y,n_1}^{\nu,(h)},\ldots,\be_{\Y,n_K}^{\nu,(s)}),\\
(\be_{\Y,n_1+k}^{\nu,(h,s)},\ldots,\be_{\Y,n_K+k}^{\nu,(h,s)})\convd& (\be_{\Y,n_1+k}^{\nu,(h)},\ldots,\be_{\Y,n_K+k}^{\nu,(s)}),
\end{align*}
as $s$ tends to zero. For every finite $s$, the strict stationarity of $\Y$ implies that $(\be_{\Y,n_1}^{\nu,(h,s)},\ldots,\be_{\Y,n_K}^{\nu,(h,s)})$ is equal in distribution to $(\be_{\Y,n_1+k}^{\nu,(h,s)},\ldots,\be_{\Y,n_K+k}^{\nu,(h,s)})$. The assertion then follows from the fact that in Polish spaces weak limits are uniquely determined (\citep[Remark 13.13]{klenke2008probability}).
\end{proof}

\subsection{Approximation of integrals}
\label{subsection-integrals}
This section is devoted to the approximations of the integrals appearing in \cref{eq-recoveryDeltaL}, namely $\int_{n-1}^n{\Y(s)\dd s}$ and $\int_{n-1}^n{\ee^{\B(n-s)}\Y(s)\dd s}$. One of the simplest approximations for definite integrals is the trapezoidal rule, see, e.g., \citep[Chapter 9]{Deuflhard2008} for an introduction to the topic of numerical integration. For any function $f:\R\to M$ with values in a metric space $M$ it is defined as
\begin{equation}
\label{eq-DefTrapez}
T_{[a,b]}^Kf \coloneqq \frac{b-a}{K}\left[\frac{f(a)+f(b)}{2}+\sum_{k=1}^{N-1}{f(n-1+k\frac{b-a}{K})}\right],\quad K\in\N.
\end{equation}
and meant to approximate the definite integral $\int_a^b{f(s)\dd s}$. We will usually set $[a,b]=[n-1,n]$, $n\in\N$ and $=h^{-1}$. It is clear that $T_{[n-1,n]}^{h^{-1}}f$ can be computed from knowledge of the values of $f$ on the discrete time grid $(0,h,2h,\ldots)$. We shall now derive properties of the approximation error of convolutions of vector-valued functions with matrix-valued kernels. For any compatible functions $f:[0,\infty]\to\R^d$ and $g:[0,1]\to M_d(\R)$ we use the notation 
\begin{align}
\label{eq-trapezoidalerror}
\beps_{g\circ f,n}^{(h)} =&T_{[n-1,n]}^{h^{-1}}g(n-\cdot)f(\cdot) - \int_{n-1}^n{g(n-s)f(s)\dd s}
\end{align}
for the difference between the exact value of the convolution integral and the one obtained from the trapezoidal approximation with sampling interval $h$. 
In the next proposition we analyse this approximation error if $f$ is a multivariate CARMA process; this is the second big step towards discretizing \cref{eq-recoveryDeltaL}. 
\begin{proposition}
\label{prop-approxintegral}
Assume that $\Lb$ is a L\'evy process. Let $\Y$ be a $d$-dimensional $\Lb$-driven MCARMA process satisfying \cref{assum-eigA}, let $F:[0,1]\to M_d(\R)$ a twice continuously differentiable function and denote by $\beps_{F\circ\Y,n}^{(h)}$ the approximation error of the trapezoidal rule, defined in \cref{eq-trapezoidalerror}. If $\E\left\|\Lb(1)\right\|^k$ is finite then $\E\left\|\beps_{F\circ\Y,n}^{(h)}\right\|^k=O(h^{2k})$, as $h\to 0$. Moreover, the sequence $\beps_{F\circ\Y}^{(h)}$ is strictly stationary and strongly mixing.
\end{proposition}
\begin{proof}
By the definition of $\beps_{F\circ \Y}^{(h)}$ (\cref{eq-DefTrapez,eq-trapezoidalerror}) we can write
\begin{equation*}
\beps_{F\circ \Y,n}^{(h)} = h\sum_{i=0}^{h^{-1}}{\alpha_i^{(h)}\Y(n-1+ih)} - \int_{n-1}^n{F(n-s)\Y(s)\dd s},
\end{equation*}
where
\begin{equation*}
\alpha_0^{(h)} = \frac{F(1)}{2},\quad \alpha_{h^{-1}} = \frac{F(0)}{2},\quad \alpha_i^{(h)} = F(1-ih),\quad i=1,\ldots h^{-1}-1.
\end{equation*}
Using Dirac's $\delta$-distribution, which is defined by the property that $\int f(x)\delta_{x_0}(x)\dd x=f(x_0)$ for all compactly supported smooth functions $f$, as well as the moving average representation \labelcref{eq-Ymovingaverage} of $\Y$ we obtain that
\begin{align*}
 \beps_{F\circ \Y}^{(h)} =& \int_{n-1}^n\left[\sum_i\alpha_i^{(h)}\delta_{n-1+ih}(s) - F(n-s)\right]\Y(s)\dd s\\
  =& \int_{n-1}^n\left[h\sum_i\alpha_i^{(h)}\delta_{n-1+ih}(s) - F(n-s)\right]\int_{-\infty}^s{\underline{B} \ee^{\A(s-u)}E_p \dd\Lb(u)}\dd s.
\end{align*}
\Cref{theorem-fubini} allows us to interchange the order of integration so that we obtain
\begin{align*}
 \beps_{F\circ \Y,n}^{(h)} =& \int_{-\infty}^n\int_{\max\{u,n-1\}}^n\left[h\sum_i\alpha_i^{(h)}\delta_{n-1+ih}(s) - F(n-s)\right]\underline{B} \ee^{\A(s-u)}E_p \dd s\dd\Lb(u)\\
  =& \int_{-\infty}^{n-1}\int_{n-1}^n\left[-h\sum_i\alpha_i^{(h)}\delta_{n-1+ih}(s) - F(n-s)\right]\underline{B} \ee^{\A(s-u)}E_p \dd s\dd\Lb(u)\\
  &\qquad + \int_{n-1}^n\int_u^n\left[h\sum_i\alpha_i^{(h)}\delta_{n-1+ih}(s) - F(n-s)\right]\underline{B} \ee^{\A(s-u)}E_p \dd s\dd\Lb(u).
\end{align*}
With the notations
\begin{align}
\label{eq-DefGammah} \Gamma^{(h)} \coloneqq& \int_0^1\left[-h\sum_i\alpha_i^{(h)}\delta_{ih}(s) - F(1-s)\right]\underline{B} \ee^{\A s} \dd s\quad\text{and}\\
\label{eq-DefGh} G^{(h)}:&\begin{cases}
    [0,1]&\to M_{d,m}(\R),\\
    t&\mapsto \int_0^{t}\left[h\sum_i\alpha_i^{(h)}\delta_{t-1+ih}(s) - F(t-s)\right]\underline{B} \ee^{\A s} \dd s\, E_p,
   \end{cases}
\end{align}
we can rewrite the previous display as
\begin{equation}
\label{eq-bepsFcircY}
\beps_{F\circ \Y,n}^{(h)} = \Gamma^{(h)} \X(n-1) + \int_{n-1}^n G^{(h)}(n-u)d\Lb(u),
\end{equation}
where we have used the moving average representation \labelcref{eq-Xmovingaverage} of the state vector process $\X$. This equation and the strict stationarity of $\X$ asserted in \cref{lemma-propMCARMA},\labelcref{lemma-propMCARMA-stationary} immediately imply that the sequence $\beps_{F\circ \Y}^{(h)}$ is strictly stationary and strongly mixing. By \cref{prop-trapez} there exists a constant $C$ such that $\left\|\Gamma^{(h)}\right\|\leq C h^2$ and $\left\|G^{(h)}(t)\right\|\leq C h^2$ for all $t\in[0,1]$, which implies that
\begin{equation*}
\E\left\|\beps_{F\circ \Y}^{(h)}\right\|^k\leq h^{2k}C^k  2^k\E\left\|\X(n-1)\right\|^k+2^k\E\left\|\int_{n-1}^n G^{(h)}(n-u)\dd\Lb(u)\right\|^k.
\end{equation*}
The $k$th moment of $\X(n-1)$ is finite by \cref{lemma-propMCARMA},\labelcref{lemma-propMCARMA-moments}, so it suffices to prove that the second term is of order $O(h^{2k})$. To this end we use the fact that $\int_{n-1}^n G^{(h)}(n-u)\dd\Lb(u)$ is an infinitely divisible random variable whose characteristic triplet $(\bgamma_G^{(h)},\Sigma_G^{(h)},\nu_G^{(h)})$ can be expressed explicitly in terms of the characteristic triplet $(\bgammaL,\Sigma_{\Lb},\nuL)$ of the L\'evy process $\Lb$. Using the explicit transformation rules \labelcref{eq-transformtriplet} one sees that the condition $\left\|G^{(h)}(s)\right\|_{L^\infty([0,1],\lambda)}=O(h^2)$ implies that 
\begin{align*}
\left\|\bgamma^{(h)}_G\right\|=O(h^2),\quad& \left\|\Sigma^{(h)}_G\right\|=O(h^4),\\
\int_{\left\|\bx\right\|< 1}\left\|\bx\right\|^r\nu^{(h)}_G(\dd\bx) =&O(h^{2r}),\quad r=2,3\ldots,\\
\int_{\left\|\bx\right\|\geq 1}\left\|\bx\right\|^r\nu^{(h)}_G(\dd\bx) =&O(h^{2r}),\quad r=2,\ldots,k,
\end{align*}
so that we can apply \cref{lemma-momentsidrv} to conclude that $\E\left\|\int_{n-1}^nG(n-u)\dd\Lb(u)\right\|^k=O(h^{2k})$.
\end{proof}
It remains to estimate $\X_q(n)$. In view of the AR(1) structure given in \cref{eq-Xqnrecursion} we compute estimates
\begin{equation}
\label{eq-DefhatXhq}
\hat\X^{(h)}_q(n)=\ee^{\B}\hat\X^{(h)}_q(n-1)+\hat I^{(h)}_n,\quad\hat\X^{(h)}_q(0)=\hat\X^{(h)}_{q,0},\quad n\geq 1,
\end{equation}
where $\hat I^{(h)}_n=T_{[n-1,n]}^{h^{-1}}\ee^{\B(n-\cdot)}E_q\Y(\cdot)$ is the trapezoidal rule approximation to $\int_{n-1}^n{\ee^{\B(n-s)}E_q\Y(s)\dd s}$ and $\hat\X^{(h)}_{q,0}$ is a deterministic or random initial value. We introduce the notation
\begin{equation}
\label{eq-DefbeXn}
\be_{\X,n}^{(h)}=\hat\X_q(n)-\X_q(n).
\end{equation}
It is easy to see that the sequence $\be_{\X}^{(h)}$ satisfies $\be_{\X,n}^{(h)}=\ee^{\B}\be_{\X,n-1}^{(h)}+\beps_{F\circ\Y,n}^{(h)}$, $n\in\N$, where $F:t\mapsto \ee^{\B t}E_q$ and $\beps_{F\circ\Y,n}^{(h)}$ is of the form analysed in \cref{prop-approxintegral}. For the following result we recall the notion of an absolutely continuous measure. By Lebesgue's decomposition theorem (\citep[Theorem 7.33]{klenke2008probability}), every measure $\mu$ on $\R^m$ can be uniquely decomposed as $\mu=\mu_c+\mu_s$, where $\mu_c$ and $\mu_s$ are absolutely continuous and singular, respectively, with respect to $m$-dimensional Lebesgue measure. If $\mu_c$ is not the zero measure we say that $\mu$ has a non-trivial absolutely continuous component.
\begin{proposition}
\label{prop-eXn}
Assume that $\Lb$ is a L\'evy process. Let $\Y$ be a $d$-dimensional $\Lb$-driven MCARMA process satisfying \cref{assum-eigA,assum-eigB}. The sequence $\be_{\X}^{(h)}$ defined by \cref{eq-DefbeXn,eq-DefhatXhq} converges almost surely to a stationary and ergodic sequence which is independent of $\hat\X^{(h)}_{q,0}$. If, for some integer $k$, $\E\left\|\Lb(1)\right\|^k$ is finite, then the absolute moment $\E\left\|\be_{\X,n}^{(h)}\right\|^k$ is of order $O(h^{2k})$ as $h\to 0$. If, moreover, the distribution of the random variable
\begin{equation}
\label{eq-DefinttildeG}
\int_0^1\widetilde G(1-s)\dd\Lb(s),\quad  \widetilde G(s)=\left(\begin{array}{cc}G(s)^T &(\exp(\A s)E_p)^T\end{array}\right)^T,\quad \text{$G(s)$ defined in \cref{eq-DefGh}},
\end{equation}
has a non-trivial absolutely continuous component, then the process $\be_{\X}^{(h)}$ is exponentially strongly mixing.
\end{proposition}
\begin{proof}
We first observe that
\begin{equation*}
\be_{\X,n}^{(h)} = \ee^{(n-1)\B}\be_{\X,1}^{(h)} + \sum_{\nu=0}^{n-2}\ee^{\nu\B}\beps_{F\circ\Y,n-\nu}^{(h)},\quad n\geq 1,
\end{equation*}
and define the sequence $\tilde{\be}_{\X}^{(h)}$ by
\begin{equation*}
\tilde{\be}_{\X,n}^{(h)} =\sum_{\nu=0}^{\infty}\ee^{\nu\B}\beps_{F\circ\Y,n-\nu}^{(h)},\quad n\in\Z.
\end{equation*}
By this definition, $\tilde{\be}_{\X}^{(h)}$ is obviously independent of $\hat\X^{(h)}_{q,0}$. Since $\beps_{F\circ\Y}^{(h)}$ is strongly mixing by \cref{prop-approxintegral}, it is in particular ergodic (\citep[Exercise 20.5.1]{klenke2008probability}). The sequence $\tilde{\be}_{\X}^{(h)}$ is the unique stationary solution of the AR(1) equations
\begin{equation*}
\tilde{\be}_{\X,n}^{(h)}=\ee^{\B}\tilde{\be}_{\X,n-1}^{(h)}+\beps_{F\circ\Y,n}^{(h)},\quad n\in\Z,
\end{equation*}
and  an application of \citep[Theorem 4.3]{krengel1985ergodic} to the infinite-order moving average representation of $\tilde{\be}_{\X}^{(h)}$ shows that this last sequence is ergodic as well. It remains to prove that $\be_{\X,n}^{(h)}$ converges to $\tilde{\be}_{\X,n}^{(h)}$ almost surely as $n\to\infty$. This follows from
\begin{align*}
\left\|\be_{\X,n}^{(h)}-\tilde{\be}_{\X,n}^{(h)}\right\|\leq \left\|\ee^{(n-1)\B}\right\|\left\|\be_{\X,1}^{(h)}\right\|+\left\|\sum_{\nu=n-1}^\infty{\ee^{\nu\B}\beps_{F\circ\Y,n-\nu}^{(h)}}\right\|,
\end{align*}
the fact that by \cref{lemma-OUXq} the eigenvalues of the matrix $\B$ have strictly negative real parts, and the almost sure convergence of the last sum (\citep[Proposition 3.1.1]{brockwell1991tst}). For the proof that the $k$th moments of $\be_{\X,n}^{(h)}$ are of order $O(h^{2k})$ we use the following generalization of H\"older's inequality, which can be proved by induction: for any $k$ random variables $Z_1,\ldots,Z_k$ and positive numbers $p_1,\ldots, p_k$ such that $\sum{1/p_i}=1$ it holds that 
\begin{equation}
\label{eq-genHoelder}
\E \left(Z_1\cdot\ldots\cdot Z_k\right)\leq\prod_{i=1}^k \left(\E Z_i^{p_i}\right)^{1/p_i}.
\end{equation}
Choosing $p_i=1/k$, $i=1,\ldots,k$, and using that, by \cref{prop-approxintegral} there exists a constant $C$, independent of $n$, such that $\E\left\|\beps_{F\circ\Y,n}^{(h)}\right\|^k\leq C h^{2k}$ it follows that
\begin{align*}
\E\left\|\tilde\be_{\X,n}^{(h)}\right\|^k\leq&\sum_{\nu_1=0}^{\infty}\cdot\ldots\cdot\sum_{\nu_k=0}^{\infty}\left\|\ee^{\nu_1\B}\right\|\cdot\ldots\cdot\left\|\ee^{\nu_k\B}\right\|\E\left(\left\|\beps_{F\circ\Y,n-\nu_1}^{(h)}\right\|\cdot\ldots\cdot\left\|\beps_{F\circ\Y,n-\nu_k}^{(h)}\right\|\right)\\
  \leq &  C h^{2k}\left(\sum_{\nu=0}^\infty\left\|\ee^{\nu\B}\right\|\right)^k,
\end{align*}
which is of order $O(h^{2k})$ because the sum is finite due to the eigenvalues of $\B$ having strictly negative real parts. In order to show that the sequence $\be_{\X}^{(h)}$ is strongly mixing, we note that the stacked process $\left(\begin{array}{cc}{\be_{\X}^{(h)}}^T & \X^T\end{array}\right)^T$ satisfies the AR(1) equation
\begin{equation*}
\left(\begin{array}{c}
       \be_{\X,n}^{(h)}\\\X(n)
      \end{array}\right) = \left(\begin{array}{cc}
                                 \ee^{\B} & \Gamma \\ 0 & \ee^A
				 \end{array}\right)\left(\begin{array}{c}
       \be_{\X,n-1}^{(h)}\\\X(n-1)\end{array}\right)+\ZZ_n,\quad \ZZ_n=\int_{n-1}^n\left(\begin{array}{c}
       G(n-s)\\\ee^{A(n-s)}E_p \end{array}\right) \dd\Lb(s)     
\end{equation*}
where $(\ZZ_n)_{n\in\Z}$ is an i.i.d$.$ noise sequence. An extension of the arguments leading to \citep[Theorem 1]{mokkadem1988mixing}, which is detailed in the proof of \citep[Theorem 4.3]{schlemmmixing2010}, shows that ARMA, and in particular, AR(1) processes are strongly mixing with exponentially decaying mixing coefficients if the driving noise sequence has a non-trivial absolutely continuous component, which is precisely what is assumed in the proposition.
\end{proof}
\begin{remark}
Sufficient conditions for the assumption made in the previous proposition to hold can be obtained from the observation that the random variable $\int_0^1\widetilde G(1-s)\dd\Lb(s)$ is infinitely divisible and that its characteristic triplet can be obtained as in \cref{eq-transformtriplet}. Sufficient conditions for an infinitely divisible random variable to be absolutely continuous, in terms of its characteristic triplet, can be found in \citep{tucker1965necessary} and \citep[Section 27]{sato1991lpa}. Since mixing is not our primary concern in this paper, and our results hold without it, we do not pursue this issue further here.
\end{remark}

\subsection[{Approximation of the increments of the L\'evy process}]{Approximation of the increments $\Delta L_n$}
If we combine what we have so far it follows that we can obtain estimates $\widehat{\Delta\Lb}_n$ of the increments of the L\'evy process $\Lb$ by discretizing \cref{eq-recoveryDeltaL}, that is
\begin{align}
\label{eq-DefDeltaLhn}
\widehat{\Delta\Lb}^{(h)}_n =&\sum_{\nu=0}^{p-q-1}{\left[E_q^T\B^{p-q-1-\nu}E_q + \sum_{k=\nu}^{p-q-2}{ A_{p-q-k-1}E_q^T\B^{k-\nu}E_q}\right]\left[\Delta_h^\nu[\Y](n)-\Delta_h^\nu[\Y](n-1)\right]}\notag\\
	      &+\left[\underline{A}_q\B^{-1}+\sum_{k=1}^{p-q}{ A_{p-q-k+1}E_q^T\B^{k-1}}+E_q^T\B^{p-q}\right]\left[\hat\X^{(h)}_q(n)-\hat\X^{(h)}_q(n-1)\right]\notag\\
  &+A_p\left[B_q^{\sim 1}B_0\right]^{-1}B_q^{\sim 1}T_{[n-1,n]}^{h^{-1}}\Y,
\end{align}
where the forward differences $\Delta_h^\nu[\Y](n)$ are defined in \cref{eq-forwarddifference}, the estimates $\hat\X^{(h)}_q$ are computed recursively by \cref{eq-DefhatXhq} and the formula for the trapezoidal approximation $T_{[n-1,n]}^{h^{-1}}\Y$ is given in \cref{eq-DefTrapez}.
 Writing
\begin{equation}
\label{eq-hatDeltaL}
\widehat{\Delta \Lb}_n^{(h)} = \Delta \Lb_n + \beps_n^{(h)},
\end{equation}
the approximation error $\beps_n^{(h)}$ is given by
\begin{align*}
\beps_n^{(h)} =& \sum_{\nu=0}^{p-q-1}{\left[E_q^T\B^{p-q-1-\nu}E_q + \sum_{k=\nu}^{p-q-2}{ A_{p-q-k-1}E_q^T\B^{k-\nu}}E_q\right]\left[\be_{\Y,n}^{\nu,(h)}-\be_{\Y,n-1}^{\nu,(h)}\right]}\\
	      &+\left[\underline{A}_q\B^{-1}+\sum_{k=1}^{p-q}{ A_{p-q-k+1}E_q^T\B^{k-1}}+E_q^T\B^{p-q}\right]\left[\be_{\X,n}^{(h)}-\be_{\X,n-1}^{(h)}\right]+A_pB_0^{-1}B_{q-1}\beps_{\Y,n}^{(h)}.
\end{align*}
The following theorem summarizes the results of the previous two subsections about the probabilistic properties of the sequence of approximation errors $\beps^{(h)}$.
\begin{theorem}
\label{theorem-Levyapprox}
Assume that $\Lb$ is a L\'evy process and $\Y$ is an $\Lb$-driven multivariate CARMA process given by the state space representation \labelcref{eq-statespacerepmod} and satisfying \cref{assum-eigA,assum-eigB}. Denote by $\Delta\Lb_n=\Lb(n)-\Lb(n-1)$ the unit increments of $\Lb$ and by $\widehat{\Delta\Lb}_n^{(h)}$ the estimates of the unit increments of $\Lb$ obtained from \cref{eq-hatDeltaL}. The stochastic process $\beps^{(h)}=\widehat{\Delta\Lb}^{(h)}-\Delta\Lb$ has the following properties:
\begin{enumerate}[i)]
 \item\label{theorem-Levyapprox-stationary} There exists a stationary, ergodic stochastic process $\tilde\beps^{(h)}$ such that $\left\|\beps_n^{(h)}-\tilde\beps_n^{(h)}\right\|\to 0$ almost surely as $n\to\infty$. If the random variable defined in \cref{eq-DefinttildeG} has a non-trivial absolutely continuous component with respect to the Lebesgue measure, then $\beps^{(h)}$ is exponentially strongly mixing.
 \item\label{theorem-Levyapprox-moments} If $\E\left\|\Lb(1)\right\|^{(k)_0}<\infty$, for some positive integer $k$, then there exists a constant $C>0$ such that
\begin{equation}
\sup_{n\in \N}\E\left\|\beps_n^{(h)}\right\|^\kappa\leq C h^{1/2},\quad \kappa=1,\ldots,k.
\end{equation}
\end{enumerate}
\end{theorem}
\begin{proof}
Both claims follow directly from \cref{prop-derivMCARMA,prop-approxintegral,prop-eXn}.
\end{proof}
For the purpose of estimating a parametric model of the L\'evy process $\Lb$ based on the noisy observations $\widehat{\Delta\Lb}^{(h)}$ it is important not only to have a sound quantitative understanding of the extent to which the true increments $\Delta\Lb$ differ from the estimated increments $\widehat{\Delta\Lb}$, but also to know how strongly this difference is affected when a function is applied to the increments. This issue is investigated in the next lemma.

\begin{lemma}
\label{lemma-convpTaylorexp}
Let $f:\R^m\to\R^q$ be a function with bounded $k$th derivative and let $l$ be some fixed positive integer. 
Assume that $\E\left\|\Lb(1)\right\|^{(kl)_0}<\infty$, and further that, for any integer $1\leq r\leq k-1$ and any integers $1\leq i_1,\ldots,i_r\leq m$, the moments of the partial derivatives of $f$ satisfy
\begin{equation}
\label{eq-assumpartialfk}
\E\left\|\ppartial_{i_1}\cdots\ppartial_{i_r}f\left(\Lb(1)\right)\right\|^{kl}<\infty.
\end{equation}
It then holds that
\begin{equation}
\label{eq-taylor1}
\sup_{n\in\N}\E\left\|f\left(\widehat{\Delta\Lb}^{(h)}_n\right)-f\left(\Delta\Lb_n\right)\right\|^l= O(h^{1/2}).
\end{equation}

\end{lemma}
\begin{proof}
By Taylor's theorem (\citep[Theorem 12.14]{apostol1974mathematical}) we have that
\begin{equation*}
f\left(\widehat{\Delta\Lb}^{(h)}_n\right)-f\left(\Delta\Lb_n\right) = f\left(\Delta\Lb_n+\beps_n^{(h)}\right)-f\left(\Delta\Lb_n\right) = \sum_{r=1}^{k-1}{\frac{1}{r!}d^{(r)}f\left(\Delta\Lb_n\right)\left(\beps_n^{(h)}\right)^r}+R\left(\Delta\Lb_n;\beps_n^{(h)}\right),
\end{equation*}
where
\begin{equation*}
d^{(r)}f\left(\Delta\Lb_n\right)\left(\beps_n^{(h)}\right)^r = \sum_{i_1=1}^m\cdots\sum_{i_r=1}^m{\ppartial_{i_1}\cdots\ppartial_{i_r}f\left(\Delta\Lb_n\right)\varepsilon_{n}^{(h),i_1}\cdots\varepsilon_{n}^{(h),i_r}}
\end{equation*}
defines the action of the $r$th derivative of $f$. We note that
\begin{equation*}
\left\|d^{(r)}f\left(\Delta\Lb_n\right)\left(\beps_n^{(h)}\right)^r\right\| \leq \sum_{i_1=1}^m\cdots\sum_{i_r=1}^m\left\|\ppartial_{i_1}\cdots\ppartial_{i_r}f\left(\Delta\Lb_n\right)\right\|\left\|\beps_n^{(h)}\right\|^r\eqqcolon \left\|d^{(r)}f\left(\Delta\Lb_n\right)\right\|\left\|\beps_n^{(h)}\right\|^r,
\end{equation*}
and assumption \labelcref{eq-assumpartialfk} implies that
\begin{equation*}
\E\left\|d^{(r)}f\left(\Delta\Lb_n\right)\right\|^{kl} \leq m^{rlk}\sum_{i_1=1}^m\cdots\sum_{i_r=1}^m\E\left\|\ppartial_{i_1}\cdots\ppartial_{i_r}f\left(\Delta\Lb_n\right)\right\|^{kl}<\infty.
\end{equation*}
It follows from the boundedness of the $k$th derivative of $f$ that the remainder $R\left(\Delta\Lb_n;\beps_n^{(h)}\right)$ satisfies
\begin{equation*}
\left\|R\left(\Delta\Lb_n;\beps_n^{(h)}\right)\right\|\leq C\left\|\beps_n^{(h)}\right\|^k,
\end{equation*}
for some constant $C$. In particular,
\begin{align*}
&\E\left\|f\left(\widehat{\Delta\Lb}^{(h)}_n\right)-f\left(\Delta\Lb_n\right)\right\|^l\\
\leq & 2^l\E\left(\sum_{r=1}^{k-1}{\frac{1}{r!}\left\|d^{(r)}f\left(\Delta\Lb_n\right)\left(\beps_n^{(h)}\right)^r\right\|}\right)^l+2^l\E\left\|R\left(\Delta\Lb_n;\beps_n^{(h)}\right)\right\|^l\\
  \leq& 2^l\sum_{r_1=1}^{k-1}\cdot\cdots\cdot\sum_{r_l=1}^{k-1}\frac{1}{r_1!\cdot\ldots\cdot r_l!}\E\left(\left\|d^{(r_1)}f\left(\Delta\Lb_n\right)\right\|\cdot\ldots\cdot\left\|d^{(r_l)}f\left(\Delta\Lb_n\right)\right\|\left\|\beps_n^{(h)}\right\|^{r_1+\ldots+r_l}\right)+C2^l\E\left\|\beps_n^{(h)}\right\|^{kl}.
\end{align*}
By \cref{theorem-Levyapprox}, the assumption that $\Lb(1)$ has a finite $(kl)_0$th absolute moment implies that $\E\left\|\beps_n^{(h)}\right\|^\kappa$ is of order $O(h^{\kappa/\kappa_0})$ as $h\to 0$ for all $1\leq \kappa\leq k$, where the constant implicit in the $O(\cdot)$ notation does not depend on $n$. It thus follows by an application of the generalized H\"older inequality \labelcref{eq-genHoelder} with exponents $p_1=\ldots=p_l=kl$, $p_{l+1}=k/(k-1)$ that
\begin{align*}
\E\left\|f\left(\widehat{\Delta\Lb}^{(h)}_n\right)-f\left(\Delta\Lb_n\right)\right\|^l \leq &\sum_{r_1,\ldots,r_l=1}^{k-1} {\left[ \prod_{i=1}^l\frac{2}{r_i!}\E\left(\left\|d^{(r_i)}f\left(\Delta\Lb_n\right)\right\|^{kl}\right)^{\frac{1}{kl}} \right]} \underbrace{\E\left(\left\|\beps_n^{(h)}\right\|^{\frac{(r_1+\ldots+r_l)k}{k-1}}\right)^{\frac{k-1}{k}}}_{=O\left(h^{\frac{r_1+\ldots+r_l}{[(r_1+\ldots+r_l)k/(k-1)]_0}}\right)} + O\left(h^{\frac{kl}{(kl)_0}}\right).
\end{align*}
Since for any $\alpha\in[0,2]$ and any positive integer $r$ it holds that $(r\alpha)_0\leq r\alpha_0$, the dominating term in this sum is the one corresponding to $r_1=\ldots=r_l=1$, which is of order $O(h^{1/2})$. Thus \cref{eq-taylor1} is shown.
\end{proof}

\section{Generalized method of moments estimation with noisy data}
\label{sec-GMM}
In this section we consider the problem of estimating a parametric model $\Pb_{\vartheta}$ if only a disturbed i.i.d$.$ sample of the true distribution is available. More precisely, assume that $\Theta$ is some parameter space, that $\left(\Pb_\vartheta:\vartheta\in\Theta\right)$ is a family of probability distributions on $\R^m$ and that
\begin{equation}
\label{eq-sampleXN}
X^N = (X_1,\ldots,X_N),\quad\R^m\ni X_n\sim \Pb_{\vartheta_0},
\end{equation}
is an i.i.d$.$ sample from $\Pb_{\vartheta_0}$. The classical generalized method of moments (abbreviated as GMM) is a well-established procedure for estimating the value of $\vartheta_0$ from the observations $X^N$, see for instance \citep{hall2005generalized,hansen1982,newey1994largesample} for a general introduction.  After introducing some relevant notation and taking a closer look at two particularly important special cases of this class of estimators we state the result about the consistency and asymptotic normality of GMM estimators for easy reference in \cref{theorem-GMMold}. Our goal in this section is to extend this result to the situation where the sample $X^N$ from the distribution $\Pb_{\vartheta_0}$ cannot be observed directly. Instead, we assume that for each $h>0$ there is a stochastic process $\varepsilon^{(h)}$ not necessarily independent of $X^N$, which we think of as a disturbance to the i.i.d$.$ sample $X^N$, and the value of $\vartheta_0$ is to be estimated from the observation $(X_1+\
varepsilon_1^{(h)},\ldots,X_N+\varepsilon_N^{(h)})$. In \cref{theorem-GMMnew} we prove under a mild moment assumption that the asymptotic properties of the GMM estimator, as $N$ becomes large and $h$ becomes small, are not altered by the inclusion of the noise process $\varepsilon^{(h)}$. Finally, we use this result in \cref{theorem-GMMLevy} to answer the question of how to estimate a parametric model for the driving L\'evy process of a multivariate CARMA process from discrete-time observations.

Underlying the construction of any GMM estimator is the existence of a function $g:\R^m\times\Theta\to\R^q$ such that for $X_1\sim\Pb_{\vartheta_0}$,
\begin{equation}
\label{eq-GMMidentifiability}
\E g\left(X_1,\vartheta\right) = 0 \Leftrightarrow \vartheta=\vartheta_0.
\end{equation}
The analogy principle, that is the philosophy that unknown population averages should be approximated by sample averages, then suggests that an estimator $\hat\vartheta^N$ of $\vartheta_0$ based on the sample $X^N$, given by \cref{eq-sampleXN},  can be defined as
\begin{equation}
\label{eq-DefhatvarthetaN}
\hat\vartheta^N = \argmin_{\vartheta\in\Theta}\left\|\frac{1}{N}\sum_{n=1}^N{g\left(X_n,\vartheta\right)}\right\|_{W_N},
\end{equation}
where $W_N$ is a positive definite, possibly data-dependent, $q\times q$ matrix defining the norm 
\begin{equation*}
\left\|\cdot\right\|_{W_N}:\R^q\to\R^+,\quad \left\|x\right\|_{W_N}=(\bx^TW_N\bx)^{1/2},\quad \bx\in\R^q.
\end{equation*}
As we will see shortly the choice of $W_N$ influences the asymptotic variance $\Sigma$ of the estimator $\hat\vartheta^N$, given in \cref{eq-SigmaGMM}. The optimal choice of weighting matrices $W_N$ is described in \cref{corollary-optimalW}.

The advantage in considering a GMM approach to the estimation problem is that it contains many classical estimation procedures as special cases. Here, we only mention two such special cases which are particularly useful in the context of estimating a parametric model for a L\'evy process. It is an immediate consequence of the \cref{Def-LevyProcess} that a L\'evy-process $\Lb$ is uniquely determined by the distribution of the unit increments $\Lb(1)-\Lb(0)$, which, in turn, is characterized by its characteristic function $\E\exp\{\ii\langle\bu,\Lb(1)-\Lb(0)\rangle\}=\exp\{\psi(\bu)\}$ in its L\'evy--Khintchine form (\cref{eq-characfunc}). It is therefore natural to specify a parametric model for $\Lb$ by parametrizing the characteristic exponents, which amounts to defining, for each $\vartheta\in\Theta$, a function $\bu\mapsto\psi_{\vartheta}(\bu)$ of the form \labelcref{eq-characfunc}. A promising estimator for $\vartheta_0$ in such a model is that value of $\vartheta$ that best matches the characteristic 
function $\bu\mapsto\exp{\{\psi_\vartheta(\bu)\}}$ with its empirical counterpart. This leads to choosing the function $g$ in \cref{eq-DefhatvarthetaN} as
\begin{equation*}
g:\R^m\times\Theta\to\R^q:\quad(X,\vartheta)\mapsto \left[\begin{array}{c}\re\left(\ee^{\ii\langle\bu_k,X\rangle}-\ee^{\psi_\vartheta(\bu_k)}\right) \\ \imag\left(\ee^{\ii\langle\bu_k,X\rangle}-\ee^{\psi_\vartheta(\bu_k)}\right) \end{array}\right]_{k=1,\ldots,q/2},
\end{equation*}
where $\bu_1,\ldots,\bu_{q/2}$ are suitable elements of $\R^m$ at which the characteristic functions are to be matched. The value of $q\in2\N$ as well as the particular $\bu_j$ are chosen such that condition \labelcref{eq-GMMidentifiability} holds, which means that the model is identifiable. Another special case of the generalized method of moments estimator of considerable practical importance arises if the parametric family of distributions $\Pb_{\vartheta}$ is given as a family of probability densities $p_{\vartheta}(\cdot)$. In this case, the choice
\begin{equation*}
g:\R^m\times\Theta\to\R^q:\quad(X,\vartheta)\mapsto\nabla_{\vartheta}\log p_{\vartheta}(X)
\end{equation*}
gives rise to the classical maximum-likelihood estimator with all its desirable asymptotic properties.

In order to be able to state the classical result about the asymptotic properties of the generalized method of moments estimator for a general moment function $g$ we introduce the notations 
\begin{equation*}
\Omega_0=\E g(X_1,\vartheta_0)g(X_1,\vartheta_0)^T,\quad\text{and}\quad G_0=-\E\nabla_\vartheta g(X_1,\vartheta_0)
\end{equation*}
for the covariance matrix of the moments and the generalized score matrix, respectively, where $\nabla$ denotes the differential operator.

\begin{theorem}[{{\citep[Theorem 2.6 and Theorem 3.4]{newey1994largesample} }}]
\label{theorem-GMMold}
Assume that $(\Pb_\vartheta)_{\vartheta\in\Theta}$ is a parametric family of probability distributions and let $X^N=(X_1,\ldots,X_N)$ be an i.i.d$.$ sample from the distribution $\Pb_{\vartheta_0}$ of length $N$. Denote by $\hat\vartheta^N$ the GMM estimator based on $X^N$ defined in \cref{eq-DefhatvarthetaN}. Assume:
\begin{enumerate}[i)]
 \item\label{theorem-GMMold1} The domain $\Theta$ of $\vartheta$ is a compact subset of $\R^r$ and $\vartheta_0$ is in the interior of $\Theta$.
 \item\label{theorem-GMMold2} For each $\vartheta\in\Theta$, the function $\bx\mapsto g(\bx,\vartheta)$ is measurable; for almost every $\bx\in\R^m$ the function $\vartheta\mapsto g(\bx,\vartheta)$ is continuous on $\Theta$ and continuously differentiable in a neighbourhood $U$ of $\vartheta_0$. Moreover there exists a function $\alpha:\R^m\to \R$ satisfying $\E\alpha(X_1)<\infty$ such that for every $\vartheta_1,\vartheta_2\in U$ it holds that $\left\|\nabla_\vartheta g(\bx,\vartheta_1)-\nabla_\vartheta g(\bx,\vartheta_2)\right\|\leq \alpha(\bx)\left\|\vartheta_1-\vartheta_2\right\|$.
 \item\label{theorem-GMMold3} $\E g\left(X_1,\vartheta\right)=0$ if and only if $\vartheta=\vartheta_0$.
 \item\label{theorem-GMMold4} $\E\left\|g\left(X_1,\vartheta\right)\right\|^2<\infty$ for all $\vartheta\in\Theta$, $\Omega_0$ is a positive definite $q\times q$ matrix and $G_0$ is a $q\times r$ matrix of rank $r$.
 \item\label{theorem-GMMold5} $W_N$ are $q\times q$ matrices converging in probability to a positive definite matrix $W$.
 \item\label{theorem-GMMold6} There exists a function $\alpha:\R^m\to\R$ satisfying $\E\alpha(X_1)<\infty$ such that $\left\|g(\bx,\vartheta)g(\bx,\vartheta)^T\right\|\leq\alpha(\bx)$ and $\left\|\nabla_\vartheta g(\bx,\vartheta)\right\|\leq\alpha(\bx)$.
\end{enumerate}
It then holds that $\hat\vartheta^N$ is consistent and asymptotically normally distributed, that is
\begin{equation}
N^{1/2}(\hat\vartheta^N-\vartheta_0)\convd \mathscr{N}(\bzero_r,\Sigma),\quad N\to\infty,
\end{equation}
where the asymptotic covariance matrix $\Sigma$ is given by
\begin{equation}
\label{eq-SigmaGMM}
\Sigma = \left[G_0^T W G_0\right]^{-1}G_0^TW\Omega_0 W G_0 \left[G_0^T W G_0\right]^{-1}.
\end{equation}
\end{theorem}
A result analogous to \cref{theorem-GMMold} holds in the more general situation, where we do not have access to the sample $X^N$ but only to a noisy variant. We first introduce the necessary notation, which we will need in the proof. The generalized method of moments estimator $\hat{\hat\vartheta}^{N,h}$ of $\vartheta_0$ based on the disturbed sample $X^{N,h}=(X_1+\varepsilon^{(h)}_1,\ldots,X_N+\varepsilon^{(h)}_N)$ is defined as
\begin{equation}
\label{eq-DefhathatvarthetaNh}
\hat{\hat\vartheta}^{N,h} = \argmin_{\vartheta\in\Theta}Q_{N,h}(\vartheta),
\end{equation}
where the (random) criterion function $Q_{N,h}:\Theta\to\R^+$ has the form
\begin{equation}
\label{eq-DefQNh}
Q_{N,h}(\vartheta) = \left\|m_{N,h}(\vartheta)\right\|^2_{W_{N,h}}
\end{equation}
and $m_{N,h}:\Theta\to\R^q$ is given as
\begin{equation*}
m_{N,h}(\vartheta)=\frac{1}{N}\sum_{n=1}^N{g\left(X_n+\varepsilon_n^{(h)},\vartheta\right)},\quad g:\R^m\times\Theta\to\R^q.
\end{equation*}
Again, $W_{N,h}$ is a positive definite $q\times q$ matrix, which might depend on the sample $X^{N,h}$. As before we write $\Omega(\vartheta)=\E g\left(X_1,\vartheta\right)g\left(X_1,\vartheta\right)^T$ and
\begin{equation*}
\Omega_{N,h}(\vartheta) = \frac{1}{N}\sum_{n=1}^N{g\left(X_n+\varepsilon_n^{(h)},\vartheta\right)g\left(X_n+\varepsilon_n^{(h)},\vartheta\right)^T}
\end{equation*}
for the covariance matrix of the moments $g$ and its empirical counterpart. The sample analogue of the score matrix $G(\vartheta)=-\E\nabla_\vartheta g\left(X_1,\vartheta\right)$ is defined as
\begin{equation*}
G_{N,h}(\vartheta) = -\frac{1}{N}\sum_{n=1}^N{\nabla_\vartheta g\left(X_n+\varepsilon_n^{(h)},\vartheta\right)}.
\end{equation*}
\begin{theorem}
\label{theorem-GMMnew}
Assume that $(\Pb_\vartheta)_{\vartheta\in\Theta}$ is a parametric family of probability distributions, that $X^N$ is an i.i.d$.$ sample from the distribution $\Pb_{\vartheta_0}$ of length $N$ and that, for each $h>0$, there is a stochastic process $\varepsilon^{(h)}=\left(\varepsilon^{(h)}_n\right)_{n\in\N}$. Denote by $\hat{\hat\vartheta}^{N,h}$ the GMM estimator based on $X^{N,h}$ defined in \cref{eq-DefhathatvarthetaNh}. In addition to the assumptions of \cref{theorem-GMMold} assume:
\begin{enumerate}[i)]
\setcounter{enumi}{6}
 \item\label{theorem-GMMnew7} There exists a function $\beta:\R^+\to\R^+$ satisfying $\beta(h)\to 0$ as $h\to 0$, such that
\begin{equation}
\sup_n\E\left\|g\left(X_n+\varepsilon_n^{(h)},\vartheta_0\right)-g(X_n,\vartheta_0)\right\|=O\left(\beta(h)\right),\quad \text{as $h\to 0$}.
\end{equation}
 \item\label{theorem-GMMnew8} For all $\vartheta\in\Theta$ it holds that $\sup_n\E\left\|g\left(X_n+\varepsilon_n^{(h)},\vartheta\right)-g\left(X_n,\vartheta\right)\right\|^2\to 0$, as $h\to 0$.
\item\label{theorem-GMMnew9} For all $\vartheta\in\Theta$, the derivative of $g$ satisfies $\sup_n\E\left\|\nabla_\vartheta g\left(X_n+\varepsilon_n^{(h)},\vartheta\right)-\nabla_\vartheta g\left(X_1,\vartheta\right)\right\|\to 0$, as $h\to 0$.
\end{enumerate}
If $h=h_N$ is chosen dependent on $N$ such that $N^{1/2}\beta(h_N)\to 0$ as $N\to \infty$, then it holds that $\hat{\hat\vartheta}^{N,h_N}$ is consistent and asymptotically normal as $N\to\infty$ with the same asymptotic covariance as $\hat\vartheta^N$, given in \cref{eq-SigmaGMM}.
\end{theorem}
The proof of \cref{theorem-GMMnew} closely follows the arguments in \citep{newey1994largesample}. We give a detailed proof in order to clarify the impact of the additional parameter $h$ and the difficulties arising from the need to take the double limit $N\to\infty$ and $h\to 0$.

\begin{proof}[Proof of \cref{theorem-GMMnew}]
The proof consists of four steps. In step 1 we show that $N^{1/2}m_{N,h}(\vartheta_0)$ is asymptotically normally distributed with mean zero and covariance matrix $\Omega_0$, that $m_{N,h}(\vartheta)$, $G_{N,h}(\vartheta)$ and $\Omega_{N,h}(\vartheta)$ converge uniformly in probability to $\E g\left(X_1,\vartheta\right)$, $G(\vartheta)$ and $\Omega(\vartheta)$, respectively, and that $N\,Q_{N,h}(\vartheta_0)$ is bounded in probability. The second step consists in showing that any estimator $\tilde\vartheta^{N,h}$ that approximately minimizes the criterion function $Q_{N,h}$ in the sense that $m_{N,h}(\tilde\vartheta^{N,h})\convp 0$, converges in probability to $\vartheta_0$. In step 3 we prove that stochastic boundedness of $N\,Q_{N,h}(\tilde\vartheta^{N,h})$ implies the stochastic boundedness of $N^{1/2}(\tilde\vartheta^{N,h}-\vartheta_0)$. We will see that steps 2 and 3 imply the consistency of $\hat{\hat\vartheta}^{N,h}$ for any sequence of weighting matrices $W_{N,h}$. In the last step the mean-value 
theorem is applied to the first-order condition for $\hat{\hat\vartheta}^{N,h}$ to prove the asymptotic normality of $N^{1/2}(\hat{\hat\vartheta}^{N,h}-\vartheta_0)$.
\paragraph*{\bf Step 1}
In order to prove that $N^{1/2}m_{N,h}(\vartheta_0)$ is asymptotically normally distributed we observe that
\begin{equation*}
N^{1/2}m_{N,h}(\vartheta_0) = \frac{1}{N^{1/2}}\sum_{n=1}^N{g(X_n,\vartheta_0)} + \frac{1}{N^{1/2}}\sum_{n=1}^N{\left[g\left(X_n+\varepsilon_n^{(h)},\vartheta_0\right)-g(X_n,\vartheta_0)\right]}.
\end{equation*}
The first term in this expression is asymptotically normal by the Lindeberg-L\'evy Central Limit Theorem (\citep[Theorem 15.37]{klenke2008probability}) since the summands $g(X_n,\vartheta_0)$ are i.i.d$.$ with finite variance. It therefore suffices to show that the second term converges to zero in probability as $N\to\infty$ if $h=h_N$ satisfies $N^{1/2}\beta(h_N)\to 0$. For convenience we introduce the notation $Y_n^{(h)}=g(X_1+\varepsilon_1^{(h)},\vartheta_0)-g(X_1,\vartheta_0)$; by the linearity of expectation and assumption \labelcref{theorem-GMMnew7} it follows that
\begin{equation}
\label{eq-L1conv1}
\E\left\|N^{-1/2}\sum_{n=1}^N{Y_n^{(h)}}\right\| \leq N^{-1/2}\sum_{n=1}^N{\E\left\|Y_n^{(h)}\right\|}\leq CN^{1/2}\beta(h),\quad \text{for some $C>0$.}
\end{equation}
This proves that $N^{-1/2}\sum_{n=1}^N{Y_n^{(h_N)}}$ converges in $L^1$, and hence in probability, to zero, thereby showing the asymptotic normality of $N^{1/2}m_{N,h}(\vartheta_0)$, that is
\begin{equation}
\label{eq-asymptoticnormality-sqrt-N-m}
\Omega_0^{-1/2}N^{1/2}m_{N,h}(\vartheta_0)\eqqcolon U_{N,h}\convd U\sim \mathscr{N}(\bzero_q,\I_q),\quad \text{ as } N\to\infty,h\to 0, N^{1/2}\beta(h)\to 0.  
\end{equation}
We now turn to the uniform convergence in probability of $m_{N,h}(\vartheta)$, $G_{N,h}(\vartheta)$ and $\Omega_{N,h}(\vartheta)$: pointwise convergence of $m_{N,h}(\vartheta)$ to $\E g\left(X_1,\vartheta\right)$ follows from the observation that
\begin{equation*}
m_{N,h}(\vartheta) = \frac{1}{N}\sum_{n=1}^N{g\left(X_n,\vartheta\right)} + \frac{1}{N}\sum_{n=1}^N{\left[g\left(X_n+\varepsilon_n^{(h)},\vartheta\right)-g\left(X_n,\vartheta\right)\right]}.
\end{equation*}
As a sample average the first term converges to $\E g\left(X_1,\vartheta\right)$ as $N\to\infty$ by the law of large numbers (\citep[Theorem 5.16]{klenke2008probability}). As in \cref{eq-L1conv1} one sees that the second term converges in $L^1$ and therefore in probability to zero as $N\to\infty$ and $h\to 0$. Analogously,
\begin{equation*}
G_{N,h}(\vartheta) = \frac{1}{N}\sum_{n=1}^N{\nabla_\vartheta g\left(X_n,\vartheta\right)} + \frac{1}{N}\sum_{n=1}^N{\left[\nabla_\vartheta g\left(X_n+\varepsilon_n^{(h)},\vartheta\right)-\nabla_\vartheta g\left(X_n,\vartheta\right)\right]}.
\end{equation*}
converges pointwise in probability to $G(\vartheta)=-\E\nabla g\left(X_1,\vartheta\right)$ by assumption \labelcref{theorem-GMMnew9}. Finally
\begin{align*}
\Omega_{N,h}(\vartheta) =& \frac{1}{N}\sum_{n=1}^N{g\left(X_n,\vartheta\right)g\left(X_n,\vartheta\right)^T} + \frac{1}{N}\sum_{n=1}^N{Y_n^{(h)}\left(Y_n^{(h)}\right)^T} \\
  &+\frac{1}{N}\sum_{n=1}^N{g\left(X_n,\vartheta\right)\left(Y_n^{(h)}\right)^T}+\frac{1}{N}\sum_{n=1}^N{Y_n^{(h)}g\left(X_n,\vartheta\right)^T},
\end{align*}
where we have again used the notation $Y_n^{(h)}=g\left(X_n+\varepsilon_n^{(h)},\vartheta\right)-g\left(X_n,\vartheta\right)$. The first term in this expression for $\Omega_{N,h}(\vartheta)$ converges to $\Omega(\vartheta)=\E g\left(X_1,\vartheta\right)g\left(X_1,\vartheta\right)^T$ by the law of large numbers, the second term converges to zero in $L^1$ and in probability due to assumption \labelcref{theorem-GMMnew8}. An application of the Cauchy-Schwarz inequality to the third term shows that
\begin{align*}
\E\left\|\frac{1}{N}\sum_{n=1}^N{g\left(X_n,\vartheta\right)\left(Y_n^{(h)}\right)^T}\right\| \leq& \frac{1}{N}\sum_{n=1}^N{\E\left\|g\left(X_n,\vartheta\right)\left(Y_n^{(h)}\right)^T\right\|}\\
  \leq & \sup_n\E\left\|g\left(X_n,\vartheta\right)\left(Y_n^{(h)}\right)^T\right\|\\
  \leq & \sup_n\E\left\|g\left(X_n,\vartheta\right)\right\|\left\|Y_n^{(h)}\right\|  \leq  \sqrt{\E\left\|g\left(X_1,\vartheta\right)\right\|^2}\sqrt{\sup_n\E\left\|Y_n^{(h)}\right\|^2}.
\end{align*}
The first factor is finite by assumption \labelcref{theorem-GMMold4}, the second one converges to zero as $h\to 0$ by assumption \labelcref{theorem-GMMnew8}. By assumptions \labelcref{theorem-GMMold2,theorem-GMMold6}, the limiting functions $\vartheta\mapsto \E g\left(X_1,\vartheta\right)$, $\vartheta\mapsto G(\vartheta)$ and $\vartheta\mapsto \Omega(\vartheta)$ are continuous and dominated and since the domain $\Theta$ is compact by assumption \labelcref{theorem-GMMold1}) we can apply \cref{lemma-uniformWLLN} to conclude that the convergence is uniform in $\vartheta$. Taking into consideration the assumed convergence in probability of $W_{N,h}$ (assumption \labelcref{theorem-GMMold5}) as well as \cref{eq-asymptoticnormality-sqrt-N-m}, \cref{lemma-basicconvp} implies that $N\, Q_{N,h}(\vartheta_0)$ is bounded in probability.
\paragraph*{\bf Step 2}
In this step the consistency of any estimator $\tilde\vartheta^{N,h}$ satisfying $Q_{N,h}(\tilde\vartheta^{N,h})\convp 0$ is proved. In step 1 we have established the uniform convergence in probability of $m_{N,h}(\vartheta)$ to $\E g\left(X_1,\vartheta\right)$. Together with assumption v) this implies that $\sup_{\vartheta\in\Theta}\left|Q_{N,h}(\vartheta)-\left\|\E g\left(X_1,\vartheta\right)\right\|^2_{W}\right|\convp 0$. To establish consistency of $\tilde\vartheta^{N,h}$ we shall show that for any neighbourhood $U$ of $\vartheta_0$ and every $\epsilon>0$ there exists an $N_\epsilon(U)$ and an $h_\epsilon(U)$ such that $\Pb\left(\tilde\vartheta^{N,h}\in U\right)\geq 1-\epsilon$ for all $N>N_\epsilon(U)$, $h<h_\epsilon(U)$. For given $U$ we define $\delta(U)\coloneqq \inf_{\vartheta\in\Theta\backslash U}\left\|\E g\left(X_1,\vartheta\right)\right\|_W$ which is strictly positive by assumptions \labelcref{theorem-GMMold1,theorem-GMMold2,theorem-GMMold3}. Choosing $N_\epsilon(U)$ and $h_\epsilon(U)$ such that
\begin{align*}
\Pb\left(Q_{N,h}(\tilde\vartheta^{N,h})\leq \delta(U)/2\right)\geq &1-\epsilon/2,\\
\Pb\left(\sup_{\vartheta\in\Theta}\left|Q_{N,h}(\vartheta)-\left\|\E g\left(X_1,\vartheta\right)\right\|_{W}\right|\leq \delta(U)/2\right)\geq& 1-\epsilon/2
\end{align*}
for all $N\geq N_\epsilon(U)$ and $h\leq h_\epsilon(U)$ it follows that
\begin{align*}
\Pb\left(\tilde\vartheta^{N,h}\in U\right) \geq & \Pb\left(\left\|\E g(X_1,\tilde\vartheta^{N,h})\right\|_W\leq \delta(U)\right)\\
  \geq & \Pb\left(Q_{N,h}(\tilde\vartheta^{N,h})\leq \frac{\delta(U)}{2} \text{ and } \sup_{\vartheta\in\Theta}\left|Q_{N,h}(\vartheta)-\left\|\E g\left(X_1,\vartheta\right)\right\|_W\right|\leq \frac{\delta(U)}{2}\right)\\
  \geq & 1-\epsilon,
\end{align*}
where in the last line we used the relation $\Pb(A\cap B)\geq\Pb(A)+\Pb(B)-1$.
\paragraph*{\bf Step 3}
This step is devoted to the implication that if $N\,Q_{N,h}(\tilde\vartheta^{N,h})$ is bounded in probability, then the sequence $N^{1/2}(\tilde\vartheta^{N,h}-\vartheta_0)$ is bounded in probability as well. The assumption $N\,Q_{N,h}(\tilde\vartheta^{N,h})=O_p(1)$ implies that $Q_{N,h}(\tilde\vartheta^{N,h})\convp 0$ and therefore, by the previous step, that $\tilde\vartheta^{N,h}\convp\vartheta_0$. By the mean-value theorem there exist $\vartheta^*_i\in\Theta$ of the form $\vartheta^*_i = \vartheta_0 + c_i(\tilde\vartheta^{N,h} - \vartheta_0)$, $0\leq c_i\leq 1$, $i=1,\ldots,r$, such that we can write
\begin{align}
\label{eq-mvTaylor}N^{1/2}m_{N,h}(\tilde\vartheta^{N,h}) =& N^{1/2}m_{N,h}(\vartheta_0) + \nabla_\vartheta m_{N,h}(\underline\vartheta^*)N^{1/2}(\tilde\vartheta^{N,h}-\vartheta_0)\notag\\
=&\Omega_0^{1/2}U_{N,h} - G_{N,h}(\underline\vartheta^*)N^{1/2}(\tilde\vartheta^{N,h}-\vartheta_0),
\end{align}
where $G_{N,h}(\underline\vartheta^*)$ denotes the matrix whose $i$th row coincides with the $i$th row of $G(\vartheta^*_i)$ and $U_{N,h}$ is defined in \cref{eq-asymptoticnormality-sqrt-N-m}. By applying the triangle inequality of the norm $\left\|\cdot\right\|_{W_{N,h}}$ to the vector
\begin{equation*}
G_{N,h}(\underline\vartheta^*)N^{1/2}(\tilde\vartheta^{N,h}-\vartheta_0)=\Omega_0^{1/2}U_{N,h}-N^{1/2}m_{N,h}(\tilde\vartheta^{N,h})
\end{equation*}
one obtains that
\begin{equation*}
\left\|G_{N,h}(\underline\vartheta^*)N^{1/2}(\tilde\vartheta^{N,h}-\vartheta_0)\right\|^2_{W_{N,h}}\leq 2\left\|\Omega_0^{1/2}U_{N,h}\right\|^2_{W_{N,h}} + 2N\,Q_{N,h}(\tilde\vartheta^{N,h}).
\end{equation*}
Since $U_{N,h}$ converges in distribution to a standard normal and $W_{N,h}$ converges in probability, the first term on the right hand side of the last display converges in distribution by \cref{lemma-basicconvp} and is in particular bounded in probability. By our hypothesis, $N\,Q_{N,h}(\tilde\vartheta^{N,h})$ is bounded in probability and so it follows that $\left\|G_{N,h}(\underline\vartheta^*)N^{1/2}(\tilde\vartheta^{N,h}-\vartheta_0)\right\|_{W_{N,h}}$ is bounded in probability as well. It follows from the uniform convergence in probability of $G_{N,h}(\vartheta)$ to $G(\vartheta)$, the fact that $\vartheta_i^*\convp\vartheta_0$ and \cref{lemma-continuousmapping} applied to the rows of $G_{N,h}$ that $G_{N,h}(\underline\vartheta^*)^TW_{N,h}G_{N,h}(\underline\vartheta^*)\convp G_0^TWG_0$, which in turn implies that $N^{1/2}(\tilde\vartheta^{N,h}-\vartheta_0)$ is bounded in probability.
\paragraph*{\bf Step 4}
In this last step we prove that the estimator $\hat{\hat\vartheta}^{N,h}=\argmin_{\vartheta\in\Theta}Q_{N,h}(\vartheta)$ is asymptotically normally distributed. The definition of $\hat{\hat\vartheta}^{N,h}$ implies that $Q_{N,h}(\hat{\hat\vartheta}^{N,h})\leq Q_{N,h}(\vartheta_0)$. We have shown in step 1 that $N\, Q_{N,h}(\vartheta_0)$ is bounded in probability and hence so is $N\, Q_{N,h}(\hat{\hat\vartheta}^{N,h})$. This implies by step 2 that $\hat{\hat\vartheta}^{N,h}$ is consistent and that $N^{1/2}(\hat{\hat\vartheta}^{N,h}-\vartheta_0)$ is bounded in probability. Since $\hat{\hat\vartheta}^{N,h}$ is an extremal point of $Q_{N,h}$ we obtain by setting the derivative equal to zero that $G_{N,h}(\hat{\hat\vartheta}^{N,h})^TW_{N,h}N^{1/2}m_{N,h}(\hat{\hat\vartheta}^{N,h})=0$. By combining the Taylor expansion \labelcref{eq-mvTaylor} with this first-order condition it follows that
\begin{equation*}
0 = G_{N,h}(\hat{\hat\vartheta}^{N,h})^TW_{N,h}\Omega_0^{1/2}U_{N,h} - G_{N,h}(\hat{\hat\vartheta}^{N,h})^TW_{N,h}G_{N,h}(\underline\vartheta^*)N^{1/2}(\hat{\hat\vartheta}^{N,h}-\vartheta_0).
\end{equation*}
As before one sees that $G_{N,h}(\hat{\hat\vartheta}^{N,h})^TW_{N,h}G_{N,h}(\underline\vartheta^*)$ converges in probability to the non-singular limit $G_0^T W G_0$, which means that $N^{1/2}(\hat{\hat\vartheta}^{N,h}-\vartheta_0) = \left[G_{N,h}(\hat{\hat\vartheta}^{N,h})^TW_{N,h}G_{N,h}(\underline\vartheta^*)\right]^{-1}G_{N,h}(\hat{\hat\vartheta}^{N,h})^TW_{N,h}\Omega_0^{1/2}U_{N,h}$ exists with probability approaching one. Since
\begin{equation*}
\left[G_{N,h}(\hat{\hat\vartheta}^{N,h})^TW_{N,h}G_{N,h}(\underline\vartheta^*)\right]^{-1}G_{N,h}(\hat{\hat\vartheta}^{N,h})^TW_{N,h} \convp \left[G_0^TWG_0\right]^{-1}G_0^TW,
\end{equation*}
it follows from \cref{lemma-basicconvp} that $N^{1/2}(\hat{\hat\vartheta}^{N,h}-\vartheta_0)\convd \left[G_0^TWG_0\right]^{-1}G_0^TW\Omega_0^{1/2}U$, a normally distributed random vector with covariance matrix $\Sigma = \left[G_0^TWG_0\right]^{-1}G_0^TW\Omega_0 W G_0\left[G_0^TWG_0\right]^{-1}$. If the dimension $r$ of the parameter space $\Theta$ is equal to the dimension $q$ of the moment vector and the matrix $G_0$ is thus square or if $W=\Omega_0^{-1}$ it follows that $\Sigma = \left[G_0^T\Omega_0^{-1} G_0\right]^{-1}$.
\end{proof}
\begin{remark}
It seems possible to extend most aspects of the asymptotic theory of the generalized method of moments beyond the Central Limit Theorem \labelcref{theorem-GMMold} to deal, for example, with non-compact parameter spaces and applications to hypothesis testing  based on a disturbed sample as in Theorem 6.2.  We choose not to pursue these possibilities further in the present paper.
\end{remark}

In view of \cref{lemma-continuousmapping}, assumption \labelcref{theorem-GMMold5} of \cref{theorem-GMMnew} is satisfied if we choose $W_{N,h} = W_{N,h}(\bar\vartheta^{N,h})$ where $\bar\vartheta^{N,h}$ is a consistent estimator of $\vartheta_0$ and the functions $\vartheta\mapsto W_{N,h}(\vartheta)$ converge uniformly in probability to $\vartheta\mapsto W(\vartheta)$. In this way one can construct a sequence $W_{N,h}$ of weighting matrices converging in probability to $\Omega_0^{-1}$. For this two-stage GMM estimation procedure one has the following optimality result.
\begin{corollary}
\label{corollary-optimalW}
Let $\tilde\vartheta^{N,h}$ be the estimate of $\vartheta$ obtained from maximizing the $W$-norm of $m_{N,h}(\vartheta)$ for any fixed $q\times q$ positive definite matrix $W$ and let $\hat{\hat\vartheta}^{N,h}$ be the estimate obtained from using the random weighting matrix
\begin{equation}
\widetilde W_{N,h} = \Omega_{N,h}(\tilde\vartheta^{N,h})^{-1}=\left[\frac{1}{N}\sum_{n=1}^N{g\left(X_n+\varepsilon_n^{(h)},\tilde\vartheta^{N,h}\right)^Tg\left(X_n+\varepsilon_n^{(h)},\tilde\vartheta^{N,h}\right)}\right]^{-1}.
\end{equation}
Under the conditions of \cref{theorem-GMMnew}, the estimator $\hat{\hat\vartheta}^{N,h}$ is consistent and asymptotically normally distributed. In the partial order induced by positive semidefiniteness, the asymptotic covariance matrix of the limiting normal distribution, $\left[G_0^T\Omega_0^{-1}G_0\right]^{-1}$, is smaller than or equal to the covariance matrix obtained from every other sequence of weighting matrices $W_{N,h}$.
\end{corollary}
\begin{proof}
It has been shown in the proof of \cref{theorem-GMMnew} that the preliminary estimator $\tilde\vartheta^{N,h}$ is consistent and that the sequence of functions $\vartheta\mapsto\Omega_{N,h}(\vartheta)$ converges uniformly in probability to the function $\vartheta\mapsto\Omega(\vartheta)$. It then follows from \cref{lemma-continuousmapping} that the sequence $\widetilde W_{N,h}$ of weighting matrices converges in probability to $\Omega_0^{-1}$ and from \cref{theorem-GMMnew} that $\hat{\hat\vartheta}^{N,h}$ is asymptotically normal with asymptotic covariance matrix 
\begin{equation*}
\left[G_0^T\Omega_0^{-1}G_0\right]^{-1}G_0^T\Omega_0^{-1}\Omega_0 \Omega_0^{-1}G_0\left[G_0^T\Omega_0^{-1}G_0\right]^{-1} = \left[G_0^T \Omega_0^{-1}G_0\right]^{-1}.
\end{equation*}
To show that this is smaller than or equal to the asymptotic covariance matrix of an estimator obtained from using a sequence of weighting matrices that converges in probability to the positive definite matrix $W$ we must show that the matrix
\begin{equation*}
\Delta=\left[G_0^TWG_0\right]^{-1}G_0^TW\Omega_0 W G_0\left[G_0^TWG_0\right]^{-1} - \left[G_0^T \Omega_0^{-1}G_0\right]^{-1}
\end{equation*}
is positive semi-definite. To see this it is enough to note that $\Delta$ can be written as
\begin{equation*}
 \Delta = \left[\Omega_0^{1/2}WG_0\left(G_0^TWG_0\right)^{-1}\right]^T\left[\I_r-\Omega_0^{-1/2}G_0\left(G_0^T\Omega_0^{-1}G_0\right)^{-1}G_0^T\Omega_0^{-1/2}\right]\left[\Omega_0^{1/2}WG_0\left(G_0^TWG_0\right)^{-1}\right].
\end{equation*}
Since the factor in the middle is idempotent and therefore positive semidefinite and semidefiniteness is preserved under conjugation, the matrix $\Delta$ is positive semidefinite.
\end{proof}

We can now state and prove our main result about the asymptotic properties of the generalized method of moments estimation of the driving L\'evy process of a multivariate CARMA process from discrete observations. This method can be used to select a suitable driving process from within a parametric family of L\'evy processes as part of specifying a CARMA model for an observed time series.
We assume that $\Theta$ is a parameter space and that $(\Lb_\vartheta)_{\vartheta\in\Theta}$ is a family of L\'evy processes. The process $\Y$ is an $\Lb_{\bth_0}$-driven multivariate CARMA(p,q) process given by a state space representation of the form \labelcref{eq-statespacerepmod} and we assume that $h$-spaced observations $\Y(0),\Y(h),\ldots,\Y(N+(p-q-1)h)$ of $\Y$ are available on the discrete time grid $(0,h,\ldots,N+(p-q-1)h)$. Based on these observed values, a set of $N$ approximate unit increments $\widehat{\Delta\Lb}_n^{(h)}$, $n=1,\ldots,N$, of the driving process is computed using \cref{eq-DefDeltaLhn}. For each integer $N$ and each sampling frequency $h^{-1}\in\N$, a generalized method of moments estimator is defined as in \cref{eq-DefhathatvarthetaNh} by
\begin{equation}
\label{eq-DefhathatvarthetaNhL}
\hat{\hat\vartheta}^{N,h} = \argmin_{\vartheta\in\Theta}\left\|\frac{1}{N}\sum_{n=1}^Ng\left(\widehat{\Delta\Lb}_n^{(h)},\vartheta\right)\right\|_{W_{N,h}},
\end{equation}
where $g:\R^m\times\Theta\to\R^q$ is a moment function and $W_{N,h}\in M_q(\R)$ is a positive definite weighting matrix. The following theorem asserts that the sequence $(\hat{\hat\vartheta}^{N,h_N})_N$ of estimators is consistent and asymptotically normally distributed if $h_N$ is chosen such that $Nh_N$ converges to zero.
\begin{theorem}
\label{theorem-GMMLevy}
Assume that $\Theta\subset\R^r$ is a parameter space, that $(\Lb_\vartheta)_{\vartheta\in\Theta}$ is a family of $m$-dimensional L\'evy processes and that $\Y$ is an $\Lb_{\vartheta_0}$-driven multivariate CARMA process satisfying \cref{assum-eigA,assum-eigB}. Denote by $\hat{\hat\vartheta}^{N,h}$ the generalized method of moments estimator defined in \cref{eq-DefhathatvarthetaNhL}. Assume that, for some integer $k$, the functions $f_\vartheta:\bx\mapsto g(\bx,\vartheta)$ possess a bounded $k$th derivative, that $\E\left\|\Lb_{\vartheta_0}(1)\right\|^{2k}$ is finite and that the partial derivatives of the functions $f_\vartheta$ satisfy
\begin{equation}
\label{eq-partialf}
\E\left\|\ppartial_{i_1}\cdot\ldots\cdot\ppartial_{i_\kappa}f_\vartheta\left(\Lb_{\vartheta_0}(1)\right)\right\|^{2k}<\infty,\quad 1\leq i_1,\ldots,i_\kappa\leq m,\quad 1\leq \kappa\leq k-1,\quad\vartheta\in\Theta.
\end{equation}
Further assume that, for each $\bx\in\R^m$, the function $\vartheta\mapsto g(\bx,\vartheta)$ is differentiable, that, for some integer $l$, the functions $h_\vartheta:\bx\mapsto\nabla_\vartheta g(\bx,\vartheta)$ have a bounded $l$th derivative and that the partial derivatives of $h_\vartheta$ satisfy
\begin{equation}
\label{eq-partialh}
\E\left\|\ppartial_{i_1}\cdot\ldots\cdot\ppartial_{i_\lambda}h_\vartheta\left(\Lb_{\vartheta_0}(1)\right)\right\|^l<\infty,\quad 1\leq i_1,\ldots,i_\lambda\leq m,\quad 1\leq \lambda\leq l-1,\quad\vartheta\in\Theta.
\end{equation}
If, in addition, assumptions \labelcref{theorem-GMMold1,theorem-GMMold2,theorem-GMMold3,theorem-GMMold4,theorem-GMMold5,theorem-GMMold6} of \cref{theorem-GMMold} are satisfied with $X_1$ replaced by $\Lb_{\vartheta_0}(1)$, and if $h=h_N$ is chosen dependent on $N$ such that $Nh_N$ converges to zero as $N$ tends to infinity, then the estimator $\hat{\hat\vartheta}^{N,h_N}$ is consistent and asymptotically normally distributed with asymptotic covariance matrix given in \cref{eq-SigmaGMM}.
\end{theorem}
\begin{proof}
It suffices to check conditions \labelcref{theorem-GMMnew7,theorem-GMMnew8,theorem-GMMnew9} of \cref{theorem-GMMnew}. All three conditions follow by assumptions \labelcref{eq-partialf,eq-partialh} from \cref{lemma-convpTaylorexp}, which also shows that the function $\beta$ in \labelcref{theorem-GMMnew7} can be taken as $\beta:h\mapsto h^{1/2}$. Consequently, the assumption $N^{1/2}\beta(h_N)\to 0$ from \cref{theorem-GMMnew} simplifies to the requirement that $Nh_N$ converges to zero and the result follows.
\end{proof}
\begin{remark}
If we introduce the notation $n=N/h$ for the total number of observations of the MCARMA process, the high-frequency condition from \cref{theorem-GMMLevy} becomes $nh_n^2\to 0$, which is the rate commonly encountered in the literature when dealing with the estimation of continuous-time processes.
\end{remark}

\section{Simulation study}
\label{section-simulation}
\FloatBarrier
In this section we illustrate the estimation procedure developed in this paper using the example of a univariate CARMA(3,1) process $\Y$ driven by a Gamma process.
A similar example was considered in \citep{brockwell2010estimation} as a model for the realized volatility of DM/\$ exchange rates. 
Gamma processes are a family of univariate infinite activity pure-jump L\'evy subordinators $\left(\Gamma_{b,a}(t)\right)_{t\in\R}$, which are parametrized by two positive real numbers $a$ and $b$, see, e.g., \citep[Example1.3.22]{applebaum2004lpa}. Their moment generating function is given by
\begin{equation*}
u\mapsto\E\ee^{\Gamma_{b,a}(t)u} = \left(1-bu\right)^{-at},\quad a,b>0;
\end{equation*}
the unit increments $\Gamma_{b,a}(n)-\Gamma_{b,a}(n-1)$ follow a Gamma distribution with scale parameter $b$ and shape parameter $a$. This distribution has density
\begin{equation*}
f_{b,a}(x) = \frac{1}{\Gamma(a)b}\left(x/b\right)^{a-1}\ee^{-x/b},
\end{equation*}
mean $ab$ and cumulative distribution function
\begin{equation}
\label{eq-CDFgamma}
F_{b,a}(x) = \int_0^x{f_{b,a}(\xi)\dd\xi} = \frac{\Gamma\left(a;x/b\right)}{\Gamma(a)},
\end{equation}
where $\Gamma(\cdot)$ and $\Gamma(\cdot;\cdot)$ denote the complete and the lower incomplete gamma function, respectively.

In contrast to the example studied in \citep{brockwell2010estimation} we chose to simulate a model of order $(3,1)$ in order to demonstrate the feasibility of approximating the derivatives $\DD^\nu\Y$ which appear in \cref{eq-DefDeltaLhn}. The dynamics of the CARMA process used in the simulations are determined by the polynomials
\begin{equation*}
P(z) = z^3 + 2z^2 + \frac{3}{2}z  + \frac{1}{2},\quad\text{ and }\quad Q(z) = 1 + z,
\end{equation*} 
corresponding to autoregressive roots $\lambda_1 = -1$ and $\lambda_{2,3}=-1\pm\ii$. The process $\Y$ is simulated 
by applying an Euler scheme with step width $5\times10^{-4}$ to the state space model (cf. \cref{theorem-AlternativeSSR})
\begin{equation}
\label{eq-Ysimu}
\dd \X(t) = \left[\begin{array}{ccc}0 & 1 & 0\\0 & 0 & 1\\ -\frac{1}{2} & -\frac{3}{2} & -2 \end{array}\right]\X(t)\dd t + \left[\begin{array}{c}0\\0\\1\end{array}\right]\dd\Gamma_{2,1}(t),\quad\Y(t) = \left[\begin{array}{cc}1&1\end{array}\right]\X(t).
\end{equation}
The initial value $\X(0)$ is set to zero. Another possibility would be to sample $\X(0)$ from the marginal distribution of the stationary solution of \cref{eq-Ysimu}, but since the effect of the choice of $\X(0)$ decays at an exponential rate this does not make a substantial difference. A typical realization of the resulting CARMA process $\Y$ on the time interval $[0,200]$ is depicted in \cref{fig-typrealCARMA31gammaY}. In the case of finite variation L\'evy processes there is a pathwise correspondence between a CARMA process and the driving L\'evy process. Since this applies in particular to Gamma processes, it is possible to show in \cref{fig-typrealCARMA31gammaL} the path of the driving process which generated the shown realization of $\Y$. Such a juxtaposition is useful in that it allows to see how big jumps in the driving process can cause spikes in the resulting CARMA process.

\begin{figure}[!htbp]
\centering
\subfloat[$\Gamma_{2,1}$-process]
{
\includegraphics[width=.45\textwidth]{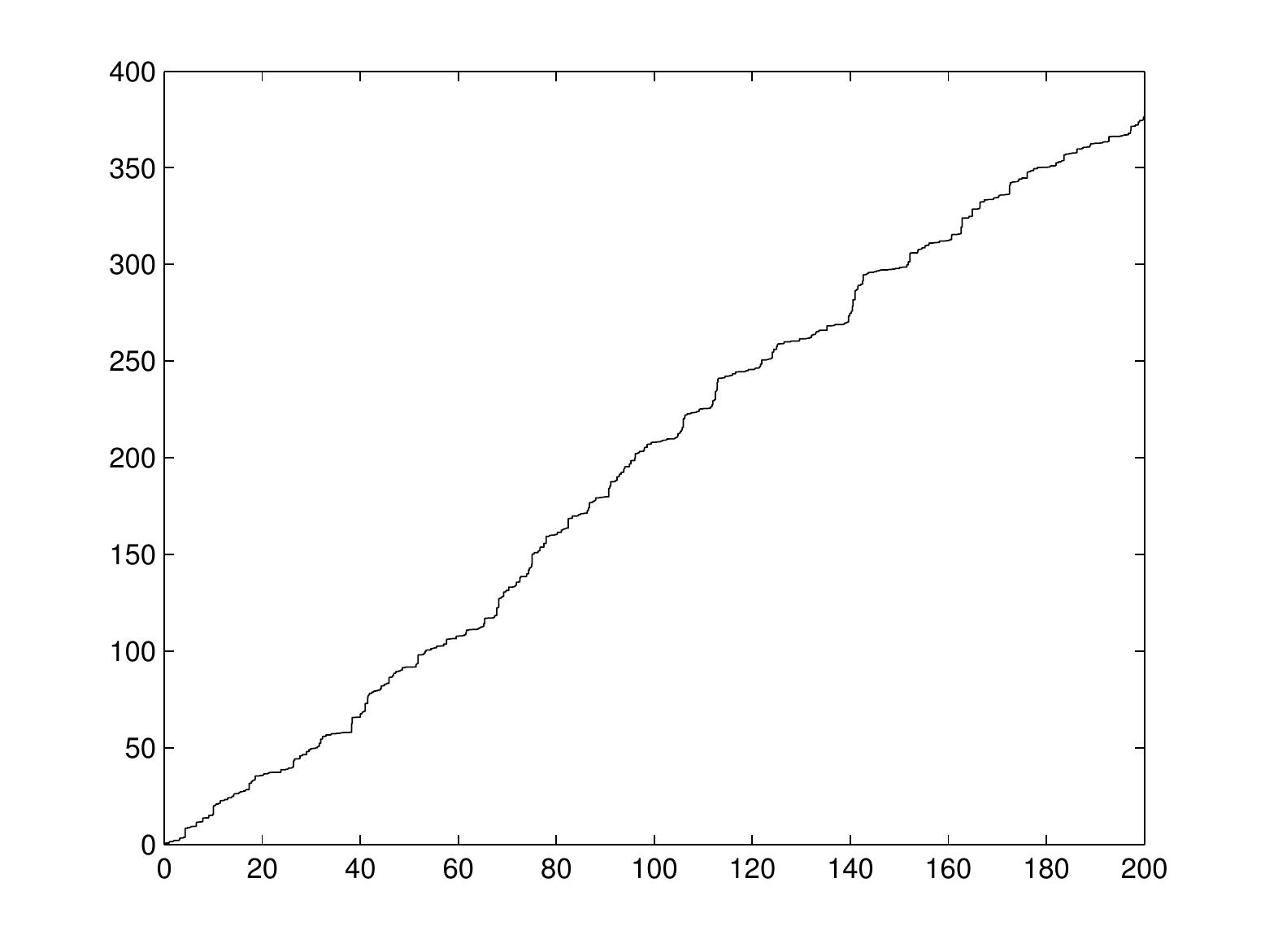}
\label{fig-typrealCARMA31gammaL}
}
\subfloat[$\Gamma_{2,1}$-driven CARMA(3,1) process]
{
\includegraphics[width=.45\textwidth]{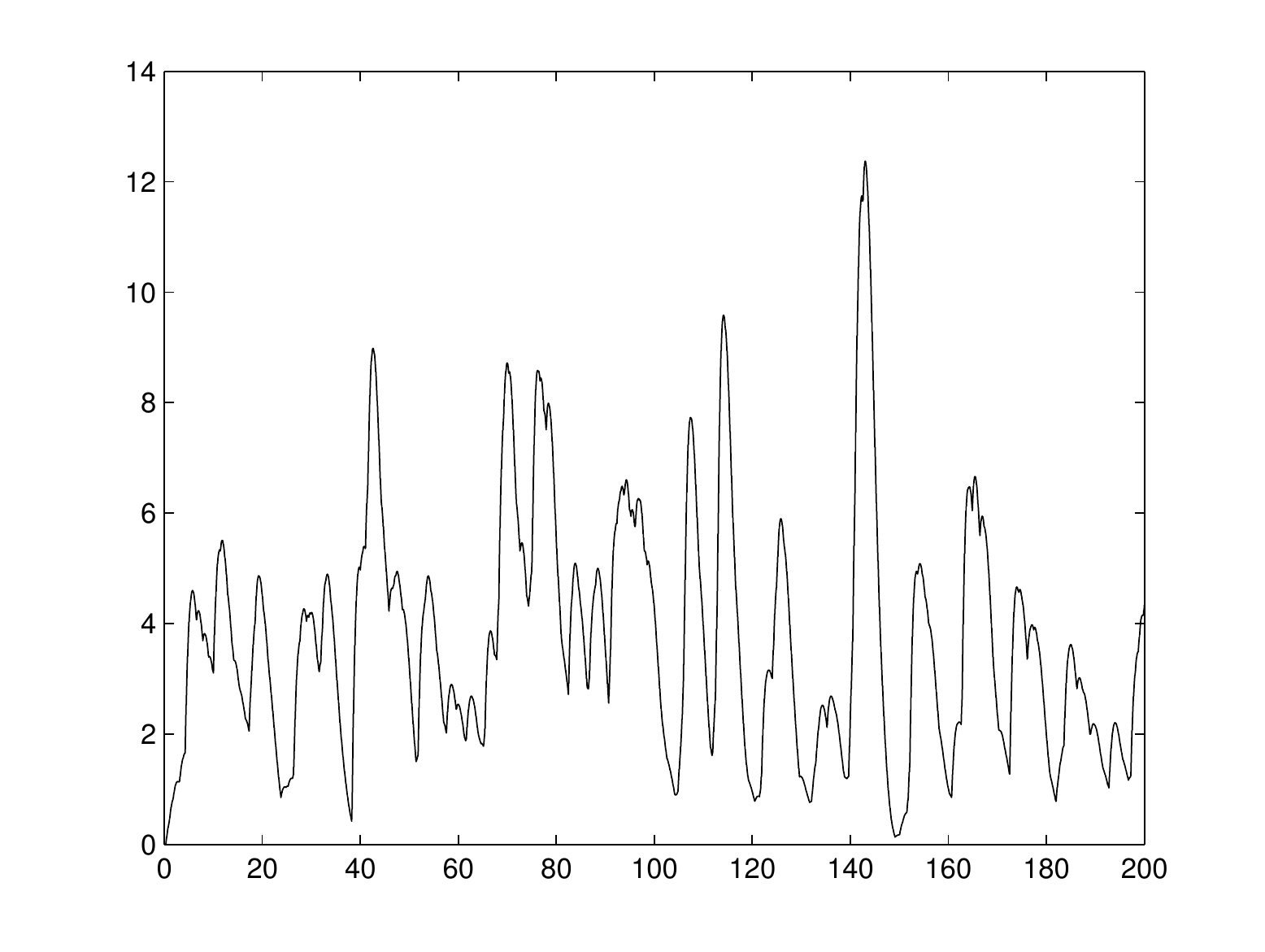}
\label{fig-typrealCARMA31gammaY}
}
\caption{{\small Typical realization of a $\Gamma_{2,1}$-process and the corresponding CARMA(3,1) process with dynamics given by \cref{eq-Ysimu}}
}
\label{fig-typrealCARMA31gamma}
\end{figure}

The first step in the implementation of our estimation procedure is to approximate the increments $\Delta\Gamma_n$ of the driving Gamma process from discrete-time observations of the CARMA process $\Y$. For the value $h=0.01$ of the sampling interval \cref{fig-compareincrements} compares the true increments with the approximations $\widehat{\Delta\Gamma}_n^{(h)}$ obtained from \cref{eq-DefDeltaLhn} both directly and in terms of their cumulative distribution functions. We see that the approximations $\widehat{\Delta\Gamma}_n^{(h)}$ are very good for each individual increment and that therefore the empirical distribution function of the reconstructed increments closely follows the CDF \labelcref{eq-CDFgamma} of the gamma distribution even if the observation period is rather short.

\begin{figure}[!htbp]
\centering
\subfloat[Bar chart of the increments of driving $\Gamma_{2,1}$-process. White bars represent the true increments, black bars indicate the values of the estimates.]
{
\label{fig-compareincrementshisto}
\includegraphics[width=.45\textwidth]{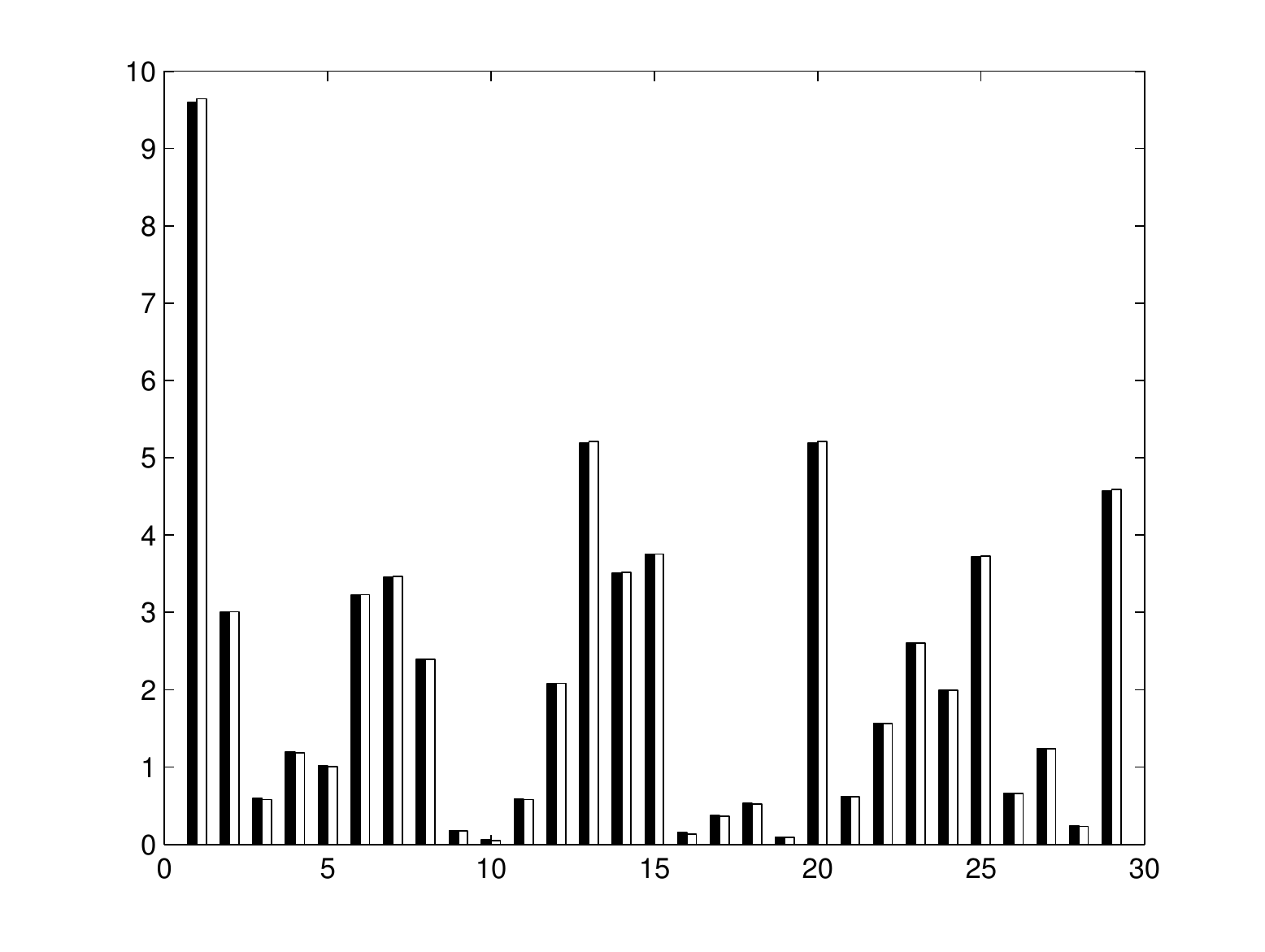}
}
\hspace{.05\textwidth}
\subfloat[Cumulative distribution function of the increments of the driving $\Gamma_{2,1}$-process. The dashed line shows the true CDF given by \cref{eq-CDFgamma}, the solid line represents the empirical distribution function of the estimates.]
{
\label{fig-compareincrementsCFD}
\includegraphics[width=.45\textwidth]{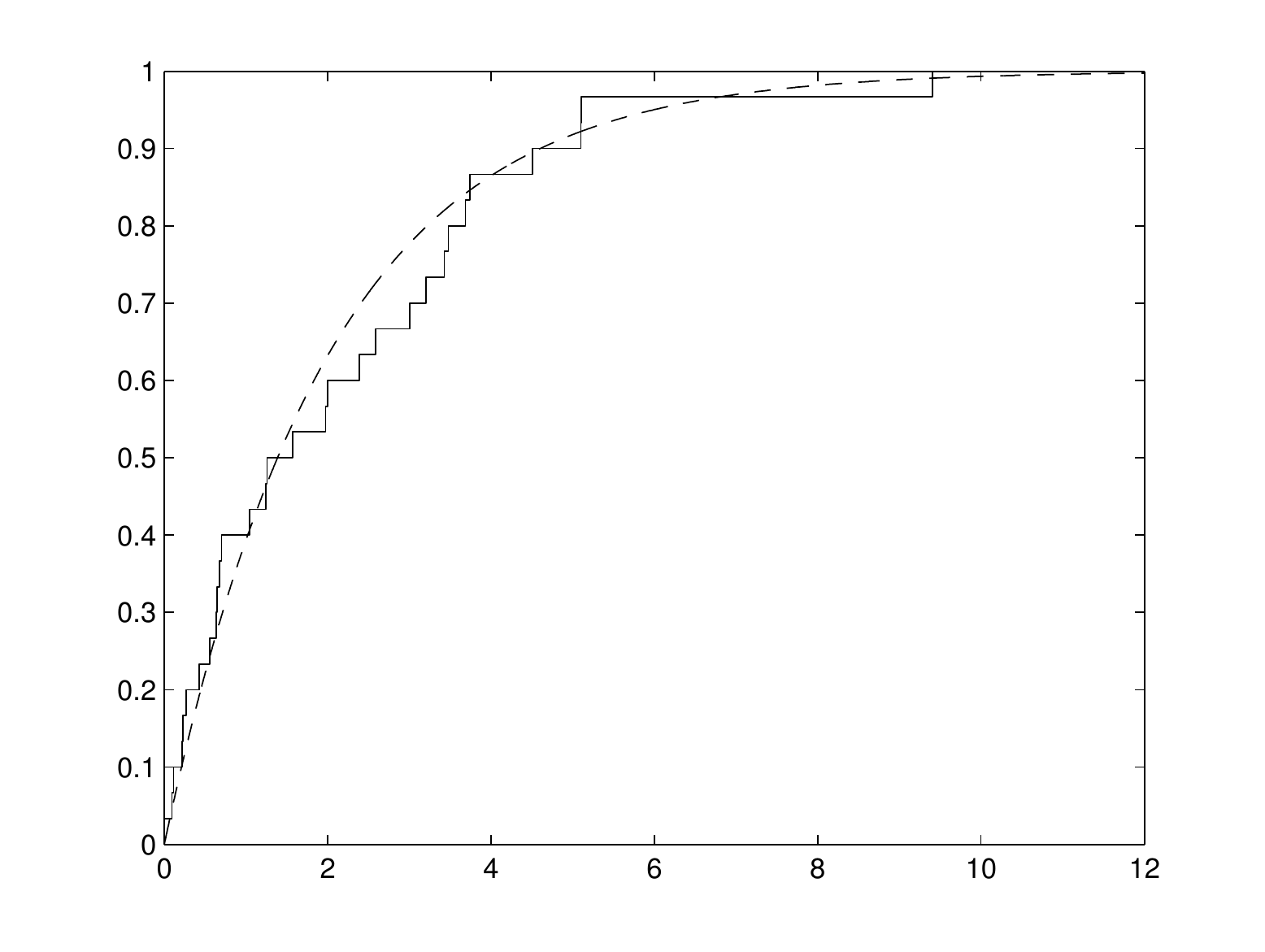}
}
\caption{{\small Comparison of the true increments of a Gamma process with parameters $(b,a)=(2,1)$ to the estimates of the increments computed via \cref{eq-DefDeltaLhn} from discrete observations of the $\Gamma_{2,1}$-driven CARMA(3,1) process defined by \cref{eq-Ysimu} on the time grid $(0,0.01,0.02,\ldots,30)$.}}
\label{fig-compareincrements}
\end{figure}

In the next step we used the approximate increments $\widehat{\Delta\Gamma}_n^{(h)}$ and a standard numerical optimization routine to compute the maximum likelihood estimator
\begin{equation}
\label{eq-MLgamma}
\left(\hat b^{N,(h)},\hat a^{N,(h)}\right) = \argmax_{(a,b)\in \R_+^2}\prod_{n=1}^Nf_{b,a}\left(\widehat{\Delta\Gamma}_n^{(h)}\right),
\end{equation}
or, equivalently,
\begin{equation*}
\left(\hat b^{N,(h)},\hat a^{N,(h)}\right) = \argmin_{(a,b)\in \R_+^2}\left\|\sum_{n=1}^N\nabla_{(b,a)}\log f_{b,a}\left(\widehat{\Delta\Gamma}_n^{(h)}\right)\right\|.
\end{equation*}
In this form, the maximum likelihood estimator falls into the class of generalized moments estimators. From the explicit form of the function $g=\nabla_{(b,a)}\log f_{b,a}$ it is easy to check that the assumptions of \cref{theorem-GMMLevy} are satisfied. Since in the present case, and for maximum likelihood estimators in general, the dimension of the moment vector is equal to the dimension of the parameter space, the choice of the weighting matrices $W_{N,h}$ is irrelevant and the estimator is always best in the sense of \cref{corollary-optimalW}.

With the goal of confirming the assertions of \cref{theorem-GMMLevy} we first focused on consistency and investigated the effect of finite sampling frequencies.  \Cref{fig-crosseconsistencyNh} visualizes the empirical means and marginal standard deviations of the maximum likelihood estimator \labelcref{eq-MLgamma} obtained from $500$ independent realizations of the CARMA process $\Y$ from \cref{eq-Ysimu} simulated over the time horizon $[0,200]$ and sampled at instants $(0,h,2h,\ldots,N)$ for different values of $h$. The picture suggests that the estimator  $\left(\hat b^{N,(h)},\hat a^{N,(h)}\right)$ is biased for positive values of $h$, even as $N$ tends to inifinity, but that it is consistent as $h$ tends to zero. This is in agreement with \cref{theorem-GMMLevy} and reflects the intuition that discrete sampling entails a loss of information compared with a genuinely continuous-time observation of a stochastic process.

\begin{figure}
\centering
\includegraphics[width=.75\textwidth]{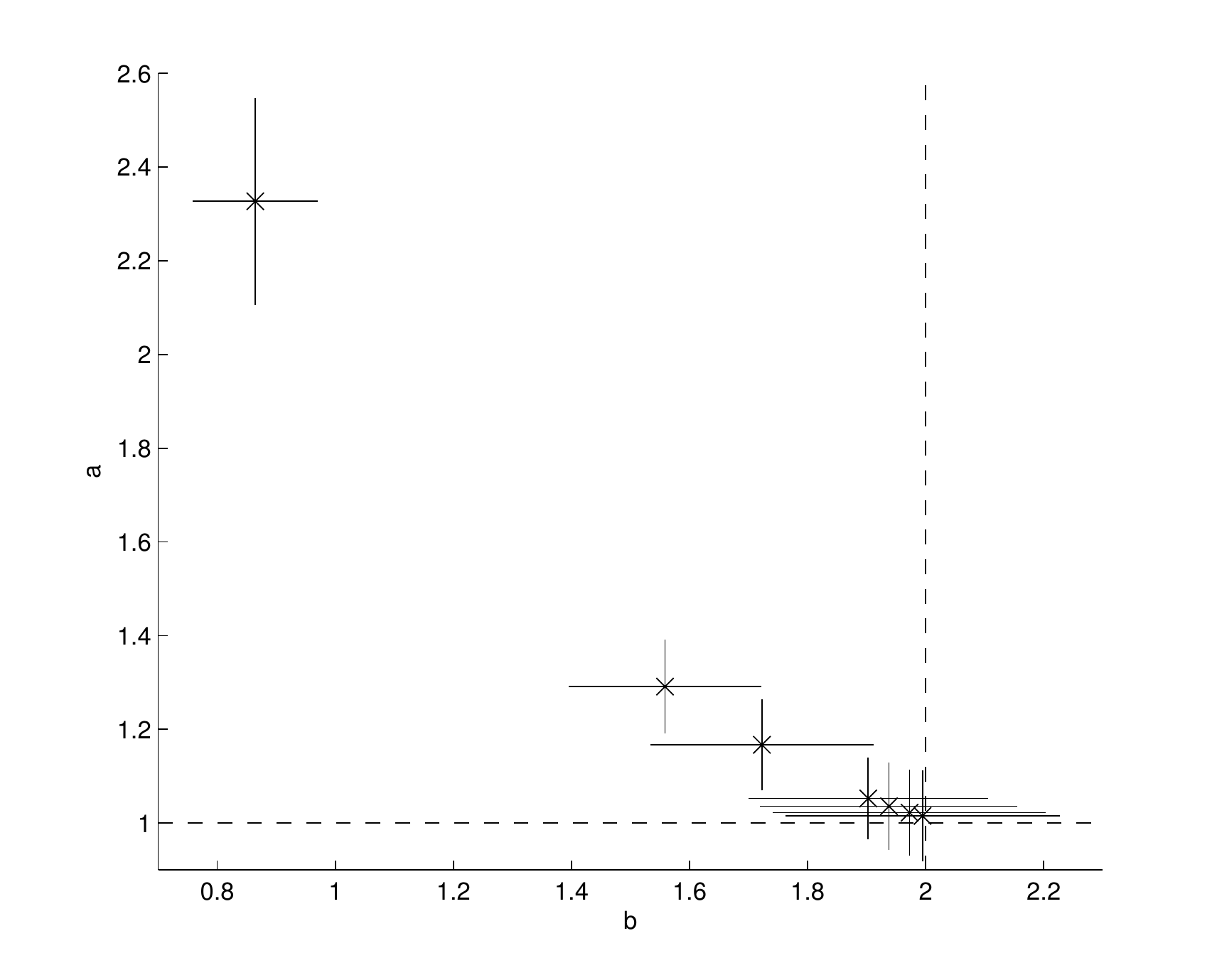}
\caption{{\small Empirical means ($\times$) and standard deviations of the estimators $\left(\hat b^{200,(h)},\hat a^{200,(h)}\right)$ based on $500$ independent observations of the MCARMA process \eqref{eq-Ysimu} on the time grid $(0,h,2h,\ldots,200)$ for $h\in\{0.5, 0.1, 0.05,0.01,0.005,0.001,0.0005\}$. The dashed lines indicate the true parameter value $(b,a)=(2,1)$.}}
\label{fig-crosseconsistencyNh}
\end{figure}

Finally, we conducted another Monte Carlo simulation with the goal of confirming the asymptotic normality of the maximum likelihood estimator \labelcref{eq-MLgamma}. \Cref{fig-CLTconfirm} compares the empirical distribution of the estimator $\left(\hat b^{200,(0.001)},\hat a^{200,(0.001)}\right)$ to the asymptotic normal distribution asserted by the Central Limit Theorem \labelcref{theorem-GMMLevy}. The points indicate the values of the estimates obtained from 500 independent realizations of the CARMA process \labelcref{eq-Ysimu}. The dashed and solid straight lines show the empirical mean $(1.9772,1.0217)$ of the estimates and the true values $(2,1)$ of the parameter $(b,a)$, respectively, which are in good agreement. The dashed and solid ellipses represent the empirical autocovariance matrix $\left(\begin{array}{cc} 4.70 & -1.45 \\ -1.45 & 0.78 \end{array}\right)\times 10^{-2}$ of the estimates and the scaled asymptotic covariance matrix $\Sigma/200\approx\left(\begin{array}{cc}5.11 & -1.55 \\ -1.55 & 0.78 
\end{array}\right)\times 10^{-2}$, respectively. Their closeness, which is also reflected by the similarity of the ellipses in \cref{fig-CLTconfirm}, means that, even for finite observation periods and sampling frequencies, $\Sigma/N$ is a good approximation of the true covariance of the estimator $\left(\hat b^{N,(h)},\hat a^{N,(h)}\right)$ and can thus be used for the construction of confidence regions. For the present example, the asymptotic covariance matrix $\Sigma$, given by \cref{eq-SigmaGMM}, can be computed explicitly as
\begin{equation*}
\Sigma^{-1} = -\E \left[ \nabla^2_{(b,a)}\log f_{b,a}\left(\Gamma_{b,a}(1)\right) \right]_{(b,a)=(2,1)} = \left.\left(\begin{array}{cc}a/b^2 & 1/b\\ 1/b & \psi_1(a)\end{array}\right)\right|_{(b,a)=(2,1)} = \left(\begin{array}{cc}1/4 & 1/2\\ 1/2 & \pi^2/6\end{array}\right),
\end{equation*}
where $\psi_1$ denotes the trigamma function, that is the second derivative of the logarithm of the gamma function. \Cref{fig-CLTconfirm} also compares histograms of $\hat b^{200,(0.001)}$ and $\hat a^{200,(0.001)}$ to the densities of the marginals of the bivariate Gaussian distribution with mean $(2,1)$ and covariance matrix $\Sigma/200$. The agreement is very good, in accordance with the Central Limit Theorem \labelcref{theorem-GMMLevy}.
\begin{figure}
\centering
\includegraphics[width=.75\textwidth]{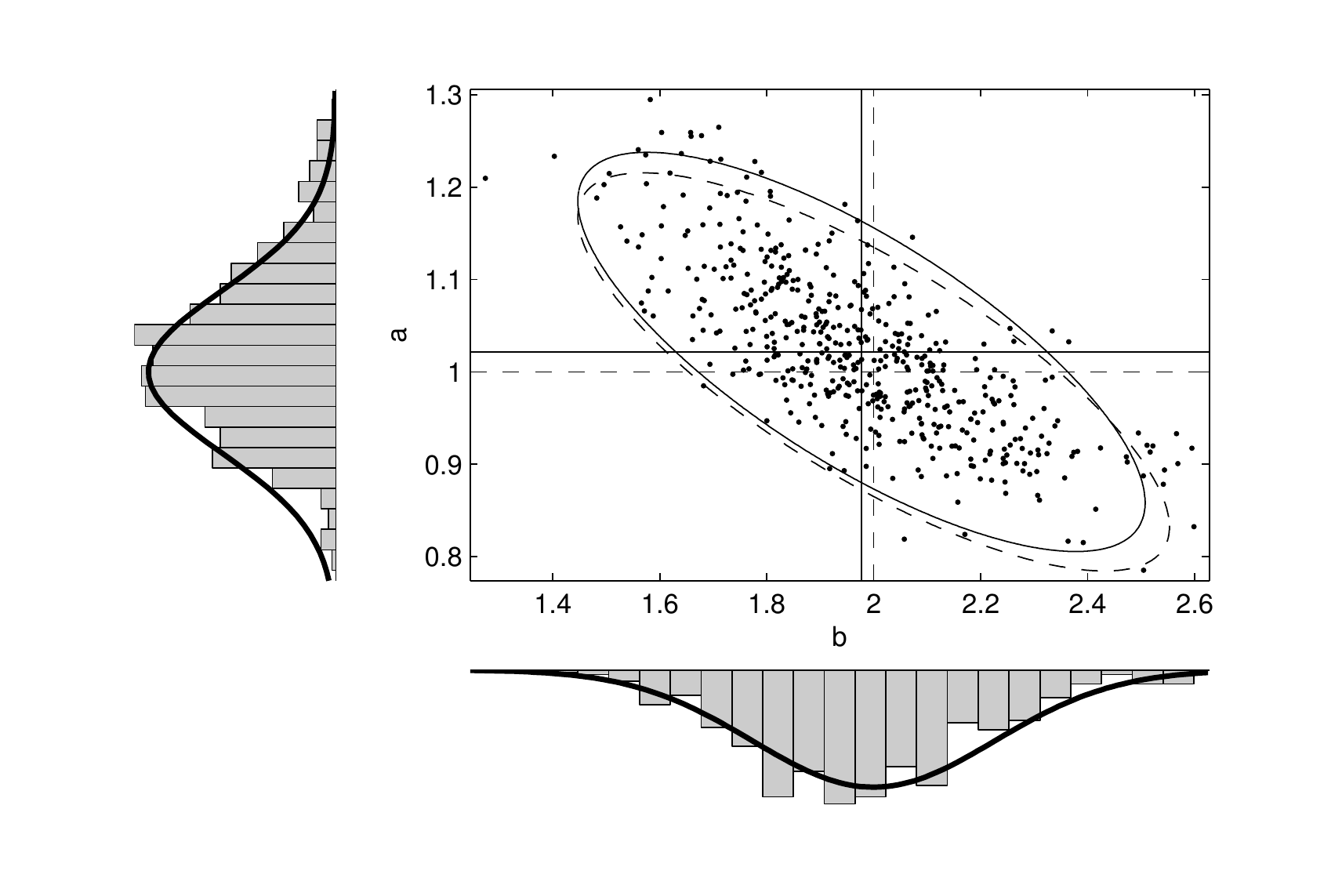}
\caption{{\small Comparison of the empirical distribution of the estimator $\left(\hat b^{200,(0.001)},\hat a^{200,(0.001)}\right)$ based on $500$ realizations of the $\Gamma_{2,1}$-driven CARMA(3,1) process given by \cref{eq-Ysimu} to the asymptotic distribution implied by the Central Limit Theorem \labelcref{theorem-GMMLevy}}}
\label{fig-CLTconfirm}
\end{figure}

\section*{Acknowledgements}
Part of this work was conceived during a visit of ES to the Statistics department of Colorado State University, whose hospitality is gratefully acknowledged. ES also thanks Robert Stelzer for helpful discussion and acknowledges financial support from the International Graduate School of Science and Engineering of the Technische Universit\"at M\"unchen and the TUM Institute for Advanced Study, funded by the German Excellence Initiative. PB gratefully acknowledges the support of National Science Foundation Grant DMS-1107031.

\appendix
\renewcommand*{\thesection}{\Alph{section}}

\section{Auxiliary results}
\label{appendix}
In this appendix we collect some auxiliary results and technical proofs to complement the derivation of the results presented in the main part of the paper.
\subsection{Auxiliary results for \cref{section-MCARMA}}
\begin{lemma}
\label{lemma-propMCARMA}
Assume that $\Lb$ is a L\'evy process and that $\Y$ is an $\Lb$-driven multivariate CARMA process given by the state space representation \labelcref{eq-statespacerepmod} and satisfying \cref{assum-eigA}. Then the following hold.
\begin{enumerate}[(i)]
 \item\label{lemma-propMCARMA-stationary} The process $\Y$ is strictly stationary.
 \item\label{lemma-propMCARMA-smooth} The paths of $\Y$ are $p-q-1$ times differentiable. Moreover, for $j=1,\ldots,p$, the paths of the $j$th $m$-block of the state process $\X$ are $p-j$ times differentiable.
 \item\label{lemma-propMCARMA-moments} For any $k>0$ and any $t,s,\in\R$, finiteness of $\E\left\|\Lb(1)\right\|^k$ implies finiteness of both $\E\left\|\X(t)\right\|^k$ and $\E\left\|\Y(s)\right\|^k$. Conversely, finiteness of the $k$th moment of $\X(t)$ implies finiteness of $\Lb(1)$.
  \item\label{lemma-propMCARMA-mixing} If $\E\left\|\Lb(1)\right\|^k$ is finite for some $k>0$, then the process $\Y$ is strongly mixing with exponentially decaying mixing coefficients.
\end{enumerate}
\end{lemma}
\begin{proof}
The first claim is an immediate consequence of the state space representation \labelcref{eq-statespacerepmod}. Parts \labelcref{lemma-propMCARMA-smooth,lemma-propMCARMA-moments} follow from \citep[Propositions 3.32 and 3.30]{marquardt2007multivariate}, respectively, if we observe that $E_p$ is injective. The assertion \labelcref{lemma-propMCARMA-mixing} follows from \citep[Theorem 4.3]{masuda2004onmultidimensionalOU}, see also the proof of \citep[Proposition 3.34]{marquardt2007multivariate}.
\end{proof}

The following lemma relates strong mixing of a continuous-time process to strong mixing of functionals of the process. It is used in the proof of \cref{prop-derivMCARMA}.
\begin{lemma}
\label{lemma-mixingfunctional}
Let $X=(X_t)_{t\in\R}$ be an $\R^d$-valued (exponentially) strongly mixing stochastic process. If, for each $n\in\Z$, the random variable $Y_n$ is measurable with respect to $\sigma(X_j:n-1\leq j\leq n)$ then the stochastic process $(Y_n)_{n\in\Z}$ is (exponentially) strongly mixing. In particular, if $f:\R^{d\times [0,1]}\to\R^m$ is a measurable function, then the $\R^m$-valued stochastic process $\left(f((X_{n-1+t})_{t\in[0,1]})\right)_{n\in\Z}$ is (exponentially) strongly mixing.
\end{lemma}
\begin{proof}
This follows immediately from \cref{eq-Defalpha}, the definition of the strong mixing coefficients.
\end{proof}

\subsection{Auxiliary results for \cref{section-DTestimation}}
The following lemma collects some useful properties of forward differences; in particular it shows that if a function $f$ is sufficiently smooth, then the derivative $\DD^{\nu}f(t)$ is well approximated by $\Delta_h^\nu[f](t)$.
\begin{lemma}
\label{lemma-propforwarddifference}
For $h>0$ and a positive integer $\nu$ let the forward-differences $\Delta_h^\nu[f](t)$ be defined by \cref{eq-forwarddifference}. The following properties hold:
\begin{enumerate}[i)]
 \item\label{lemma-propforwarddifference-recursive} For every positive integer $k<\nu$ and every function $f$, one has $\Delta_h^\nu[f]=\Delta_h^k\left[\Delta_h^{\nu-k}[f](\cdot)\right]$.
 \item\label{lemma-propforwarddifference-smooth} If the function $f:\R\to\R^m$ is $\nu+1$ times continuously differentiable on the interval $[t,t+\nu h]$ then there exist $t_i^*\in[t,t+\nu h]$, $i=1,\ldots,m$, such that
\begin{equation}
\label{eq-derivativeproxysmooth}
\Delta_h^\nu[f](t)=\DD^\nu f(t)-\frac{h}{2}\DD^{\nu+1}f(\underline t^*),
\end{equation}
where $\DD^{\nu+1}f(\underline t^*)$ is the vector whose $i$th component equals the $i$th component of $\DD^{\nu+1}f(t_i^*)$. In particular, for every polynomial $\mathfrak{p}$ of degree at most $\nu$, one has $\Delta_h^\nu[\mathfrak{p}]=\DD^\nu\mathfrak{p}$.
 \item\label{lemma-propforwarddifference-nonsmooth} If the $(\nu+1)$th derivative of $f$ is not assumed to be continuous it holds that
\begin{equation}
\label{eq-derivativeproxysmooth2}
\left\|\Delta_h^\nu[f](t)-\DD^\nu f(t)\right\|\leq h\sup_{s\in[t,t+\nu h]}\left\|\DD^{\nu+1}f(s)\right\|.
\end{equation}
\end{enumerate}
\end{lemma}
\begin{proof}
Property \labelcref{lemma-propforwarddifference-recursive} is immediate from the definition \labelcref{eq-forwarddifference}. The assertions of \labelcref{lemma-propforwarddifference-smooth,lemma-propforwarddifference-nonsmooth} follow from a component-wise application of Taylor's theorem (\citep[Theorem 5.19]{apostol1974mathematical}).
\end{proof}

In the next lemma we will show that the supremum of an Ornstein--Uhlenbeck-type process has finite absolute $k$th moments if and only if the driving L\'evy process has finite $k$th moments. This will allow us to effectively employ the error bound \labelcref{eq-derivativeproxysmooth2} for multivariate CARMA processes.
\begin{lemma}
\label{lemma-momentssup}
Let $\left(\Lb(t)\right)_{t\geq 0}$ be an $m$-variate L\'evy process and let $A\in M_N(\R)$, $B\in M_{N,m}(\R)$ be given coefficient matrices. Assume that all eigenvalues of $A$ have strictly negative real parts and that $\X=\left(\X(t)\right)_{t\geq 0}$ is the unique stationary solution of the stochastic differential equation
\begin{equation}
\label{eq-LevyOUDoobMeyer}
\dd\X(t) = A\X(t)\dd t+B\dd\Lb(t).
\end{equation}
Further denote by
\begin{equation}
\X^*(t)=\sup_{0\leq s\leq t}\left\|\X(s)\right\|
\end{equation}
the supremum of $\left\|\X\right\|$ on the compact interval $[0,t]$. It then holds that, for every $t\in\R$ and every $k>0$, the $k$th moment $\E\left(\X^*(t)\right)^k$ is finite if and only if $\E\left\|\Lb(1)\right\|^k$ is finite.
\end{lemma}
\begin{proof}
If $\E\left\|\Lb(1)\right\|^k$ is infinite it follows from \cref{lemma-propMCARMA},\labelcref{lemma-propMCARMA-moments} that $\E\left\|\X(t)\right\|^k$ is infinite as well for every $t\in\R$ and that therefore $\E\left(\X^*(t)\right)^k$ must be infinite. The other implication requires more work.

We first note that $\X^*(t)\leq\sum_{i=1}^NX_i^*(t)$ where $X_i^*(t)=\sup_{0\leq s\leq t}\left|X^i(s)\right|$ is the supremum of the $i$th component of $\X$ over the interval $[0,t]$. Since each $X^i$ is a semi-martingale, \citep[Theorem V.2]{protter1990sia} shows that there exists a universal constant $c_k$ such that $\E\left(X_i^*(t)\right)^k\leq c_k\left\|X^i\right\|_{\mathscr{H}^k_t}$, where the norm $\left\|\cdot\right\|_{\mathscr{H}^k_t}$ is defined by
\begin{equation*}
 \left\|X^i\right\|_{\mathscr{H}^k_t} = \inf_{X^i=\widetilde V_i+\widetilde M_i}{\E\left(\int_0^t{\left|\dd \widetilde V_i(s)\right|}+[\widetilde M_i,\widetilde M_i]_t^{1/2}\right)^k}.
\end{equation*}
Here, the infimum is taken over all decompositions of $X^i$ into a local martingale $\widetilde M_i$ and an adapted, c\`adl\`ag process $\widetilde V_i$ with finite variation, and $[\cdot,\cdot]$ denotes the quadratic variation process. In our situation, \cref{eq-LevyOUDoobMeyer} defines a canonical decomposition of $X^i$, $i=1,\ldots,N$, into the finite variation process $V_i=\left(V_i(t)\right)_{t\geq 0}$ given by
\begin{equation*}
V_i(t) = \be_i^T\left[\X(0)+\int_0^t A\X(s)\dd s+tB\E\Lb(1)\right],
\end{equation*}
where $\be_i$ denotes the $i$th unit vector in $\R^N$, and the martingale $M_i=\left(M_i(t)\right)_{t\geq 0}$ given by
\begin{equation*}
M_i(t)=\be_i^TB\left[\Lb(t)-t\E\Lb(1)\right].  
\end{equation*}
Since clearly,
\begin{align*}
\left(\X^*(t)\right)^k = \sup_{0\leq s\leq t}\left(X^1(s)^2+\ldots+X^N(s)^2\right)^{k/2}\leq& \left(X_1^*(t)^2+\ldots+X_N^*(t)^2\right)^{k/2}\\
  \leq & N^{k/2}\max_{1\leq i\leq N}{X_i^*(t)^k}\leq N^{k/2}\sum_{i=1}^N{X_i^*(t)^k},
\end{align*}
it suffices to bound the $k$th moments of $X_i^*(t)$ in order to obtain a bound for the $k$th moment of $\X^*(t)$. The former can be estimated as
\begin{equation}
\label{eq-EXikdecomp}
\E\left(X_i^*(t)\right)^k\leq c_k\left\|X^i\right\|_{\mathscr{H}^k_t} \leq c_k\left[\E\left(\int_0^t{\left|\dd V_i(s)\right|}\right)^k+\E[M_i,M_i]_t^{k/2}\right].
\end{equation}
The first term in this expression is seen to satisfy
\begin{align*}
\E\left(\int_0^t{\left|\dd V_i(s)\right|}\right)^k &\leq \E\left(\int_0^t{\left|\be_i^TA\X(s)\right|\dd s}+t\left|\be_i^TB\Lb(1)\right|\right)^k\\
  &\leq 2^k\left[\left\|A\right\|^k\int_0^t\E\left\|\X(s)\right\|^k\dd s+t^k\left\|B\right\|^k\E\left\|\Lb(1)\right\|^k\right]<\infty,
\end{align*}
where the finiteness of the integral $\int_0^t\E\left\|\X(s)\right\|^k\dd s$ follows from the assumption that $\E\left\|\X(s)\right\|^k$ is finite and the strict stationarity of $\X$. For the second term in \cref{eq-EXikdecomp} one obtains the bound
\begin{align*}
\E[M_i,M_i]_t^{k/2} & = \E\left(\be_i^TB[\Lb,\Lb]_tB^T\be_i\right)^{k/2}\\
  &\leq \left\|B\right\|^k\E\left\|[\Lb,\Lb]_t\right\|^{k/2}  \leq 2^k\left\|B\right\|^k\left\{\left\|\Sigma^{\mathcal{G}}\right\|^{k/2}t^{k/2} + \E\left\|\int_0^t\int_{\R^m} \bx\bx^T N(\dd s,\dd\bx)\right\|^{k/2}\right\},
\end{align*}
where we have used \citep[Theorem I.4.52]{jacod2003limit} to compute the quadratic variation of the L\'evy process $\Lb$ with characteristic triplet $(\bgammaL,\Sigma^{\mathcal{G}},\nuL)$. To see that this expression is finite we observe that
\begin{align*}
\E\left\|\int_0^t\int_{\R^m} \bx\bx^T N(\dd s,\dd\bx)\right\|^{k/2} \leq& m^{k/2}\E\left(\int_0^t\int_{\R^m} \left\|\bx\right\|^2 N(\dd s,\dd\bx)\right)^{k/2}\\
  =& m^{k/2}\E\left(\lim_{\epsilon\to 0}\int_0^t\int_{\left\|\bx\right\|\geq\epsilon} \left\|\bx\right\|^2 N(\dd s,\dd\bx)\right)^{k/2}\\
  =& m^{k/2}\lim_{\epsilon\to 0}\E\left(\int_0^t\int_{\left\|\bx\right\|\geq\epsilon} \left\|\bx\right\|^2 N(\dd s,\dd\bx)\right)^{k/2}\eqqcolon m^{k/2} \lim_{\epsilon\to 0}\E Y_\epsilon^{k/2},
 \end{align*}
where we have applied the Monotone Convergence Theorem (\cite[Theorem 4.20]{klenke2008probability}) to interchange the order of expectation and passing to the limit. By \citep[Proposition 19.5]{sato1991lpa}, for each $\epsilon>0$, the random variable $Y_\epsilon=\int_0^t\int_{\left\|\bx\right\|\geq\epsilon} \left\|\bx\right\|^2 N(\dd s,\dd\bx)$ is infinitely divisible with characteristic measure $\rho_\epsilon = (\lambda|_{[0,t]}\otimes\nuL|_{\{\left\|\bx\right\|\geq\epsilon\}})\phi_\epsilon^{-1}$, where $\phi_\epsilon:[0,t]\times\{\left\|\bx\right\|\geq\epsilon\}\to\R^+$ maps $(s,\bx)$ to $\left\|\bx\right\|^2$, and with characteristic drift $\bgamma_\epsilon = \int_\R y\rho_\epsilon(\dd y)$. From this it follows that
\begin{equation*}
\int_0^\infty{y^{k/2}\rho_\epsilon(\dd y)} = t\int_{\left\|\bx\right\|\geq\epsilon}\left\|\bx\right\|^k\nuL(\dd\bx)\leq t\int_{\left\|\bx\right\|<1}\left\|\bx\right\|^2\nuL(\dd\bx)+t\int_{\left\|\bx\right\|\geq1}\left\|\bx\right\|^k\nuL(\dd\bx)<\infty,\quad \forall\epsilon>0,
\end{equation*}
and
\begin{equation*}
 \bgamma_\epsilon=\int_0^\infty y\rho_\epsilon(\dd y) = t\int_{\left\|\bx\right\|\geq\epsilon}\left\|\bx\right\|^2\nuL(\dd\bx)\leq t\int_{\left\|\bx\right\|<1}\left\|\bx\right\|^2\nuL(\dd\bx)+t\int_{\left\|\bx\right\|\geq1}\left\|\bx\right\|^2\nuL(\dd\bx)<\infty,\quad \forall\epsilon>0.
\end{equation*}
\Cref{lemma-momentsidrv} then implies that $\lim_{\epsilon\to 0}\E Y_\epsilon^{k/2}$ is finite, which completes the proof.
\end{proof}

For the upcoming proof of proof of \cref{lemma-integrallevyapprox} we first show the following locality property of the approximation errors $\be_{I_f^\nu,n}^{\nu,(h)}$.
\begin{lemma}
\label{lemma-diffintiid}
For every positive integer $\nu\geq 2$ and every function $f$, the approximation error $\be_{I_f^\nu,n}^{\nu,(h)}$ is a function only of the increments $\{f(t)-f(n):n\leq t\leq n+\nu h\}$. This function is independent of $n$. In particular, $\be_{I_{\Lb}^\nu}^{\nu,(h)}$ is an i.i.d$.$ sequence.
\end{lemma}
\begin{proof}
The claim can be shown by direct calculations: \cref{lemma-propforwarddifference},\labelcref{lemma-propforwarddifference-recursive} implies that
\begin{align*}
\Delta_h^\nu\left[I^\nu_f\right](n)=&\Delta_h^1\Delta_h^{\nu-1}\left[I^\nu_f\right](n)\\
=&\Delta_h^1\left[\frac{1}{h^{\nu-1}}\sum_{i=0}^{\nu-1}{(-1)^{\nu-1-i}\binom{\nu-1}{i}I^\nu_f(\cdot+ih)}\right](n)\\
=& \frac{1}{h^\nu}\sum_{i=0}^{\nu-1}{(-1)^{\nu-1-i}\binom{\nu-1}{i}\int_{n+ih}^{n+(i+1)h}I^{\nu-1}_f(s)\dd s}\\
  =&\frac{1}{h^\nu}\sum_{i=0}^{\nu-1}{(-1)^{\nu-1-i}\binom{\nu-1}{i}\int_{n+ih}^{n+(i+1)h}\left[\int_{0\leq t_{\nu-1}\leq\cdots\leq t_1\leq s}f(t_{\nu-1})\dd t_{\nu-1}\cdots \dd t_1\right]\dd s}.
\end{align*}
Using that the set $\{t_{\nu-1}\leq t_{\nu-2}\leq\cdots\leq t_1\leq s\}$ is congruent to the $(\nu-2)$-dimensional simplex in the hypercube with side lengths $s-t_{\nu-1}$ and that thus
\begin{equation*}
\int_{t_{\nu-1}\leq t_{\nu-2}\leq\cdots\leq t_1\leq s}\dd t_{\nu-2}\cdots \dd t_1=\frac{1}{(\nu-2)!}(s-t_{\nu-1})^{\nu-2},
\end{equation*}
we obtain that
\begin{align*}
\Delta_h^\nu\left[I^\nu_f\right](n)=&\frac{1}{h^\nu(\nu-2)!}\sum_{i=0}^{\nu-1}{(-1)^{\nu-1-i}\binom{\nu-1}{i}\int_{n+ih}^{n+(i+1)h}\int_{0}^s(s-t_{\nu-1})^{\nu-2}f(t_{\nu-1})\dd t_{\nu-1}\dd s}\\
=&\frac{1}{h^\nu(\nu-2)!}\int_{0}^n\left[\sum_{i=0}^{\nu-1}{(-1)^{\nu-1-i}\binom{\nu-1}{i}\int_{n+ih}^{n+(i+1)h}(s-t_{\nu-1})^{\nu-2}}\dd s\right]f(t_{\nu-1})\dd t_{\nu-1}\\
&+\frac{1}{h^\nu(\nu-2)!}\sum_{i=0}^{\nu-1}{(-1)^{\nu-1-i}\binom{\nu-1}{i}\int_{n+ih}^{n+(i+1)h}\int_{n}^s(s-t_{\nu-1})^{\nu-2}f(t_{\nu-1})\dd t_{\nu-1}\dd s}.
\end{align*}
It is easy to see that $\int_{n+ih}^{n+(i+1)h}(s-t_{\nu-1})^{\nu-2}\dd s$ is equal to $\mathfrak{p}_{\nu,h}(n-t_{\nu-1}+ih)$ for some polynomial $\mathfrak{p}_{\nu,h}$ of degree $\nu-2$.  It then follows from \cref{lemma-propforwarddifference},\labelcref{lemma-propforwarddifference-smooth} that
\begin{equation*}
\sum_{i=0}^{\nu-1}{(-1)^{\nu-1-i}\binom{\nu-1}{i}\int_{n+ih}^{n+(i+1)h}(s-t_{\nu-1})^{\nu-2}}\dd s =\Delta_h^{\nu-1}\left[\mathfrak{p}_{\nu,h}\right](n-t_{\nu-1}) = 0 ,\quad \forall t_{\nu-1}\in[0, n],
\end{equation*}
which implies that the first term in the last expression for $\Delta_h^\nu\left[I^\nu_f\right](n)$ vanishes. It is similarly easy to see that
\begin{equation*}
\frac{1}{h^\nu(\nu-2)!}\sum_{i=0}^{\nu-1}{(-1)^{\nu-1-i}\binom{\nu-1}{i}\int_{n+ih}^{n+(i+1)h}\int_{n}^s(s-t_{\nu-1})^{\nu-2}\dd t_{\nu-1}\dd s}=1.
\end{equation*}
Consequently,
\begin{align}
\label{eq-fnnuh}
\be_{I_f^\nu,n}^{\nu,(h)} =& \Delta_h^\nu\left[I^\nu_f\right](n)-f(n)\notag\\
  =& \frac{1}{h^\nu(\nu-2)!}\sum_{i=0}^{\nu-1}{(-1)^{\nu-1-i}\binom{\nu-1}{i}\int_{ih}^{(i+1)h}\int_0^{s}(s-t_{\nu-1})^{\nu-2}\left[f(n+t_{\nu-1})-f(n)\right]\dd t_{\nu-1}\dd s},
\end{align}
which completes the proof of the first part of the lemma. The fact that L\'evy processes have stationary and independent increments together with the last display implies that the sequence $\be_{I_{\Lb}^\nu}^{\nu,(h)}$ is i.i.d$.$
\end{proof}

\begin{proof}[Proof of \cref{lemma-integrallevyapprox}]
Let $\epsilon>0$ be given. By the right-continuity of the process $\Lb$ there exists $\delta_{\epsilon,n}$ such that $\left\|\Lb(n+t)-\Lb(n)\right\|\leq\epsilon$ for all $t\in[0,\delta_{\epsilon,n}]$. Hence, assuming $\nu h\leq\delta_{\epsilon,n}$, \cref{eq-fnnuh} implies that
\begin{align*}
\left\|\be_{I_{\Lb}^\nu,n}^{\nu,(h)}\right\|\leq&\frac{\epsilon}{h^\nu(\nu-2)!}\sum_{i=0}^{\nu-1}{\binom{\nu-1}{i}\int_{ih}^{(i+1)h}\int_0^{s}(s-t_{\nu-1})^{\nu-2}\dd t_{\nu-1}\dd s}\\
  =&\frac{\epsilon}{\nu!}\sum_{i=0}^{\nu-1}{\binom{\nu-1}{i}\left[(i+1)^\nu-i^\nu\right]}.
\end{align*}
This proves that $\left\|\be_{I_{\Lb}^\nu,n}^{\nu,(h)}\right\|\to 0$ as $h\to 0$. We now turn to the absolute moments of $\be_{I_{\Lb}^\nu,n}^{\nu,(h)}$. Again it entails no loss of generality to assume that $n=0$. \Cref{eq-fnnuh} and the triangle inequality lead to
\begin{align*}
\E\left\|\be_{I_{\Lb}^\nu,0}^{\nu,(h)}\right\|^k=&\E\left\|\frac{1}{h^\nu(\nu-2)!}\sum_{i=0}^{\nu-1}{(-1)^{\nu-1-i}\binom{\nu-1}{i}\int_{ih}^{(i+1)h}\int_0^{s}(s-t)^{\nu-2}\Lb(t)\dd t\dd s}\right\|^k\\
\leq&\left[\frac{1}{h^\nu(\nu-2)!}\right]^k\E\left(\sum_{i=0}^{\nu-1}{\binom{\nu-1}{i}\int_{ih}^{(i+1)h}\int_0^{s}(s-t)^{\nu-2}\left\|\Lb(t)\right\|\dd t\dd s}\right)^k.
\end{align*}
An application of H\"older's inequality with the dual exponent $k'$ determined by $1/k+1/k'=1$ shows that the last line of the previous display is dominated by
\begin{align*}
\leq& \left[\frac{1}{h^\nu(\nu-2)!}\right]^k \E\left(\sum_{i=0}^{\nu-1}{\binom{\nu-1}{i}^{k'}\int_{ih}^{(i+1)h}\int_0^{s}(s-t)^{k'(\nu-2)}\dd t\dd s}\right)^{k/{k'}}\left(\int_0^{\nu h}\int_0^{s}\left\|\Lb(t)\right\|^k\dd t\dd s\right)\\
=& \frac{C}{h^2}\E\left(\int_0^{\nu h}(\nu h-t)\left\|\Lb(t)\right\|^k\dd t\right),
\end{align*}
where the constant $C$ depends only on $\nu$ and $k$ and is given by
\begin{equation*}
 C=\frac{1}{(\nu-2)!} \left[ \frac{1}{(\nu-2)![k'(\nu-2)+2][k'(\nu-2)+1]}\sum_{i=0}^{\nu-1}{\binom{\nu-1}{i}^{k'}  \left[(i+1)^{k'(\nu-2)+2}-i^{k'(\nu-2)+2}\right]  }\right]^{k-1}.
\end{equation*}
\Cref{prop-levymoments} asserts the existence of a constant $C'$ such that $\E\left\|\Lb(t)\right\|^k<C't^{k/(k)_0}$ for all $t\leq \nu h$. Consequently
\begin{align*}
\E\left\|\be_{I_{\Lb}^\nu,0}^{\nu,(h)}\right\|^k\leq \frac{CC'}{h^2}\int_0^{\nu h}(\nu h-t)t^{k/(k)_0}\dd t=\frac{CC'\nu^{k/(k)_0+2}}{\left[k/(k)_0+1\right]\left[k/(k)_0+2\right]}h^{k/(k)_0},
\end{align*}
showing that $\E\left\|\be_{I_{\Lb}^\nu,0}^{\nu,(h)}\right\|^k=O(h^{k/(k)_0)})$ and thereby completing the proof of the lemma.
\end{proof}

We now turn our attention t the approximation of integrals. The following result provides a quantitative bound for the accuracy with which the trapezoidal rule approximates a definite integral if the integrand is a smooth function.
\begin{proposition}
\label{prop-trapez}
Let $[a,b]\subset\R$ be an interval and let $K$ be a positive integer.
\begin{enumerate}[i)]
\item\label{prop-trapez-scalar}Assume that $f:[a,b]\to\R$ is a twice differentiable function. Then
\begin{equation}
\left|\int_a^b{f(s)\dd s}-T_{[a,b]}^Kf\right|\leq\frac{(b-a)^3}{12K^2}\sup_{t\in[a,b]}\left|f''(t)\right|.
\end{equation}
\item\label{prop-trapez-vector} Assume that $F:[a,b]\to \R^d$ is a twice differentiable function. Then
\begin{equation}
\label{eq-vectortrapez}
\left\|\int_a^b{F(s)\dd s}-T_{[a,b]}^KF\right\|\leq\frac{(b-a)^3\sqrt{d}}{12K^2}\sup_{t\in[a,b]}\left\|F''(t)\right\|.
\end{equation}
\item\label{prop-trapez-matrix} Assume that $F:[a,b]\to M_d(\R)$ is a twice differentiable function. Then
\begin{equation}
\label{eq-matrixtrapez}
\left\|\int_a^b{F(s)\dd s}-T_{[a,b]}^KF\right\|\leq\frac{(b-a)^3d^{3/2}}{12K^2}\sup_{t\in[a,b]}\left\|F''(t)\right\|.
\end{equation}
\end{enumerate}
\end{proposition}
\begin{proof}
Part \labelcref{prop-trapez-scalar} is \citep[Lemma 9.8]{Deuflhard2008}. To see that \labelcref{prop-trapez-vector} holds it is enough to apply \labelcref{prop-trapez-scalar} componentwise to obtain that 
\begin{align*}
\left\|\int_a^b{F(s)\dd s}-T_{[a,b]}^KF\right\|\leq&\sqrt{d}\max_{1\leq i\leq d}\left|\left[\int_a^b{F(s)\dd s}-T_{[a,b]}^KF\right]_i\right|\\
  \leq& \frac{(b-a)^3\sqrt{d}}{12K^2}\max_{1\leq i\leq d}\sup_{t_i\in[a,b]}\left|F_i''(t_i)\right|\\
  =& \frac{(b-a)^3\sqrt{d}}{12K^2}\sup_{t\in[a,b]}\max_{1\leq i\leq d}\left|F_i''(t)\right|\leq\frac{(b-a)^3\sqrt{d}}{12K^2}\sup_{t\in[a,b]}\left\|F''(t)\right\|.
\end{align*}
The claim \labelcref{eq-matrixtrapez} about matrix-valued integrands follows from the fact that $M_d(\R)$ is canonically isomorphic to $\R^{d^2}$ and that the operator norm and the Euclidean vector norm induced by this isomorphism satisfy (\citep{stone1962})
\begin{equation*}
\frac{1}{\sqrt{d}}\left\|M\right\|_{R^{d^2}}\leq \left\|M\right\|\leq \left\|M\right\|_{\R^{d^2}},\quad \text{ for all } M\in M_d(\R).{{{}}}\qedhere
\end{equation*}
\end{proof}

\subsection{Auxiliary results for \cref{sec-GMM}}
The following lemmas are used in the proof of the Central Limit Theorem \labelcref{theorem-GMMnew}.
\begin{lemma}
\label{lemma-basicconvp}
For sequences $(Y_n)_{n\geq 1}$, $(Z_n)_{n\geq 1}$ of vector- or matrix-valued random variables the following hold:
\begin{enumerate}[i)]
 \item For every constant $c$, $Y_n\convp c$ if and only if $Y_n\convd c$.
 \item If $Y_n\convd Y_{\infty}$ and $Z_n-Y_n\convp 0$ then $Z_n\convd Y_{\infty}$.
 \item Denote by $\operatorname{supp}Y_n$ the support of $Y_n$. If $Y_n\convd Y_\infty$ and the function $f$ is defined on $\bigcap_{n\geq 1}\operatorname{supp}Y_n$ and continuous on an open set containing $\operatorname{supp}Y_\infty$ then $f(Y_n)\convd f(Y_\infty)$.
\end{enumerate}
\end{lemma}
\begin{proof}
Parts i) and ii) are proved in \citep[Theorem 2.7]{vanderwaart1998}. Assertion iii) is \citep[Theorem 13.25]{klenke2008probability}).
\end{proof}
The next result we will need is a uniform version of the weak law of large numbers, given by \citep[Lemma 2.4]{newey1994largesample}.
\begin{lemma}
\label{lemma-uniformWLLN}
Assume that for every $\vartheta\in\Theta$, $\Theta$ a compact subset of $\R^r$, there is a sequence $\left(Y_n(\vartheta)\right)_{n\geq 1}$ of independent identically distributed random variables with finite expectation $\psi(\vartheta)=\E Y_1(\vartheta)<\infty$. Further assume that for each $\vartheta'\in\Theta$ the random function $\vartheta\mapsto Y_1(\vartheta)$ is almost surely continuous at $\vartheta'$ and that there exists a random variable $Z$ satisfying $\E Z<\infty$ such that $\sup_{\vartheta\in\Theta}\left\|Y_1(\vartheta)\right\|\leq Z$.  It then holds that $\vartheta\mapsto\psi(\vartheta)$ is a continuous function and the time averages $\overline Y_N(\vartheta)=\sum_{n=1}^N{Y_n(\vartheta)}$ converge uniformly in probability to $\psi(\vartheta)$, that is $\sup_{\vartheta\in\Theta}\left\|\overline Y_N(\vartheta)-\psi(\vartheta)\right\|\convp 0$.
\end{lemma}

\begin{lemma}
\label{lemma-continuousmapping}
For each $\vartheta\in\Theta$, let $\left(Y_n(\vartheta)\right)_{n\geq 1}$ be a sequence of random variables. If $Y_n(\vartheta)\convp Y_{\infty}(\vartheta)$ uniformly in $\vartheta$, the sequence $(\vartheta_n)_{n\geq 1}$ of random elements of $\Theta$ converges in probability to some $\vartheta_\infty$ and the mapping $\vartheta\mapsto Y_\infty(\vartheta)$ is almost surely continuous at $\vartheta_\infty$, then $Y_n(\vartheta_n)\convp Y_\infty(\vartheta_\infty)$.
\end{lemma}
\begin{proof}
For any $\epsilon>0$ it holds that
\begin{align*}
\Pb\left(\left\|Y_n(\vartheta_n)-Y_\infty(\vartheta_\infty)\right\|\leq \epsilon\right) \geq & \Pb\left(\left\|Y_n(\vartheta_n)-Y_\infty(\vartheta_n)\right\|\leq\frac{\epsilon}{2}\text{ and } \left\|Y_\infty(\vartheta_n)-Y_\infty(\vartheta_\infty)\right\|\leq\frac{\epsilon}{2}\right)\\
  \geq & \Pb\left(\left\|Y_n(\vartheta_n)-Y_\infty(\vartheta_n)\right\|\leq\frac{\epsilon}{2}\right) + \Pb\left(\left\|Y_\infty(\vartheta_n)-Y_\infty(\vartheta_\infty)\right\|\leq\frac{\epsilon}{2}\right) - 1\to 1.
\end{align*}
The first probability in the last line converges to one as $n\to\infty$ by the assumption of uniform convergence of $Y_n$ to $Y_\infty$, the second because $Y_\infty$ is almost surely continuous at $\vartheta_\infty$ and $\vartheta_n\convp\vartheta_\infty$.
\end{proof}

\FloatBarrier

\printbibliography

\end{document}